\documentclass[11pt]{amsart}
\usepackage{amssymb}
\usepackage{amscd}
\usepackage{amsmath}
\usepackage[all]{xy}
\usepackage{mathrsfs}
\usepackage{color}
\usepackage{bbm}
\usepackage{dsfont}
\usepackage[usenames,dvipsnames,svgnames,table]{xcolor}
\usepackage[utf8]{inputenc}
\usepackage[T1]{fontenc}
\usepackage{ tipa }
\usepackage{helvet}
\usepackage{color}
\usepackage{stmaryrd}  
\usepackage{amsmath,amsxtra,amsthm,amssymb,xr,enumerate,fullpage,hyperref, comment}

 \definecolor{pAlgae}{RGB}{87,115,135}
\definecolor{airforceblue}{rgb}{0.36, 0.54, 0.66}
	\definecolor{bondiblue}{rgb}{0.0, 0.58, 0.71}
\definecolor{britishracinggreen}{rgb}{0.0, 0.26, 0.15}
\definecolor{camouflagegreen}{rgb}{0.47, 0.53, 0.42}
\definecolor{darkcyan}{rgb}{0.0, 0.55, 0.55}
	\definecolor{ao}{rgb}{0.0, 0.0, 1.0}
	\definecolor{azure}{rgb}{0.0, 0.5, 1.0}
	\definecolor{babyblueeyes}{rgb}{0.63, 0.79, 0.95}
	\definecolor{ballblue}{rgb}{0.13, 0.67, 0.8}
	\definecolor{bleudefrance}{rgb}{0.19, 0.55, 0.91}
	
\hypersetup{
    colorlinks = true,
    linkcolor=azure,
     filecolor=blue,
     citecolor = darkcyan,      
     urlcolor=cyan,
    linkbordercolor = {white},
}

\usepackage[all]{xy}
\usepackage[bbgreekl]{mathbbol}

\makeatletter
\newcommand{\mylabel}[2]{#2\def\@currentlabel{#2}\label{#1}}
\makeatother

\theoremstyle{plain}
\newtheorem{theorem}{Theorem}[section]
\newtheorem{proposition}[theorem]{Proposition}
\newtheorem{lemma}[theorem]{Lemma}
\newtheorem{corollary}[theorem]{Corollary}
\newtheorem{conjecture}[theorem]{Conjecture}

\theoremstyle{remark}
\newtheorem{remark}[theorem]{Remark}

\theoremstyle{definition}
\newtheorem{definition}[theorem]{Definition}

\newtheorem*{acknowledgements}{Acknowledgements}

  1

\DeclareMathOperator{\Ext}{Ext}
\DeclareMathOperator{\Gal}{Gal}

\DeclareMathOperator{\Hom}{Hom}
\DeclareMathOperator{\End}{End}

\DeclareMathOperator{\im}{im}

\newcommand{\bC}{\mathbb{C}}
\newcommand{\bF}{\mathbb{F}}
\newcommand{\bG}{\mathbb{G}}
\newcommand{\bQ}{\mathbb{Q}}
\newcommand{\bL}{\mathbb{L}}

\newcommand{\bT}{\mathbb{T}}
\newcommand{\bZ}{\mathbb{Z}}
\newcommand{\QQ}{\mathbb{Q}}

\newcommand{\cA}{\mathcal{A}}

\newcommand{\cG}{\mathcal{G}}

\newcommand{\cL}{\mathcal{L}}

\newcommand{\cO}{\mathcal{O}}

\newcommand{\cU}{\mathcal{U}}

\newcommand{\ff}{\mathfrak{f}}

\newcommand{\fp}{\mathfrak{p}}

\newcommand{\fm}{\mathfrak{m}}

\newcommand{\ZZ}{\mathbb{Z}}
\newcommand{\Z}{\mathbb{Z}}

\newcommand{\rec}{\mathrm{rec}}

\definecolor{pinegreen}{rgb}{0.0, 0.47, 0.44}

\begin{document}

\title[]{O\lowercase{n the non-critical exceptional zeros of} \\ K\lowercase{atz $p$-adic} $L$-\lowercase{functions for} CM \lowercase{fields}}

\author[K. B\"uy\"ukboduk]{K\^az\i m B\"uy\"ukboduk}
\address[B\"uy\"ukboduk]{UCD School of Mathematics and Statistics\\ University College Dublin\\ Ireland}
\author[R. Sakamoto]{Ryotaro Sakamoto}
\address[Sakamoto]{Department of Mathematics\\University of Tsukuba\\1-1-1 Tennodai\\Tsukuba\\Ibaraki 305-8571\\Japan}

\email{kazim.buyukboduk@ucd.ie}
\email{rsakamoto@math.tsukuba.ac.jp}

\begin{abstract}
The primary goal of this article is to study $p$-adic Beilinson conjectures in the presence of exceptional zeros for Artin motives over CM fields. In more precise terms, we address a question raised by Hida and Tilouine on the order of vanishing of Katz $p$-adic $L$-functions associated to CM fields, by means of a leading term formula we prove in terms of Rubin--Stark elements. In the particular case when the CM field in question is imaginary quadratic, our leading term formula and its consequences we record herein are unconditional.
\end{abstract}

\maketitle

\section{Introduction}

Fix forever a prime $p{>2}$. Let us fix an algebraic closure $\overline{\QQ}$ of $\QQ$ and fix embeddings $\iota_\infty: \overline{\QQ} \hookrightarrow \mathbb{C}$ and $\iota_p: \overline{\QQ} \hookrightarrow \mathbb{C}_p$ as well as an isomorphism $j:\mathbb{C}\stackrel{\sim}{\longrightarrow} \mathbb{C}_p$ in a way that the diagram
$$\xymatrix@R=.1cm{&\mathbb{C}\ar[dd]^{j}\\
\overline{\QQ}\ar[ur]^{\iota_\infty}\ar[rd]_{\iota_p} &\\
&\mathbb{C}_p
}$$
commutes. 

Perrin-Riou's $p$-adic variant of Beilinson's conjectures expresses the special values of $p$-adic $L$-functions in terms of motivic cohomology, so long as one avoids exceptional zeros. The key feature of these conjectures is that they involve  the special values of $p$-adic $L$-functions outside their interpolation range, where they have no direct link to complex $L$-values (which Beilinson's conjectures explain in terms of motivic cohomology). The main contribution in the current article is a formulation and a proof of an extension\footnote{The extension of the conjectures of Perrin-Riou we formulate and prove here not only allow treatment of exceptional zeros, but also to study the leading terms of multivariate $p$-adic $L$-functions.} of these conjectures for certain Artin motives (where the corresponding $p$-adic $L$-function is that constructed by Katz), in a manner to allow a treatment exceptional zeros. 

To be able to give a slightly more precise account of Perrin-Riou's conjectures, we let $\mathscr{M}$ be a pure motive over $\QQ$ which is ordinary and crystalline at $p$, with motivic weight $w$. Let $L(\iota_\infty\mathscr{M},s)$ denote the $L$-series associated to $\iota_\infty\mathscr{M}$ and let $\mathscr{L}_p(\iota_p\mathscr{M}) \in \mathbb{C}_p[[\Gamma]]$ (where $\Gamma\cong \ZZ_p^\times$ is the Galois group of the cyclotomic $\ZZ_p^\times$-extension of $\QQ$) denote the conjectural $p$-adic $L$-function whose existence is predicted by Coates and Perrin-Riou~\cite{CoatesPerrinRiou}. This $p$-adic $L$-function is characterised by a conjectural interpolation property, which in very rough form requires that 
$$\chi_{\rm cyc}^{n-1}\eta^{-1}\left(\mathscr{L}_p(\iota_p\mathscr{M})\right)= \mathscr{E}_p(\mathscr{M},\eta,n)\cdot j\left(\frac{L(\iota_\infty\mathscr{M}\otimes\eta,n)}{\Omega(\mathscr{M},\eta,n)}\right)$$
for all integers $n$ which are critical for $L(\iota_\infty\mathscr{M},s)$ in the sense of Deligne, where $\chi_{\rm cyc}:\Gamma\stackrel{\sim}{\rightarrow}\ZZ_p^\times$ is the cyclotomic character and $\eta: \Gamma\rightarrow \mathbb{C}_p^\times$ is a character of finite order. Here, $\mathscr{E}_p(\mathscr{M},\eta,n)$ is the Euler-like interpolation factor that corresponds to $\Phi_p(\mathscr{M}_p(\eta), n)$ in~\cite[(4.11)]{CoatesPerrinRiou}. Perrin-Riou~\cite{PR94} conjecturally describes the values $\chi_{\rm cyc}^{n-1}\eta^{-1}\left(\mathscr{L}_p(\iota_p\mathscr{M})\right)$ in terms of the motivic cohomology, \emph{for all integers $n$} that lie to the left of $\frac{w+1}{2}$, so long as we have
$$\mathscr{E}_p(\mathscr{M},\eta,n)\neq 0.$$ 
This is akin to Beilinson's conjectural formula for the  leading term of $L(\iota_\infty\mathscr{M}\otimes\eta,s)$ at non-critical integers $s=n$, in terms of his regulators on motivic cohomology groups. We say that the $p$-adic $L$-function $\mathscr{L}_p(\iota_p\mathscr{M})$ has an exceptional zero at $\chi_{\rm cyc}^{n-1}\eta\in {\rm Spec}\,\ZZ_p[[\Gamma]](\mathbb{C}_p)$ whenever $\mathscr{E}_p(\mathscr{M},\eta,n)=0$. In case $\mathscr{E}_p(\mathscr{M},\eta,n)=0$ and $n$ lies left of $\frac{w+1}{2}$ but it is non-critical, we say that $\mathscr{L}_p(\iota_p\mathscr{M})$ has a \emph{non-critical exceptional zero} at $\chi_{\rm cyc}^{n-1}\eta$. Exceptional zeros of $p$-adic $L$-functions in the cyclotomic variable alone have been extensively studied (c.f. \cite{Benois1, Benois2, Benois3} for the current state of art, including the treatment of the non-critical exceptional zeros), but to best of our knowledge, non-critical exceptional zeros along $\Z_p$-extensions other than the cyclotomic tower\footnote{Note in particular that the work of Benois is built on the theory of $(\varphi,\Gamma)$-modules. For many of the $\ZZ_p$-towers (with Galois group $\Gamma$) we consider in this work, one does not expect to have a reasonable theory of $(\varphi,\Gamma)$-modules; c.f. \cite{berger2014JEP}.} have never been explored. The main objective of the current article is to initiate this study for Katz' $p$-adic $L$-functions associated to CM fields. 

\subsection{Non-critical exceptional zeros of Katz $p$-adic $L$-functions of imaginary quadratic fields} \label{subsec_intro_imag_quad_case}
We shall first explain our results in the particular case when the CM field $k$ in question is imaginary quadratic. We defer the exposition of our results for general CM fields to Section~\ref{sec_results_general_new}.

We assume that $p=\fp\fp^c$ splits in the quadratic extension $k/\QQ$ and let $S_p(k)=\{\fp,\fp^c\}$ denote the set of primes of $k$ above $p$, where $c\in {\rm Gal}(k/\QQ)$ is a generator. Let us assume that $\fp$ is the prime induced from the embedding $\iota_p$. Let $k(p^{\infty})$ denote the compositum of all $\bZ_{p}$-extensions of $k$ and set $\Gamma_\infty={\rm Gal}(k(p^\infty)/k)$. In the setting we have placed ourselves where $k$ is an imaginary quadratic field, we have $\Gamma_\infty\simeq \ZZ_p^2$. 

Let $\chi \colon G_{k} \to \overline{\bQ}^{\times}$ be a non-trivial character with prime-to-$p$ order and let $L/k$ denote the extension cut by $\chi$. 

We let 
\[ 
L_{\fp}^{\chi} \in \bZ_{p}^{\rm ur}[[\Gamma_{\infty}]]
\]  
denote the Katz $p$-adic $L$-function introduced in \cite{katz_imaginary} (associated to the branch character $\chi^{-1}$, in the sense of Hida); see also Section~\ref{section_main_conjANDExceptionalZeros} below for its defining interpolative property when $k$ is a general CM field (whose construction and properties in this more general set-up are the subject of \cite{Katz78, HT93}). Here, $\bZ_{p}^{\rm ur}$ stands for the ring of integers of the completion of the maximal unramified extension of $\bQ_{p}$. Let us note for the sake of clarity that $\Sigma$ (resp. $\Sigma^c$) that appears in the interpolation property \eqref{Eqn_Katz_padic_interpolation} is the set $\{\fp\}$ (resp. $\{\fp^c\}$) when $k$ is imaginary quadratic and $L_{\fp}^{\chi}=L_{p,\Sigma}^\chi$ in this particular setting. We remark that the trivial character $\mathds{1}$ is not in the interpolation range for $L_{\fp}^{\chi}$, in particular, its defining property \eqref{Eqn_Katz_padic_interpolation} allows no conclusion concerning the triviality (or not) of the value $\mathds{1}(L_{\fp}^{\chi})$. It is easy to see that 
$$e:=\#\{v\in \Sigma^c: \mathscr{E}_v(\chi,\mathds{1})=0\}=\begin{cases}
1\,&  \hbox{ if }\chi(G_{k_{\fp^c}})=\{1\}\,,\\
0\,& \hbox{otherwise.}
\end{cases}$$
In view of this observation, expanding on the discussion Hida and Tilouine in \cite[\S1.5]{HT94}, one is lead to predict that 
\begin{equation}
\label{eqn_weak_exceptional_zero_conj_imaginary_quadratic}
L_{\fp}^{\chi} \stackrel{?}{\in} \mathcal{A}_{\Gamma_\infty}^{e}\setminus \mathcal{A}_{\Gamma_\infty}^{e+1}\,\,,
\end{equation}
where $ \mathcal{A}_{\Gamma_\infty}:=\ker\left(\bZ_{p}^{\rm ur}[[\Gamma_{\infty}]]\to \bZ_p^{\rm ur}\right)$ is the augmentation ideal. We shall call this guessed containment the \emph{exceptional zero conjecture for Katz' $p$-adic $L$-function (over imaginary quadratic fields) at the trivial character}.

Our main results concerning the particular case when $k$ is imaginary quadratic (Theorem~\ref{thm_intro_imaginary_quadratic_full}(ii) below) asserts that the question \eqref{eqn_weak_exceptional_zero_conj_imaginary_quadratic} of Hida and Tilouine has an affirmative answer. This result follows from a non-critical exceptional zero formula (recorded as Theorem~\ref{thm_intro_imaginary_quadratic_full}(i) below), which one may think of an instance of $p$-adic Beilinson conjecture in the presence of exceptional zeros. Before stating Theorem~\ref{thm_intro_imaginary_quadratic_full}, we need to introduce more notation.

Let $k_{\Gamma}/k$ be any $\bZ_{p}$-extension with Galois group $\Gamma$ and set $\mathcal{A}_\Gamma:=\ker\left(\bZ_{p}^{\rm ur}[[\Gamma]]\to \bZ_p^{\rm ur}\right)$. We denote by $L_{\fp}^{\chi}\vert_{\Gamma} \in \bZ_{p}^{\rm ur}[[\Gamma]]$  the image of $L_{\fp}^{\chi}$ under the natural map $\bZ_{p}^{\rm ur}[[\Gamma_\infty]]\to \bZ_{p}^{\rm ur}[[\Gamma]]$. Let us assume that $e>0$. Note that this is equivalent to requirement that $\chi(\fp^c)=1$ in the present setting. We define the $\cL$-invariant along $\Gamma$ as the unique element $\cL_{\frak{p},\Gamma} \in \QQ_p\otimes_{\Z_p}\Gamma$ that validates the equality 
\begin{equation}
\label{eqn_L_invariant_imaginary_quadratic}
    {\rm rec}_{\Gamma,\frak{p}^c}=\cL_{\frak{p},\Gamma}\cdot\,{\rm ord}_{\frak{p}^c}\,.
\end{equation}
Here, 
$${\rm rec}_{\Gamma,\frak{p}^c} \in {\rm Hom}\left((L\otimes_{k,\iota_p^c}\QQ_p)^{\times,\chi},\Gamma\right)={\rm Hom}\left(k_{\frak{p}^c}^\times,\Gamma\right)={\rm Hom}\left(\QQ_p^\times,\Gamma\right)$$ 
 is induced from the local Artin reciprocity map (and the first equality is because $\chi(\frak{p}^c)=1$); whereas
$${\rm ord}_{\frak{p}^c} \in {\rm Hom}\left((L\otimes_{k,\iota_p^c}\QQ_p)^{\times,\chi},\Z_p\right)={\rm Hom}\left(k_{\frak{p}^c}^\times,\bZ_{p}\right)$$ 
is the $\frak{p}^c$-adic valuation. 

We are now ready to state our main results in the particular case of imaginary quadratic fields.
\begin{theorem}
\label{thm_intro_imaginary_quadratic_full} Suppose $k$ is an imaginary quadratic field where the prime $p=\frak{p}\frak{p}^c$ splits and whose class number is prime-to-$p$. Let $\frak{p}$ be the prime induced from the embedding $\iota_p$. Suppose that $\chi(\frak{p}^c)=1$ and let $\{u\} \in \cO_L^{\times,\chi}$ denote a basis.
\item[i)] We have $L_{\frak{p}}^\chi(\mathds{1})=0$ and 
\begin{equation}\label{eq_ref_main_intro_imag_quad}
(L_{\frak{p}}^{\chi}\vert_{\Gamma})^{\prime}(\mathds{1}) = \cL_{\frak{p}, \Gamma} \left(1 - \frac{\chi(\frak{p})}{p}\right)\cdot\frac{\log_{\frak{p}} |\iota_p(u)|}{\log |\iota_\infty(u)|^2}\cdot
L'(\chi^{-1},0),
\end{equation}
 where the equality takes place in $ \cA_{\Gamma}/\cA_{\Gamma}^{2}$. 
 
\item[ii)] In particular, $L_{\frak{p}}^\chi \in \cA_{\Gamma_\infty}\setminus  \cA_{\Gamma_\infty}^2$.
\end{theorem}

This theorem is a special case of our much more general results, recorded as Theorem~\ref{thm_intro_leading_term} and Corollary~\ref{cor_Hida_Tilouine_Conj_strongform} below. We refer the reader to Remark~\ref{remark_how_to_deduce_the_case_of_imaginary_fields_from_the_general_case} where we explain how to deduce Theorem~\ref{thm_intro_imaginary_quadratic_full} from these two more general assertions.

\begin{remark} 
\label{remark_Kronecker_Katz_Gross_Ferrero_Greenberg_Koblitz}
We shall highlight two important special cases of Theorem~\ref{thm_intro_imaginary_quadratic_full} (which are already known to experts). We refer the reader to Section~\ref{subsubsec_new_expand_on_remark_1_7_new_remark_1_3} for additional details pertaining to the discussion in this remark.
\item[i)] Suppose $\Gamma=\Gamma_{\frak{p}}$ is the Galois group of the $\Z_p$-extension $k_{\Gamma_{\frak{p}}}$ of $k$ unramified outside $\{\frak{p}\}$. In this situation, we have $\cL_{\Gamma_{\frak{p}}}=1$ (see Remark~\ref{remark_L_invariant_vanishnonvanish}(iv) for a justification of this observation in full generality) and the leading term formula in Theorem~\ref{thm_intro_imaginary_quadratic_full}(i) takes the following form:
\begin{equation}
\label{eqn_Katz_Kronecker_limit_formula_imag_quad_new}
    (L_{\frak{p}}^{\chi}\vert_{\Gamma})^{\prime}(\mathds{1}) = \left(1 - \frac{\chi(\frak{p})}{p}\right)\cdot\frac{\log_{\frak{p}} |\iota_p(u)|}{ \log |\iota_\infty(u)|^2}\cdot
L'(\chi^{-1},0).
\end{equation}
We note that Theorem~\ref{thm_intro_imaginary_quadratic_full}(ii) can be deduced from Theorem~\ref{thm_intro_imaginary_quadratic_full}(i) relying on the fact that $\cL_{\Gamma_{\frak{p}}}\neq 0$ (since it equals $1$). We also remark that the equality \eqref{eqn_Katz_Kronecker_limit_formula_imag_quad_new} can be deduced directly as a consequence of Katz' $p$-adic Kronecker limit formula. 
\item[ii)] Let us assume now that $\chi$ is the restriction to $G_k$ of an even character $\chi_{\QQ} \colon G_{\bQ} \to \overline{\bQ}^{\times}$ of $G_{\bQ}$. In the situation when $\Gamma=\Gamma_{\rm cyc}$,  our leading term formula \eqref{eq_ref_main_intro_imag_quad} can be obtained on combining Gross' factorisation of Katz $p$-adic $L$-function in~\cite[Corollary 3.9]{GroosFactorisation} with results of Ferrero--Greenberg in~\cite{FerrreroGreenberg} and Gross--Koblitz in~\cite{GrossKoblitz} alongside Remark~\ref{remark_L_invariant_for_Gamma_cyc_imaginary_quadratic}; see also \cite[Theorem 3.3.2]{Benois3}.
\end{remark}

\begin{remark} 
\label{remark_L_invariant_for_Gamma_cyc_imaginary_quadratic} 
In this paragraph, we shall compare the $\cL$-invariant we have introduced in \eqref{eqn_L_invariant_imaginary_quadratic} to their classical counterparts, which were defined by Greenberg when $\Gamma=\Gamma_{\rm cyc}$ is the Galois group of the cyclotomic $\Z_p$-extension of $k$. Let {$\chi_{\rm cyc} \colon \Gamma_{\rm cyc} \xrightarrow{\sim} p\bZ_{p}$ denote the composite map of the cyclotomic character and the $p$-adic logarithm. Then, 
$$
\chi_{\rm cyc}(\cL_{\frak{p},\Gamma_{\rm cyc}})=-\frac{\log_{p}(u)}{{\rm ord}_{\frak{p}^c}(u)}\,,
$$
where $u\in X=\ker\left(\cO_L[1/p]^{\times,\chi}\to L_{\frak{p}}^\times\right)$ is an arbitrary non-trivial element and 
\[
\log_p \in  {\rm Hom}\left((L\otimes_{k,\iota_p^c}\QQ_p)^{\times,\chi},p\bZ_{p}\right)={\rm Hom}\left(\QQ_p^\times,p\bZ_{p}\right)
\] 
is Iwasawa's $p$-adic logarithm, which is normalized so that $\log_p(p)=0$. 

Let us assume further that $\chi$ is the restriction of a character $\chi_{\QQ} \colon G_{\bQ} \to \overline{\bQ}^{\times}$ of $G_{\bQ}$.
Put $F := \overline{\bQ}^{\ker(\chi)}$ and let $\epsilon_{k} \colon \Gal(k/\bQ) \to \{\pm 1\}$ denote the non-trivial character.
Then,
\[
\chi_{\rm cyc}(\cL_{\frak{p},\Gamma_{\rm cyc}})=-\frac{\log_{p}(u_{\rm Gr})}{{\rm ord}_{p}(u_{\rm Gr})}\,,
\]
where 
\[
u_{\rm Gr}:=e_\chi \cdot x \in 
\begin{cases}
\cO_{F}[1/p]^{\times,\chi} & \text{ if $\chi$ is odd},
\\
\cO_{L}[1/p]^{\times,\chi\epsilon_{k}} & \text{ if $\chi$ is even}. 
\end{cases}
\]
and $x\in \cO_L[1/p]^{\times,\pm}$ is such that $\cO_L\cdot x=\wp^h(\wp^{c})^{\pm h}$ for any prime $\wp$ of $L$ above $p$ and such that $\langle c \rangle=\Gal(k/\QQ)$ acts on it by $\pm 1$.  

In particular, Brumer's $p$-adic version of Baker's theorem implies that $\cL_{\frak{p},\Gamma_{\rm cyc}}$ is non-zero in the special case when $\chi$ arises as the restriction of a character of $G_{\QQ}$.} 
\end{remark}

\subsection{Notation}
Before we describe our main results in full generality, we need to introduce more notation.

Let $p$ be an odd prime and let $\mu_{p^{n}}$ denote the set of $p^{n}$th roots of unity. 
For any field $K$, we write $G_{K}$ for the absolute Galois group of $K$.  

Let $\bZ_{p}^{\rm ur}$ denote the ring of integers of the completion of the maximal unramified extension of $\bQ_{p}$. 
Note that $\bZ_{p}^{\rm ur}$ equals the Witt ring of an algebraic closure of $\bF_{p}$ and hence $\bZ_{p}^{\rm ur}$ is a complete discrete valuation ring  with maximal ideal $p\bZ_{p}^{\rm ur}$. 

For any number field $K$, we denote by $S_{p}(K)$ and $S_{\infty}(K)$ the set of places of $K$ above $p$ and $\infty$, respectively. 
For a finite set $S \supseteq S_{\infty}(K)$, we denote by $K_{S}$ the maximal extension of $K$ which is unramified outside $S$. 
Put 
\[
G_{K,S} := \Gal(K_{S}/K).
\] 
For a prime $v$ of $K$, we denote by $K_{v}$ the completion of $K$ at $v$ and by $\cO_{K_{v}}$ the ring of integers of $K_{v}$. 

Henceforth, we shall denote by $k$ a general CM field and let $k^{+}$ denote its maximal totally real subfield. We let $c$ denote a generator of $\Gal(k/k^+)$ and set 
\[
g := [k^{+} \colon \bQ].
\] 
We assume that the following $p$-ordinary condition of Katz holds true:
\begin{itemize}
\item[\bf{(}\mylabel{item_ord}{\bf ord}\bf{)}] Every prime of $k^{+}$ above $p$ splits in $k$\,.
\end{itemize}
We fix a subset $\Sigma \subseteq S_{p}(k)$ such that  $\Sigma \cup \Sigma^{c} = S_{p}(k)$ and $\Sigma \cap \Sigma^{c} = \emptyset$.

Throughout this paper, for each prime $v$ of $k$, we fix a prime $w$ in $\overline{\bQ}$ above $v$.  
For any set $S$ of places of $k$ and algebraic extension $K/k$, we denote by $S_{K}$ the set of places of $K/k$ lying above $S$ and by $\cO_{K,S}$ the ring of $S_{K}$-integers in $K$. 
We also write $S_{\rm ram}(K/k)$ for the set of primes at which $K/k$ is ramified.

Let $\chi \colon G_{k} \to \overline{\bQ}^{\times}$ be a non-trivial character with prime-to-$p$ order. Set $L := \overline{\bQ}^{\ker(\chi)}$. 
Throughout this paper, we assume that the prime $p$ is unramified in $L$. We put  
\[
e_{\chi} := \frac{1}{[L \colon k]} \sum_{\sigma \in \Gal(L/k)}\chi(\sigma)\sigma^{-1}  
\]
and, for any $\bZ_{(p)}[{\rm im}(\chi)][\Gal(L/k)]$-module $M$, define $M^{\chi} := e_{\chi}M$.

\begin{remark} Let us briefly comment on our assumptions on the choice of $p$.
\item[i)] The assumption that $p \nmid [L \colon k]$ is used only to guarantee that the functor $M \mapsto M^\chi$ is exact in the category of $\ZZ_p[{\rm im}(\chi)][\Gal(L/k)]$-modules. We note that the formulation of the hypothesis \eqref{item_H5} below (which we need in our proof the main result of this article) requires this as an input. We expect that a leading term formula akin to the one proved in our main result (Theorem~\ref{thm_intro_leading_term} below) should hold (up to powers of $p$) also in the more general case when $p$ is allowed to divide the order of $\chi$. Note that when $\chi$ has order divisible by $p$, passing to $\chi$-parts of Rubin--Stark elements (or elliptic units in the special case when $k$ is imaginary quadratic) will require to invert $p$.

\item[ii)] The assumption that $p$ be unramified in $L$ is used in the construction of Coleman power series and the computation of their constant terms in Section~\ref{section_Coleman_map}. 
\end{remark}

Let $T := \bZ_{p}^{\rm ur}(1) \otimes \chi^{-1}$, that is, $T$ is equal to $\bZ_{p}^{\rm ur}$ as a $\bZ_{p}^{\rm ur}$-module and $G_{k}$ acts on $T$ via the character 
$\chi_{\rm cyc}\chi^{-1}$, where $\chi_{\rm cyc}$ denotes the cyclotomic character of $k$. 
Since the prime $p$ is unramified in $L$, we note that 
\[
H^{0}(G_{k,S},T/pT) = 0
\]  
for any finite set $S$ of places of $k$ containing $S_{\infty}(k) \cup S_{p}(k) \cup S_{\rm ram}(L/k)$. 

For an abelian extension $K/k$, we let $\iota: \Z_p[[\Gal(K/k)]] \longrightarrow \Z_p[[\Gal(K/k)]]$ denote the involution given by $g\mapsto g^{-1}$ on the group-like elements. For a $\Z_p[[\Gal(K/k)]]$-module $M$, we set $M^\iota:=M\otimes_{\Z_p[[\Gal(K/k)]],\iota} \Z_p[[\Gal(K/k)]]$. 
We also set 
\[
T_{K} := T \otimes_{\bZ_{p}} \bZ_{p}[[\Gal(K/k)]]^{\iota}. 
\]
We note that for any set $S \supseteq S_{\infty}(k) \cup S_{p}(k) \cup S_{\rm ram}(KL/k)$, the Galois group $G_{k,S}$ acts on $T_{K}$. 
By Hilbert's Theorem 90,  there is a canonical identification  
\begin{align}\label{eq:hilb1}
\varprojlim_{K'}(\bZ_{p}^{\rm ur} \otimes_{\bZ_{p}} (\cO_{LK', S})^{\times, \wedge})^{\chi} = H^{1}(G_{k,S},T_{K}). 
\end{align}
Here $K'$ runs over all finite extension of $k$ contained in $K$ and we denote by 
\[
M^{\wedge} := \varprojlim_{n}M/p^{n}M
\] 
the pro-$p$-completion  of a $\bZ$-module $M$. 

For any prime $v$ of $k$, we define  
\begin{align}\label{finite condition}
H^{1}_{f}(G_{k_{v}},T_{K}) := \varprojlim_{K'}(\bZ_{p}^{\rm ur} \otimes_{\bZ_{p}} (\cO_{KL}\otimes_{\cO_{k}} \cO_{k_{v}})^{\times,\wedge})^{\chi}.
\end{align}
 Let $k(p^{\infty})$ denote the compositum of all $\bZ_{p}$-extensions of $k$ and set $\Gamma_\infty={\rm Gal}(k(p^\infty)/k)$. We put 
\[
\bT_{\infty} := T_{k(p^{\infty})} \ \text{ and } \ \Gamma_{\infty} := \Gal(k(p^{\infty})/k). 
\]
If $K/k$ is a $\bZ_{p}^{s}$-extension with Galois group $\Gamma$ for some $s \geq 1$, then 
we also define 
\[
\bT_{\Gamma} := T_{K} \ \text{ and } \ \cA_{\Gamma} := \ker\left(\bZ_{p}^{\rm ur}[[\Gamma]] \longrightarrow \bZ_{p}^{\rm ur}\right). 
\]

\subsection{Results in the general case}
\label{sec_results_general_new}
We will close our introduction explaining our results in full generality. let 
\[ L_{p,\Sigma}^{\chi} \in \bZ_{p}^{\rm ur}[[\Gamma_{\infty}]]\]  
denote  Katz' $p$-adic $L$-function introduced in \cite{Katz78, HT93}; see also  \eqref{Eqn_Katz_padic_interpolation} below for its defining interpolative property. We remark that the trivial character $\mathds{1}$ is not in the interpolation range for $L_{p,\Sigma}^{\chi}$, in particular, its defining property \eqref{Eqn_Katz_padic_interpolation} allows no conclusion concerning the triviality (or not) of the value $\mathds{1}(L_{p,\Sigma}^{\chi})$. Notice that 
\begin{equation}
\label{eqn_define_e_in_two_ways}
    \#\{v\in \Sigma^c: \mathscr{E}_v(\chi,\mathds{1})=0\}=\#\{v\in \Sigma^c: \chi(G_{k_{v}})=1\}=:e.
\end{equation}
In view of this observation, and expanding on the discussion Hida and Tilouine in \cite[\S1.5]{HT94}, one is lead to predict that 
\begin{equation}
\label{eqn_weak_exceptional_zero_conj}
L_{p,\Sigma}^{\chi} \stackrel{?}{\in} \mathcal{A}_{\Gamma_\infty}^{e}\setminus \mathcal{A}_{\Gamma_\infty}^{e+1}\,\,.
\end{equation}

We shall call this guessed containment the \emph{exceptional zero conjecture for Katz' $p$-adic $L$-function at the trivial character}.
Our Corollary~\ref{cor_Hida_Tilouine_Conj} below asserts that \eqref{eqn_weak_exceptional_zero_conj} indeed holds true if one assumes  a number of folklore conjectures (and as we have noted as part of Theorem~\ref{thm_intro_imaginary_quadratic_full}, it holds true unconditionally in the situation when $k/\QQ$ is imaginary quadratic). Corollary~\ref{cor_Hida_Tilouine_Conj}  follows from a non-critical exceptional zero formula (Theorem~\ref{thm_intro_leading_term}), which one may think of an instance of $p$-adic Beilinson conjecture in the presence of exceptional zeros, and which involves the $\mathcal{L}$-invariant we introduce also in the present article.

Our main objectives in the current article can be summarized as in the list below of items \eqref{item_ObjA}--\eqref{item_ObjC}. Let $k_{\Gamma}/k$ be any $\bZ_{p}$-extension with Galois group $\Gamma$ and recall that $\mathcal{A}_\Gamma:=\ker\left(\bZ_{p}^{\rm ur}[[\Gamma]]\to \bZ_p^{\rm ur}\right)$ is the augmentation ideal. 
\begin{itemize}
 \item[\mylabel{item_ObjA}{\bf A}\bf{)}] Introduce an explicit $\cL$-invariant $\cL_{\Sigma, \Gamma} \in \bC_{p} \otimes_{\bZ_{p}^{\rm ur}}\cA_{\Gamma}^{e}/\cA_{\Gamma}^{e+1}$ along the $\ZZ_p$-extension $k_\Gamma/k$ (see Definition~\ref{defn_L_invariant}).
 \item[\mylabel{item_ObjB}{\bf B}\bf{)}]  Prove a leading term formula for the the restriction $L_{p,\Sigma}^{\chi}\vert_{\Gamma} \in \bZ_{p}^{\rm ur}[[\Gamma]]$  of Katz' $p$-adic $L$-function to $\Gamma$ {(see Theorem~\ref{thm_intro_leading_term} for a precise statement of our result and Theorem~\ref{thm:main} in the main text)} under the following hypotheses:
\begin{itemize}
 \item[(\mylabel{item_H1}{H.1})] No prime in $S_{p}(k)$ splits completely in $k_{\Gamma}/k$.
\item[(\mylabel{item_H2}{H.2})] $p$ splits completely in $k$.
\item[(\mylabel{item_H3}{H.3})] $\Sigma$--Leopoldt Conjecture~\ref{conj:leop} for the number field $L$ is true. 
\item[(\mylabel{item_H4}{H.4})] Rubin--Stark conjecture for each finite subextension of $Lk(p^\infty)/k$ holds true. 
\item[(\mylabel{item_H5}{H.5})] Iwasawa theoretic Mazur--Rubin--Sano Conjecture~\ref{MRS} and the Reciprocity Conjecture~\ref{conj:rec} hold.
\end{itemize}
 \item[\mylabel{item_ObjC}{\bf C}\bf{)}]  Conclude that the guessed containment \eqref{eqn_weak_exceptional_zero_conj} holds true under the hypotheses \eqref{item_H1}--\eqref{item_H5} (c.f. Corollary~\ref{cor_Hida_Tilouine_Conj}; see also the final portion of Theorem~\ref{thm:main} in the main text).
\end{itemize}

\begin{theorem}[$p$-adic Beilinson conjecture with exceptional zeros]
\label{thm_intro_leading_term}
Let $e\in \ZZ_{\geq 0}$ be as in \eqref{eqn_define_e_in_two_ways} and suppose that the hypotheses \eqref{item_H1}--\eqref{item_H5} hold true.

Then ${\rm ord}_{\cA_{\Gamma}}\left(L_{p,\Sigma}^\chi\vert_{\Gamma}\right) \geq e$ and  
\begin{equation}\label{Eqn_leading_term_formula}
L_{p,\Sigma}^\chi\vert_{\Gamma} =  (-1)^{e} \cL_{\Sigma, \Gamma} \cdot 
\prod_{v \in \Sigma}\left(1 - \frac{\chi(v)}{p}\right) \prod_{v \in \Sigma^{c} \setminus V_{p}}(1-\chi(v)^{-1})\cdot 
\frac{{\rm Reg}_p^{\Sigma,\chi}}{{\rm Reg}_\infty^{\chi}}\cdot
L^{*}(\chi^{-1},0)
\end{equation}
inside $ \cA_{\Gamma}^{e}/\cA_{\Gamma}^{e+1}$, where 
\begin{itemize}
\item $V_p$ is the set of places of $k$ above $p$ which split completely in $L$, \item $L^{*}(\chi^{-1},0) := {\displaystyle \lim_{s\to 0}s^{-g}L(\chi^{-1},s)}$ is the leading term of the $L$-function associated to $\chi^{-1}$ at $s=0$, 
\item  ${\rm Reg}_p^{\Sigma,\chi}/{\rm Reg}_\infty^{\chi}$ is the determinant of the compositum of the $\bC_p$-linear map 
\[
\bC_{p}^{g}\xrightarrow{\left(\lambda_{L,S_{\infty}(k)}^{\chi}\right)^{-1}}\bC_{p} \otimes_{\bZ} \cO_{L}^{\times,\chi} \xrightarrow{\oplus_{v \in \Sigma}\log_{v}} \bC_{p}^{g}\,. 
\]
Here, $\lambda_{L,S_{\infty}(k)}^{\chi}$ is the $\chi$-part of the Dirichlet regulator isomorphism $($given as in  Equation \eqref{isom:reg}$)$ and $\log_{v}$ is the $v$-adic logarithm $($see Equation \eqref{def:log} for its definition$)$. 
\end{itemize}
\end{theorem}

We note that we present in Section~\ref{section_main_conjANDExceptionalZeros} (c.f. Proposition~\ref{imp2} and Corollary~\ref{corollary_imp2}) an alternative proof of the assertion that ${\rm ord}_{\cA_{\Gamma}}(L_{p,\Sigma}^\chi\vert_{\Gamma}) \geq e$ (a weak form of the exceptional zero conjecture) under less restrictive hypotheses, based on our discussion in Appendix~\ref{appendix_main_conjecture}.

\begin{remark}\label{remark_L_invariant_vanishnonvanish}
\item[i)] The leading term formula \eqref{Eqn_leading_term_formula} follows on combining Theorem~\ref{thm:main} with the discussion in Remark~\ref{rem_arch_reg_rubin_stark}.
\item[ii)] We need to assume \eqref{item_H2} solely because the Coleman maps that we recall in Section~\ref{section_Coleman_map} are currently available only under this hypothesis.
\item[iii)] If a prime in $\Sigma^{c}$ splits completely in $k_{\Gamma}/k$, then $\cL_{\Sigma,\Gamma} = 0$ (see Remark~\ref{remark_L_invariant_properties}(ii)) and Iwasawa Main Conjecture implies (see Corollary~\ref{lem:vanish-line} below) that $L_{p,\Sigma}^\chi\vert_{\Gamma}$ is identically zero. 
\item[iv)] Let $k_{\Gamma_\Sigma}/k$ be a $\Z_p$-extension which is unramified outside $\Sigma$; under the validity of the $\Sigma$--Leopoldt conjecture there is a unique such extension. It follows from Definition~\ref{defn_L_invariant} of the $\mathcal{L}_{\Sigma,\Gamma_\Sigma}$  (and the fact that ${\rm rec}_{\Gamma_{\Sigma},v}={\rm ord}_v$ whenever $v\in \Sigma^c$) that $\mathcal{L}_{\Sigma,\Gamma_\Sigma}=1$.
\item[v)] The term $\frac{{\rm Reg}_p^{\Sigma,\chi}}{{\rm Reg}_\infty^{\chi}}$ is non-zero if and only if the $\Sigma$--Leopoldt conjecture holds true.
\item[iv)] It would be interesting to deduce the leading term formula \eqref{Eqn_leading_term_formula} in the particular case $\Gamma=\Gamma_{\rm cyc}$, extending the methods of \cite{DarmonDasguptaPollack, DasguptaKakdeVentullo}, without relying on the conjectural Rubin--Stark elements and the conjectural reciprocity law\footnote{During the review process of the present article, the authors learned that this approach is being employed in the forthcoming work of Betina and Hsieh~\cite{BetinaHsieh}. We are grateful to Ming-Lun Hsieh for informing us about this development and for sharing an initial draft of their paper with us.}. It seems to the authors that the methods in op. cit. are not likely to allow treatment beyond the case $\Gamma=\Gamma_{\rm cyc}$. Assuming $e=1$, one may hope that an explicit construction of a test vector $u\in X$ (where $X$ is given in Lemma~\ref{lemma_X}) in a manner that parallels \cite{DarmonDasguptaPollack} would allow us to conclude that ${\cL}_{\Sigma,\Gamma_{\rm cyc}}\neq 0$. The case $e>1$ seems to require new ideas. We finally remark that the argument of Remark~\ref{remark_L_invariant_for_Gamma_cyc_imaginary_quadratic} applies for a general CM field and shows that, if $\chi$ arises as the restriction of a character of $G_{k^+}$ to $G_k$ and $e=1$, then ${\cL}_{\Sigma,\Gamma_{\rm cyc}}\neq 0$. 
\end{remark}

\begin{corollary}
\label{cor_Hida_Tilouine_Conj}
In the notation and hypotheses of Theorem~\ref{thm_intro_leading_term}, we have 
$${\rm ord}_{\cA_{\Gamma}}\left(L_{p,\Sigma}\vert_{\Gamma}\right)= e 
 \quad \iff{}\quad  \cL_{\Sigma, \Gamma} \neq 0.\,. $$
\end{corollary}
\begin{corollary}
\label{cor_Hida_Tilouine_Conj_strongform}
Under the hypotheses of Theorem~\ref{thm_intro_leading_term}, prediction \eqref{eqn_weak_exceptional_zero_conj} holds true:
$$L_{p,\Sigma}^{\chi}\,{\in}\, \mathcal{A}_{\Gamma_\infty}^{e}\setminus \mathcal{A}_{\Gamma_\infty}^{e+1}\,.$$
\end{corollary}
\begin{proof}
This is an immediate consequence of Theorem~\ref{thm_intro_leading_term} and Remark~\ref{remark_L_invariant_vanishnonvanish}(iv).
\end{proof}

\begin{remark}
\label{remark_how_to_deduce_the_case_of_imaginary_fields_from_the_general_case}
In this remark, we shall explain how to deduce Theorem~\ref{thm_intro_imaginary_quadratic_full} (our main results in the particular case when $k/\QQ$ is an imaginary quadratic extension) from Theorem~\ref{thm_intro_leading_term} and Corollary~\ref{cor_Hida_Tilouine_Conj_strongform}. To do so, we need to check that
\begin{itemize}
    \item [i)] the hypotheses \eqref{item_H1}--\eqref{item_H5} hold true,
    \item [ii)] the identity \eqref{Eqn_leading_term_formula} agrees with \eqref{eq_ref_main_intro_imag_quad}
\end{itemize}
 in the setting of Theorem~\ref{thm_intro_imaginary_quadratic_full}.

 The hypothesis \eqref{item_H1} trivially holds true when $k/\QQ$ is imaginary quadratic. Since we have assumed in Theorem~\ref{thm_intro_imaginary_quadratic_full} that the prime $p$ splits in $k/\QQ$, the condition \eqref{item_H2} is also valid. The hypotheses \eqref{item_H3}--\eqref{item_H5} hold true since we assume in Theorem~\ref{thm_intro_imaginary_quadratic_full} that the class number of $k$ is coprime to $p$ (c.f. Remark~\ref{rem:MRS}(ii) and Remark~\ref{rem:rec1} below).  

Note that $e=1$ in the setting of Theorem~\ref{thm_intro_imaginary_quadratic_full}. The definition of $\cL_{\Sigma,\Gamma}$ reduces to that of $\cL_{\fp,\Gamma}$ given as in \eqref{eqn_L_invariant_imaginary_quadratic} with $\Sigma=\{\fp\}$. Since $V_p\supset \{\fp^c\}=\Sigma^c$, the Euler-like factors also agree. Finally, observe that
\[
\frac{{\rm Reg}_p^{\{\fp\},\chi}}{{\rm Reg}_\infty^{\chi}}= - \frac{\log_{\frak{p}} |\iota_p(u)|}{\log |\iota_\infty(u)|^2}\,,
\]
where $\{u\} \subset \cO_L^{\times,\chi}$ is a basis.
\end{remark}

\begin{remark}
In the follow-up paper \cite{BS_IMRN}, the authors proved that the $\cL$-invariants $\cL_{\Sigma,\Gamma}$ (which appear in our leading term formula \eqref{Eqn_leading_term_formula} and given as in Definition~\ref{defn_L_invariant} below) can be interpolated as $\Gamma$ varies to a universal (multivariate) $\cL$-invariant $\cL_{\Sigma,\Gamma_\infty}\in \cA_{\Gamma_\infty}^e/\cA_{\Gamma_\infty}^{e+1}$ (c.f. Section 3.4 in op. cit.), by interpreting $\cL_{\Sigma,\Gamma}$ in terms of $p$-adic height pairings. 
\end{remark}
\begin{acknowledgements}
{K.B. thanks Ming-Lun Hsieh for inquiring about K.B.'s prior work on exceptional zeros of $p$-adic $L$-functions~\cite{kbbCMH,kbbMTT} (as well as for his comments and suggestions on an earlier draft), which started the train of thought that lead to this project. He also thanks Academica Sinica for their hospitality. The authors would like to thank the anonymous referee, Mladen Dimitrov and Takamichi Sano for their very helpful comments and suggestions.

The authors' research is supported by: European Union's Horizon 2020 research and innovation programme under the Marie Sk{\l}odowska-Curie Grant Agreement No. 745691 (B\"uy\"ukboduk);  the Program for Leading Graduate Schools, MEXT, Japan and JSPS KAKENHI Grant Number 17J02456 (Sakamoto).} 
\end{acknowledgements}

$\,$

\section{Rubin--Stark elements and Conjecture of Mazur--Rubin--Sano}\label{subsec:Rubin--Stark}

In this section, we first recall (in Section~\ref{subsec_RS_elements_and_conj}) the definition of the $(p,\chi)$-part of the Rubin--Stark element for our fixed prime $p$ and character $\chi$ and recall the Rubin--Stark conjecture. 
In Section~\ref{subsec_conj_MRS_02_02_2022}, we then recall a certain special case of a conjecture of Mazur--Rubin--Sano that was formulated in~\cite{MRselmer,sano}.

\subsection{Rubin--Stark elements}
\label{subsec_RS_elements_and_conj}$\,$
The contents of this subsection, where we review the definition of Rubin--Stark elements and their properties, dwells largely on \cite{rubinstark,sano,bks2}.

Let $K/k$ be a finite abelian extension satisfying $K \cap L = k$.  We then have a canonical identification $\Gal(KL/k) = \Gal(K/k) \times \Gal(L/k)$. Fix a finite set $S$ of places of $k$ such that $S_{\infty}(k) \cup S_{\rm ram}(KL/k) \subseteq S$. 
We also fix a labeling $S = \{v_{1}, \ldots ,v_{n}\}$ and a positive integer $r \leq n$. 
Assume that the places $v_{1}, \ldots, v_{r}$ split completely in $KL$ and then set $V := \{v_{1}, \ldots, v_{r}\}$.

Given a set $\Sigma$ of places in $k$, we define the $\bZ[\Gal(KL/k)]$-module $X_{KL,\Sigma}$ to be the augmentation kernel of the divisor group on $\Sigma_{KL}$: 
\[
X_{KL, \Sigma} := \ker\left(\bigoplus_{w \in \Sigma_{KL}}\bZ w \longrightarrow \bZ; w \longmapsto 1 \right). 
\] 
Here, we denote by $\Sigma_{KL}$ the set of places of $KL$ lying above the places in $\Sigma$. In the particular case when $\Sigma=S$, we shall write $\cO^{\times}_{KL, S}$ for the ring of $S_{KL}$-integers in $KL$. 
We put 
\[
\bC \cO^{\times}_{KL, S} := \bC \otimes_{\bZ} \cO^{\times}_{KL, S} \ \text{ and } \ 
\bC X_{KL, S} := \bC \otimes_{\bZ} X_{KL, S}. 
\]
It is well-known that there is the regulator isomorphism 
\begin{align}\label{isom:reg}
\lambda_{KL,S} \colon \bC \cO^{\times}_{KL, S} \stackrel{\sim}{\longrightarrow} \bC X_{KL, S}
\end{align}
of $\bC[\Gal(KL/k)]$-modules, which is given for $a \in \cO^{\times}_{KL, S}$ by 
$$\lambda_{KL,S}(a) = - \sum_{w \in S_{KL}}\log |a|_{w} w\,.$$

For a non-trivial character $\psi \in \Hom(\Gal(KL/k), \bC^{\times})$, let $L_{S}(\psi,s)$ denote the $S$-truncated Dirichlet $L$-function for $\psi$ and 
\[
e_{\psi} := \frac{1}{\#\Gal(KL/k)}\sum_{\sigma \in \Gal(KL/k)}\psi(\sigma)\sigma^{-1}. 
\]
We write $r_{\psi, S}$ for the order of vanishing of $L_{S}(\psi, s)$ at $s = 0$. 
Since $\psi$ is non-trivial, the class number formula shows that 
\[
r_{\psi, S} = \dim_{\bC}(e_{\psi}\bC \cO_{KL,S}^{\times}) = \dim_{\bC}(e_{\psi}\bC X_{KL,S}) = 
\# \{v \in S \mid \psi(G_{k_{v}}) = 1\}\,. 
\] 
Therefore, by the assumption of $V$, we have $r \leq r_{\psi, S}$ and hence 
${\displaystyle \lim_{s \to 0}s^{-r}L_{S}(\psi, s)} $ is well-defined. 
Set 
\begin{align*}
e_{\chi} \theta^{(r)}_{KL/k,S} := e_{\chi} \left(\sum_{\psi}{\displaystyle \lim_{s \to 0}s^{-r}L_{S}(\psi^{-1}, s)}\cdot e_{\psi} \right) \in e_{\chi} \bC[\Gal(KL/k)]. 
\end{align*}

Recall that, throughout this paper, we shall fix a prime $w_{i}$ of $KL$ lying above $v_{i} \in V$. 
We also note that 
\[
e_{\chi}(\bC X_{KL,S}) = e_{\chi}\biggl(\bigoplus_{w \in \Sigma_{KL}}\bC w\biggr)
\]
since $\chi$ is a non-trivial character. 
\begin{definition}
The $\chi$-part of the ($r^{\rm th}$ order) Rubin--Stark element 
\[
\epsilon^{V,\chi}_{KL/k,S} \in e_{\chi}\left({\bigwedge}_{\bC_{p}[\Gal(KL/k)]}^{r}\bC\cO^{\times}_{KL, S}\right) = 
{\bigwedge}_{\bC[\Gal(K/k)]}^{r}(\bC\cO^{\times}_{KL, S})^{\chi}
\]
is defined to be the element which corresponds to
\[
e_{\chi} \theta^{(r)}_{KL/k,S} \cdot w_{1} \wedge \cdots \wedge w_{r}  \in e_{\chi} \left({\bigwedge}_{\bC[\Gal(KL/k)]}^{r}\bC X_{LK, S} \right)
\]
under the isomorphism 
\[
{\bigwedge}_{\bC[\Gal(KL/k)]}^{r}\bC \cO^{\times}_{KL, S} \cong {\bigwedge}_{\bC[\Gal(KL/k)]}^{r}\bC X_{LK, S} 
\]
induced by the regulator isomorphism~(\ref{isom:reg}). 
By using the fixed isomorphism $j \colon \bC \cong \bC_{p}$, we regard 
\[
\epsilon^{V,\chi}_{KL/k,S} \in {\bigwedge}_{\bC_{p}[\Gal(K/k)]}^{r}(\bC_{p}\cO^{\times}_{KL, S})^{\chi}. 
\]
Here $\bC_{p}\cO^{\times}_{KL, S} := \bC_{p} \otimes_{\bZ} \cO^{\times}_{KL, S}$. 
\end{definition}

\begin{remark}
\begin{itemize}
\item[(i)] The Rubin--Stark element depends on the choice of the fixed prime $w_{i}$ above $v_{i}$ for $1 \leq i \leq r$. 
For example, if we take another prime $w_{1}'$ in $KL$ satisfying $w_{1}' = \sigma w_{1}$ for some $\sigma \in \Gal(KL/k)$, then we get another element $\epsilon'$ and we have 
\[
\epsilon' = \sigma \cdot \epsilon^{V,\chi}_{KL/k,S}. 
\]
\item[(ii)] The Rubin--Stark element depends on the labeling of the elements $v_{1},\ldots,v_{r}$ in $V$. 
\end{itemize}
\end{remark}

By using the canonical isomorphism 
\begin{align*}
\Hom_{\bZ}(\cO_{KL,S}^{\times},\bZ) &\stackrel{\sim}{\longrightarrow} 
\Hom_{\bZ[\Gal(KL/k)]}(\cO_{KL,S}^{\times},\bZ[\Gal(KL/k)]); 
\\
f &\longmapsto \left(x \longmapsto \sum_{\sigma \in \Gal(KL/k)}f(\sigma x)\sigma^{-1}\right), 
\end{align*}
the order map $\cO_{KL,S}^{\times} \hookrightarrow (KL)_{w}^{\times} \twoheadrightarrow \bZ$ at $w$ induces a $\bZ[\Gal(KL/k)]$-homomorphism
\begin{equation}\label{eqn_define_ord_map}
{\rm ord}_{v} \colon \cO_{KL,S}^{\times} \longrightarrow \bZ[\Gal(KL/k)]. 
\end{equation}

\begin{remark}\label{rem:ord}
This notation is lightly abusive in that the map ${\rm ord}_{v}$ depends on the choice of the fixed prime $w$ above $v$. 
Namely, if we take another prime $w'$ in $KL$ satisfying $w' = \sigma w$ for some $\sigma \in \Gal(KL/k)$, then there is the map ${\rm ord}_{v}'$ associated with $w'$ and we have 
\[
{\rm ord}_{v}' = \sigma^{-1} \cdot {\rm ord}_{v}. 
\]
\end{remark}

For any subset $T \subseteq S \setminus S_{\infty}(k)$ and non-negative integer $s$, 
the maps $\{{\rm ord}_{v}\}_{v \in T}$ induce a map 
\[
\bigwedge_{v \in T}{\rm ord}_{v} \colon 
{\bigwedge}_{\bC_{p}[\Gal(K/k)]}^{s + \#T}(\bC_{p}\cO^{\times}_{KL, S})^{\chi} \longrightarrow {\bigwedge}_{\bC_{p}[\Gal(K/k)]}^{s}(\bC_{p}\cO^{\times}_{KL, S})^{\chi}
\]
(c.f. Equation \eqref{eq:Phi} in Appendix~\ref{sec:bi-dual}). 

\begin{proposition}\label{prop:stark-reln}
Let $K' \subseteq K$ be a subfield containing $k$ and $S_{\infty}(k) \cup S_{\rm ram}(K'L/k) \subseteq S'$. 
\begin{itemize}
\item[(i)] If $V \subseteq S'$, then 
\[
{\rm N}_{K/K'}(\epsilon^{V, \chi}_{KL/k, S}) = 
\prod_{v \in S \setminus S'}(1-{\rm Frob}_{v}^{-1})\epsilon^{V, \chi}_{K'L/k, S'}. 
\]
Here, ${\rm Frob}_{v}$ denotes the arithmetically normalized Frobenius element at $v$ and 
${\rm N}_{K/K'} \colon (LK)^{\times} \longrightarrow (LK')^{\times}$ denotes the norm map. 
In particular, if $S = S'$, then 
\[
{\rm N}_{K'/K}(\epsilon^{V, \chi}_{K'L/k, S}) = 
\epsilon^{V, \chi}_{KL/k, S}.
\] 
\item[(ii)] Let $V' \subseteq V$ such that $S \setminus V = S' \setminus V'$. Then we have 
\[
\bigwedge_{v \in V\setminus V'}{\rm ord}_{v}(\epsilon^{V, \chi}_{KL/k, S}) = {\rm sgn}(V,V') \epsilon^{V', \chi}_{K'L/k, S'}. 
\]
Here the sign ${\rm sgn}(V,V') \in \{\pm 1\}$ is defined by the relation 
\[
{\rm sgn}(V,V') \bigwedge_{v \in V} =  \bigwedge_{v \in V'} \wedge  \bigwedge_{v \in V \setminus V'}. 
\]
\end{itemize}
\end{proposition}
\begin{proof}
Claim~(i) is \cite[Proposition~6.1]{rubinstark}; see also \cite[Proposition~3.5]{sano}. 
Claim~(ii) is \cite[Lemma~5.1(iv) and Proposition~5.2]{rubinstark}; see also \cite[Proposition~3.6]{sano}. 
\end{proof}

Let $\overline{S} := S \cup S_{p}(k)$. By (\ref{eq:hilb1}), we have 
\[
e_{\chi}(\bZ_{p}^{\rm ur} \otimes_{\bZ} \cO^{\times}_{KL, \overline{S}}) = H^{1}(G_{k,\overline{S}},T_{K}). 
\]
Let $G := \Gal(K/k)$ and put $(-)^{*} := \Hom_{\bZ_{p}^{\rm ur}[G]}(-,\bZ_{p}^{\rm ur}[G])$. 
Then by Remark~\ref{rem:rubin-lattice}, there is a canonical injection 
\[
{\bigcap}^{r}_{\bZ_{p}^{\rm ur}[G]}H^{1}(G_{k,\overline{S}},T_{K}) := \left({\bigwedge}^{r}_{\bZ_{p}^{\rm ur}[G]}H^{1}(G_{k,\overline{S}},T_{K})^{*}\right)^{*} \hookrightarrow {\bigwedge}_{\bC_{p}[\Gal(K/k)]}^{r}(\bC_{p}\cO^{\times}_{KL, \overline{S}})^{\chi}. 
\]
Then the following is the $(p,\chi)$-part of the Rubin--Stark conjecture. 

\begin{conjecture}[Rubin--Stark Conjecture]
The Rubin--Stark element  $\epsilon^{V,\chi}_{KL/k,S}$ is contained in the Rubin-lattice 
${\bigcap}^{r}_{\bZ_{p}^{\rm ur}[G]}H^{1}(G_{k,\overline{S}},T_{K})$. 
\end{conjecture}

\begin{remark}
In the articles \cite{rubinstark,sano,bks2}, the authors choose an auxiliary finite set $T$ of places of $k$ which is disjoint from $S$ and such that the $(S,T)$-unit group of $KL$ is torsion-free, in order to state the Rubin--Stark conjecture. However, since $\chi$ cannot be the Teichm\"uller character thanks to our running hypothesis that $L/\bQ$ be unramified at $p$, the group $H^{0}(G_{k,\overline{S}},T/pT)$ vanishes and therefore, $H^{1}(G_{k,{\overline{S}}},T_{K})$ is torsion-free. For this reason, we can (and will) take $T = \emptyset$. 
\end{remark}

\subsection{Conjecture of Mazur--Rubin--Sano}
\label{subsec_conj_MRS_02_02_2022}
Our goal in this subsection is to formulate an Iwasawa-theoretic variant of a conjecture of Mazur--Rubin--Sano (c.f. Conjecture~\ref{MRS}).

Let $k_{\Gamma}$ be a $\bZ_{p}$-extension of $k$ with Galois group $\Gamma$.  
We write $k_{n}$ for the $n^{\rm th}$ layer of $k_{\Gamma}/k$ and put $\Gamma_{n} := \Gal(k_{n}/k)$. 
We also set $L_{n} := k_{n}L$ and $L_{\Gamma} := k_{\Gamma}L$. 
Note that $\Gal(L_{n}/k) = \Gamma_{n} \times \Gal(L/k)$ since $p \nmid [L \colon k]$. 
Let 
\[
S:=S_{\infty}(k) \cup S_{p}(k) \cup S_{\rm ram}(L/k) 
\ \text{ and } \ 
V := \{v \in S \mid \chi(G_{k_{v}}) = 1\}. 
\]
Note that $S_{\infty}(k) \subseteq V$ since $k$ is a CM field.  
Put $r := \# V - g$. 
We fix a labeling $S = \{v_{1}, \ldots ,v_{n}\}$ such that 
\begin{align*}
S_{\infty}(k) &= \{v_{1},\ldots, v_{g}\}, 
\\
V \setminus S_{\infty}(k) &= \{v_{{g}+1},\ldots, v_{g+r}\}. 
\end{align*}
We then have the $(p,\chi)$-components of the Rubin--Stark elements 
\[
\epsilon^{S_{\infty}(k), \chi}_{L_{n}/k, S} \in {\bigwedge}^{g}_{\bC_{p}[\Gamma_{n}]}(\bC_{p}\cO^{\times}_{L_{n},S})^{\chi} \ \text{ and } \ 
\epsilon^{V, \chi}_{L/k, S} \in {\bigwedge}^{g+r}_{\bC_{p}}(\bC_{p}\cO^{\times}_{L,S})^{\chi}. 
\]
Assume the $(p,\chi)$-part of the Rubin--Stark conjecture, namely that, 
\[
\epsilon^{S_{\infty}(k), \chi}_{L_{n}/k, S} \in {\bigcap}^{g}_{\bZ_{p}^{\rm ur}[\Gamma_{n}]}H^{1}(G_{k,S},T_{k_{n}}) \ \text{ and } \ 
\epsilon^{V, \chi}_{L/k, S} \in {\bigcap}^{g+r}_{\bZ_{p}^{\rm ur}}
H^{1}(G_{k,S},T). 
\]
As $H^{0}(G_{k,S},T/pT)$ vanishes, we have a canonical isomorphism 
\[
{\bigcap}^{g}_{\bZ_{p}^{\rm ur}[[\Gamma]]}H^{1}(G_{k,S},\bT_{\Gamma}) 
\cong
\varprojlim_{n}{\bigcap}^{g}_{\bZ_{p}^{\rm ur}[\Gamma_{n}]}H^{1}(G_{k,S},T_{k_{n}}) 
\]
by Corollary~\ref{cor:inv-bidual}. Here, $\bT_{\Gamma} := T \otimes_{\bZ_{p}} \bZ_{p}[[\Gamma]]$. 
Furthermore, the Rubin--Stark elements $\{\epsilon^{S_{\infty}(k), \chi}_{L_{n}/k, S}\}_{n}$ are norm-compatible by Proposition~\ref{prop:stark-reln}. 
One can therefore define an element 
\begin{equation}\label{eqn_rubin_stark_infty}
\epsilon^{S_{\infty}(k), \chi}_{L_{\Gamma}/k, S} := \varprojlim_{n}\epsilon^{S_{\infty}(k), \chi}_{L_{n}/k, S} \in {\bigcap}^{g}_{\bZ_{p}^{\rm ur}[[\Gamma]]}H^{1}(G_{k,S},\bT_{\Gamma}).
\end{equation}

Let $\cA_{\Gamma} := \ker(\bZ_{p}^{\rm ur}[[\Gamma]] \longrightarrow \bZ_{p}^{\rm ur})$ denote the augmentation ideal of $\bZ_{p}^{\rm ur}[[\Gamma]]$. 
\begin{conjecture}[Exceptional Zero Conjecture for Rubin--Stark elements]\label{MRS1} 
$$\displaystyle{
\epsilon^{S_{\infty}(k), \chi}_{L_{\Gamma}/k, S} \in \cA_{\Gamma}^{r} \cdot {\bigcap}^{g}_{\bZ_{p}^{\rm ur}[[\Gamma]]}H^{1}(G_{k,S},\bT_{\Gamma}).}$$
\end{conjecture}

Recall that $w$ denotes the fixed prime of $L_{\Gamma}$ above $v \in V$. 
We have the reciprocity map 
\[
{\rm rec}_{w} \colon L^{\times}_{w} \longrightarrow \Gal(L_{\Gamma,w}/L_{w})  \hookrightarrow \Gal(L_{\Gamma}/L) \cong \Gamma. 
\]
Since $\chi(G_{k_{v}}) = 1$, so that the prime $v$ splits completely in $L$, 
the fixed prime $w$ induces an isomorphism of $\bZ[\Gal(L/k)]$-modules
\[
\bigoplus_{w' \mid v}L^{\times}_{w'} \cong L_{w}^{\times} \otimes_{\bZ} \bZ[\Gal(L/k)], 
\]
and the reciprocity map ${\rm rec}_{w}$ induces a $\bZ[\Gal(L/k)]$-homomorphism 
\[
\bigoplus_{w' \mid v}L^{\times}_{w'} \longrightarrow \Gamma \otimes \bZ[\Gal(L/k)]. 
\]
On taking the $\chi$-part, we obtain a $\bZ^{\rm ur}_{p}$-homomorphism 
\begin{equation}
\label{eqn_the_map_rec}
{\rm rec}_{\Gamma,v} \colon H^{1}(G_{k_{v}},T) \longrightarrow \bZ^{\rm ur}_{p} \otimes_{\bZ_{p}} \Gamma \cong \cA_{\Gamma}/\cA_{\Gamma}^{2}. 
\end{equation}

\begin{remark}\label{rem:rec}
The map ${\rm rec}_{\Gamma,v}$ depends on the choice of the prime $w$ above $v$, as does ${\rm ord}_v$ defined in \eqref{eqn_define_ord_map}; see Remark~\ref{rem:ord}. 
\end{remark}

By abuse of notation, we also write ${\rm rec}_{\Gamma, v}$ for the composite map 
\[
H^{1}(G_{k,S},T) \xrightarrow{{\rm loc}_{v}} H^{1}(G_{k_{v}},T) \xrightarrow{{\rm rec}_{v}}  \cA_{\Gamma}/\cA_{\Gamma}^{2}. 
\]
We then have an induced $\bZ_{p}^{\rm ur}$-morphism 
\[
\bigwedge_{v \in V \setminus S_{\infty}(k)} {\rm rec}_{\Gamma, v} \colon {\bigcap}^{g+r}_{\bZ_{p}^{\rm ur}}H^{1}(G_{k,S},T) \longrightarrow   
\cA_{\Gamma}^{r}/\cA_{\Gamma}^{r+1} \otimes_{\bZ_{p}^{\rm ur}} 
{\bigcap}^{g}_{\bZ_{p}^{\rm ur}}H^{1}(G_{k,S},T)  
\]
(see Appendix~\ref{sec:bi-dual}). 
Assume that Conjecture~\ref{MRS1} holds true. 
Fix a topological generator $\gamma$ of $\Gamma$. 
Since $H^{1}(G_{k,S},\bT_{\Gamma})$ is a free $\bZ_{p}^{\rm ur}[[\Gamma]]$-module by Corollary~\ref{cor:free}, there is a unique element $\kappa_{\infty, \gamma} \in 
{\bigcap}^{g}_{\bZ_{p}^{\rm ur}[[\Gamma]]}H^{1}(G_{k,S},\bT_{\Gamma})$  such that 
\[
(\gamma -1)^{r} \cdot \kappa_{\infty,\gamma} = \epsilon^{S_{\infty}(k), \chi}_{L_{\Gamma}/k, S}. 
\]
Let $\kappa_{\gamma}$ denote the image of $\kappa_{\infty,\gamma}$ in ${\bigcap}^{g}_{\bZ_{p}^{\rm ur}}H^{1}(G_{k,S},T)$. 
The following is the Iwasawa theoretic version of the conjecture of Mazur--Rubin--Sano (see \cite[Conjecture~4.2]{bks2}). 
\begin{conjecture}[Iwasawa theoretic Mazur--Rubin--Sano Conjecture]\label{MRS}
Suppose that no prime in $S$ splits completely in $k_{\Gamma}/k$. 
Conjecture~\ref{MRS1} holds true and 
\[
\bigwedge_{v \in V \setminus S_{\infty}(k)} {\rm rec}_{\Gamma,w}(\epsilon^{V, \chi}_{L/k, S}) = (-1)^{gr} (\gamma-1)^{r} \otimes \kappa_{\gamma}.
\] 
Note that ${\rm sgn}(V, S_{\infty}(k)) = (-1)^{gr}$. 
\end{conjecture}

\begin{remark}\label{rem:MRS}
\item[(i)] Conjecture~\ref{MRS} is slightly stronger than the conjecture of Mazur--Rubin--Sano in that their original conjecture does not require the validity of Conjecture~\ref{MRS1}. 
However, under Conjecture~\ref{MRS1}, 
Conjecture~\ref{MRS} is equivalent to \cite[Conjecture~4.2]{bks2}. 
\item[(ii)] Assume that $k$ is an imaginary quadratic field, the prime $p$ splits in $k/\bQ$, and the class number of $k$ is coprime to $p$. ETNC for abelian extensions of $k$ in this set up is proved by Bley in \cite{Bley06}. 
Furthermore, Conjecture~\ref{MRS1} holds true by Lemma~\ref{lem:imp1} below (see also Remarks~\ref{rem:iwasawa} and \ref{rem:rec1}). 
It therefore follows from \cite[Theorem~1.1]{bks1} that Conjecture~\ref{MRS} holds true in this situation.
\end{remark}

\section{Coleman Maps}
\label{section_Coleman_map}
In this section, we assume that the prime $p$ splits completely in $k$. Following the discussion in \cite{deshalit}, we recall here the construction of the Coleman maps for ${\bG_{\bf m}}/\QQ_p$.

\subsection{Construction of the Coleman Maps}\label{sec:coleman constr} Fix a prime $v \in S_{p}(k)$ and a $\bZ_{p}$-extension $k_{\Gamma}$ of $k$ in which $v$ is ramified. We also take a finite $p$-abelian extension $K$ of $k$ which is unramified at $v$. 
Recall that $w$ denotes the fixed prime in $k_{\Gamma}KL(\mu_{p})$ above $v$. 
For notational simplicity, we set 
\[
E_{\infty} := (k_{\Gamma}KL(\mu_{p}))_{w} 
\]
and we denote by $H$ the maximal unramified extension of $k_{v}$ contained in $E_{\infty}$. 
We also set 
\[
E_{1} := H(\mu_{p}) 
\]
and $E_{n}$ denotes the Galois extension of $E_{1}$ contained in $E_{\infty}$ such that 
$\Gal(E_{n}/E_{1}) \cong \bZ/(p^{n-1})$. 

Since $k_{v} = \bQ_{p}$ and $E_{\infty}/H$ is a totally ramified $\bZ_{p}^{\times}$-extension, 
there is an element $\pi \in \cO_{k_{v}}$ such that 
\[
{\rm N}_{E_{n}/k_{v}}(E_{n}^{\times}) = \pi^{\bZ} \times (1 + p^{n}\bZ_{p}).
\] 
Furthermore, by the definition of $H$, we can take a uniformizer $\varpi \in \cO_{H}$ such that 
\[
N_{H/k_{v}}(\varpi) = \pi. 
\]
Fix a power series $f(X) = \varpi X + \cdots \in X\cO_{H}[[X]]$ such that 
\[
f(X) \equiv X^{p} \bmod \fm_{H}. 
\]
Here $\fm_{H}$ denotes the maximal ideal of $\cO_{H}$. 
Then there exists a unique relative Lubin-Tate group $F_{f} \in \cO_{H}[[X,Y]]$ defined over $\cO_{H}$ with 
\[
F^{\varphi}_{f}(f(X),f(Y)) = f(F_{f}(X,Y)),
\] 
where $\varphi \in \Gal(H/k_{v})$ denotes the $p^{\rm th}$ Frobenius element (see \cite[Theorem~1.3]{deshalit}). 
Let $\fm_{\bC_{p}}$ denote the maximal ideal of the ring of integers in $\bC_{p}$ and put 
\[
x+_{f}y:=F_{f}(x,y) 
\] 
for $x,y \in \fm_{\bC_{p}}$. 
Then  there is a canonical ring isomorphism 
\[
\cO_{k_{v}} \cong \End(F_{f}) \subseteq \cO_{k_{v}}[[X]]; a \mapsto [a]_{f} = aX + \cdots
\] 
and we define an $\cO_{k_{v}}$-module structure on $(\fm_{\bC_{p}},+_{f})$ by 
$a \cdot x := [a]_{f}(x)$ for $a \in \cO_{k_{v}}$ and $x \in \fm_{\bC_{p}}$. 
Define 
\[
f_{n} := f^{\varphi^{n-1}} \circ \cdots \circ f 
\]
and set 
\[
W_{f}^{n} := \{x \in \fm_{\bC_{p}} \mid f_{n}(x) = 0\}. 
\]
Then it is well-known that 
\[
{\rm N}_{H(W_{f}^{n})/k_{v}}(H(W_{f}^{n})^{\times}) = \pi^{\bZ} \times (1+p^{n}\bZ_{p}) = 
{\rm N}_{E_{n}/k_{v}}(E_{n}^{\times})
\]
(c.f. \cite{yoshida}, Proposition~5.10). 
It follows by local class field theory that 
\[
H(W_{f}^{n}) = E_{n}. 
\]
For each positive integer $n$, we fix a generator $\omega_{n}$ of the $\cO_{k_{v}}$-module $(W^{n}_{f^{\varphi^{-n}}},+_{f^{\varphi^{-n}}})$ such that $f^{\varphi^{-n}}(\omega_{n}) = \omega_{n-1}$.


\begin{proposition}[Theorem~2.2 and Corollary~2.3 in {\cite{deshalit}}]\label{prop:coleman map}
Let 
\[
\beta = (\beta_{n+1})_{n} \in \varprojlim_{n} E_{n}^{\times}
\] 
be a norm-coherent sequence. 
Set $\nu(\beta) = {\rm ord}(\beta_{1})$. 

Then there is a unique power series $g_{\beta}(X) \in X^{\nu(\beta)} \cdot \cO_{H}[[X]]^{\times}$ (called by Coleman power series of $\beta$) such that, for any positive integer $n$,
\[
g_{\beta}^{\varphi^{-n}}(\omega_{n}) = \beta_{n}. 
\]
Furthermore, Coleman power series have the following properties: 
\begin{itemize}
\item[(i)] For any $\beta'  \in \varprojlim_{n} E_{n}^{\times}$, we have $g_{\beta\beta'}(X) = g_{\beta}(X)g_{\beta'}(X)$. 
\item[(ii)] We have $\prod_{\omega \in W_{f}^{1}}g_{\beta}(\omega) = g_{\beta}^{\varphi}(0)$. 
\item[(iii)] If $\beta \in \varprojlim_{n}\cO_{E_{n}}^{\times}$, then we have $(1-\varphi^{-1}) \cdot g_{\beta}(0) = \beta_{0}$. 
\item[(iv)] For any $u \in \cO_{k_{v}}^{\times}$, we have $g_{{\rm rec}(u^{-1})\beta}(X) = g_{\beta}(X) \circ [u]_{f}(X)$. Here 
$${\rm rec} \colon \cO_{k_{v}}^{\times} \longrightarrow \Gal(E_{\infty}/H)$$ 
stands for the reciprocity map. 
\end{itemize}
\end{proposition}

\begin{lemma}\label{lem:const1}
Let $u \in \cO_{k_{v}}^{\times}$ and let $\widetilde{\pi} = (\pi_{n}) \in \varprojlim_{n} E_{n}^{\times}$ be a norm-coherent sequence of uniformizers. 
Put $\beta := ({\rm rec}(u^{-1})-1)\widetilde{\pi} \in \varprojlim_{n}\cO_{E_{n}}^{\times}$ 
Then $g_{\beta}(0) = u$. 
\end{lemma}
\begin{proof}
By Proposition~\ref{prop:coleman map}, we have 
\[
g_{\beta}(X) = g_{\widetilde{\pi}}(X) \circ [u]_{f}(X)/g_{\widetilde{\pi}}(X). 
\]
Since $g_{\widetilde{\pi}}(X) \in X \cdot \cO_{H}[[X]]^{\times}$ and 
$[u]_{f}(X) = uX + \cdots$, the proof that $g_{\beta}(0) = u$ follows. 
\end{proof}

By \cite[Proposition~1.6]{deshalit}, there is an element $\theta(S) \in \bZ_{p}^{\rm ur}[[S]]$ such that 
$\theta$ induces an isomorphism of Lubin-Tate groups: 
\[
\theta \colon \hat{\bG}_{m} \cong F_{f}. 
\]
For each non-negative integer $n$, we fix primitive $p^{n}$th roots of unity $\zeta_{n}$ such that $\zeta_{n+1}^{p} = \zeta_{n}$ 
and we put
\[
\omega_{n} := \theta^{\varphi^{-n}}(\zeta_{n}-1). 
\]
We note that $\omega_{n}$ is a generator of $W^{n}_{f^{\varphi^{-n}}}$ and $f^{\varphi^{-n}}(\omega_{n}) = \omega_{n-1}$.

Take a norm-coherent sequence $\beta = (\beta_{n})_{n} \in \varprojlim_{n}\cO_{E_{n}}^{\times, \wedge}$. 
Since $\beta_{i}$ is a principal unit for each $i$, we see that 
\[
g_{\beta}(X) \equiv 1 \bmod (\varpi, X).
\] 
We may therefore define a power series $\log g_{\beta}(X) \in \cO_{H}[[X]]$. 
By \cite[p.18, Lemma]{deshalit}, the power series 
\[
\widetilde{\log g_{\beta}}(X) := \log g_{\beta}(X) - \frac{1}{p}\sum_{\omega \in W_{f}^{1}}\log g_{\beta}(X+_{f}\omega) 
\]
has integral coefficients. 
The following lemma is a direct consequence of Proposition~\ref{prop:coleman map}. 

\begin{lemma}\label{lem:const}
For a norm-coherent sequence $\beta = (\beta_{n})_{n} \in \varprojlim_{n}\cO_{E_{n}}^{\times, \wedge}$, we have 
\begin{align*}
\widetilde{\log g_{\beta}}(0) = \left(1-\frac{\varphi}{p}\right)\log g_{\beta}(0). 
\end{align*}
\end{lemma}

The reciprocity map defines a map 
\[
\kappa \colon \Gal(E_{\infty}/H) \xrightarrow{{\rm rec}_{w}^{-1}}
\bZ_{p}^{\times} \xrightarrow{x \mapsto x^{-1}} \bZ_{p}^{\times}. 
\]
We then choose an element $\mu_{\beta} \in \bZ_{p}^{\rm ur}[[\Gal(E_{\infty}/H)]]$ such that 
\[
\widetilde{\log g_{\beta}}(X) \circ \theta(S) = \int_{\Gal(E_{\infty}/H)}(1+S)^{\kappa(g)}\, d\mu_{\beta}(g). 
\] 
This element canonically extends to an element $\widetilde{\mu}_{\beta} \in \bZ_{p}^{\rm ur}[[\Gal(E_{\infty}/k_{v})]]$ (see \cite[p.~20]{deshalit}). 

\begin{remark}\label{rem:exten}
Note that 
\[
\int_{\Gal(E_{\infty}/k_{v})}\, d\widetilde{\mu}_{\beta}(g) = [H \colon k_{v}] \cdot \widetilde{\log g_{\beta}}(0) = [H \colon k_{v}] \left(1-\frac{\varphi}{p}\right)\log g_{\beta}(0). 
\]
\end{remark}

As a result of the discussion above, one has the following homomorphism of $\bZ^{\rm ur}_{p}[[\Gal(E_{\infty}/k_{v})]]$-modules: 
\begin{align}\label{eqn_coleman_map_crude_form}
\varprojlim_{n}(\bZ_{p}^{\rm ur} \otimes_{\bZ_{p}} \cO_{E_{n}}^{\times, \wedge}) &\longrightarrow \bZ^{\rm ur}_{p}[[\Gal(E_{\infty}/k_{v})]]\\
\notag\beta &\longmapsto d\widetilde{\mu}_{\beta}\,. 
\end{align}
Since $k_{v} = \bQ_{p}$, we have $\mu_{p^{\infty}} \cap (k_{\Gamma}KL)_{w} = 1$. 
Hence, by \cite[page~21, (13)]{deshalit}, the map \eqref{eqn_coleman_map_crude_form}  induces an isomorphism of $\bZ^{\rm ur}_{p}[[\Gal(E_{\infty}/k_{v})]]$-modules
\[
\varprojlim_{n}(\bZ_{p}^{\rm ur} \otimes_{\bZ_{p}} \cO_{(k_{n}KL)_{w}}^{\times,\wedge}) \stackrel{\sim}{\longrightarrow} \bZ^{\rm ur}_{p}[[\Gal((k_{\Gamma}KL)_{w}/k_{v})]]. 
\]
We recall here that $k_{n}$ denotes the $n^{\rm th}$ layer of the $\bZ_{p}$-extension $k_{\Gamma}/k$. Using the fixed prime $w$ and on passing to $\chi$-parts, we obtain the isomorphism
\[
{\rm Col}_{\Gamma,K,v} \colon 
H^{1}_{f}(G_{k_{v}},T_{k_{\Gamma}K})
\stackrel{\sim}{\longrightarrow}
\bZ^{\rm ur}_{p}[[\Gamma]][\Gal(K/k)] 
\]
of $\bZ^{\rm ur}_{p}[[\Gamma]][\Gal(K/k)]$-modules. Here, $T_{k_{\Gamma}K} := T \otimes_{\bZ_{p}}\bZ_{p}[[\Gamma]][\Gal(K/k)]$. 
For an extension $K'/k$  contained in $K$, by construction, we see that the diagram 
\[
\xymatrix@C=50pt{
H^{1}_{f}(G_{k_{v}},T_{k_{\Gamma}K}) \ar[r]^-{{\rm Col}_{\Gamma,K,v}} \ar[d] &
\bZ^{\rm ur}_{p}[[\Gamma]][\Gal(K/k)] \ar[d] 
\\
H^{1}_{f}(G_{k_{v}},T_{k_{\Gamma}K'}) \ar[r]^-{{\rm Col}_{\Gamma,K',v}} & 
\bZ^{\rm ur}_{p}[[\Gamma]][\Gal(K'/k)]
}
\]
commutes. 
This in turn induces an isomorphism 
\begin{equation}
\label{eqn_define_Coleman_map_02_02_2022}
    {\rm Col}_{v} \colon H^{1}_{f}(G_{k_{v}},\bT_{\infty}) \stackrel{\sim}{\longrightarrow} \bZ^{\rm ur}_{p}[[\Gamma_{\infty}]] 
\end{equation}
of $\bZ^{\rm ur}_{p}[[\Gamma_{\infty}]]$-modules.

\begin{remark}
We note that, by \cite[Chapter 1, Proposition~3.9]{deshalit}, 
the isomorphism ${\rm Col}_{v}$ is independent of the choice of $f$, $\theta$, and $\zeta_{i}$. 
We also note that ${\rm Col}_{v}$ depends on the choice of the prime $w$ above $v$; compare with the discussion in Remark~\ref{rem:ord}. 
\end{remark}

The order map $(KL)_{w}^{\times} \twoheadrightarrow \bZ$ induces a surjective homomorphism 
\[
{\rm ord}_{v} \colon (KL \otimes_{k} k_{v})^{\times,\wedge} \twoheadrightarrow \bZ_{p}[\Gal(KL/k)/D_{KL,v}], 
\]
where $D_{KL,v}$ denotes the decomposition group at $v$ in $\Gal(KL/k)$. 
Furthermore, ${\rm ord}_{v}$ induces an isomorphism 
\begin{align}\label{exact1}
H^{1}(G_{k_{v}},T_{K})/H^{1}_{f}(G_{k_{v}},T_{K}) \stackrel{\sim}{\longrightarrow} e_{\chi}\bZ_{p}^{\rm ur}[\Gal(LK/k)/D_{KL,v}]
\end{align}
of $\bZ_{p}^{\rm ur}[\Gal(K/k)]$-modules. 
We therefore infer that 
\[
H^{1}_{f}(G_{k_{v}},\bT_{\infty}) = H^{1}(G_{k_{v}},\bT_{\infty})
\]
and we have 
\begin{align*}
{\rm Col}_{v} \colon H^{1}(G_{k_{v}},\bT_{\infty}) = H^{1}_{f}(G_{k_{v}},\bT_{\infty}) \stackrel{\sim}{\longrightarrow} \bZ^{\rm ur}_{p}[[\Gamma_{\infty}]]. 
\end{align*}

\subsection{``Constant terms'' of Coleman maps} 
\label{subsec_constant_terms_Coleman_maps} Our objective in the present subsection is to calculate the images of the multivariate Coleman maps we have introduced in \S\ref{eqn_define_Coleman_map_02_02_2022} under certain augmentation maps. We separately treat the cases when the place $v$ of $k$ splits completely in $L$ and when not. In the latter case, the main calculation is Proposition~\ref{prop:const1} and it is well-known to the experts. The former case is more involved and it occupies \S\ref{subsec:split}, where our main calculation is recorded as Proposition~\ref{prop:const2}.

Let $v \in S_{p}(k)$ and $K/k$ an abelian $p$-extension which is unramified at $v$. 
Since $L_{w}$ is an unramified extension of $\bQ_{p}$ thanks to our running assumptions on $\chi$, the $p$-adic logarithm defines a $\bZ_{p}$-isomorphism 
\[
\cO_{(KL)_{w}}^{\times, \wedge} \stackrel{\sim}{\longrightarrow} p\cO_{(KL)_{w}}\,;\,\,\,\, x \mapsto \sum_{n=1}^{\infty}(-1)^{n}\frac{(x-1)^{n}}{n}. 
\]
The fixed prime $w$ induces a $\bZ_{p}[\Gal(K/k)]$-homomorphism 
\[
 (\cO_{KL} \otimes_{\cO_{k}}\cO_{k_{v}})^{\chi} = \biggl( \bigoplus_{w' \mid v} \cO_{(KL)_{w'}}\biggr)^{\chi} \longrightarrow \cO_{k_{v}} \otimes_{\cO_{k}} \cO_{K}.  
\]
Thence, the logarithm and the fixed prime $w$  together induce a $\bZ_{p}^{\rm ur}[\Gal(K/k)]$-homomorphism 
\[
\log_{K,v} \colon H^{1}_{f}(G_{k_{v}},T_{K}) \longrightarrow p(\bZ_{p}^{\rm ur} \otimes_{\bZ_{p}} \cO_{k_{v}} \otimes_{\cO_{k}} \cO_{K}). 
\]
In particular, if $K=k$, since $k_{v} = \bQ_{p}$ and $p \nmid [L \colon k]$, we have a $\bZ_{p}^{\rm ur}$-isomorphism 
\begin{align}\label{def:log}
\log_{v} \colon H^{1}_{f}(G_{k_{v}},T) \stackrel{\sim}{\longrightarrow} p\bZ_{p}^{\rm ur}. 
\end{align}

\begin{remark}\label{rem:log}
The map $\log_{K,v}$  depends on the choice of the prime $w$ above $v$; compare with the discussion in Remark~\ref{rem:ord}. 
\end{remark}

\subsubsection{Non-Split Case}\label{subsub_non_split}

Suppose throughout \S\ref{subsub_non_split} that $v$ does not split completely in $L$ (namely that, $\chi(G_{k_{v}}) \neq 1$). 

\begin{lemma}\label{lem:str}
Let $K$ be a subfield of $k(p^{\infty})/k$. 
\begin{itemize}
\item[(i)] We have 
$H^{1}_{f}(G_{k_{v}},T_{K}) = H^{1}(G_{k_{v}},T_{K})$. 
\item[(ii)] The norm map 
\[
H^{1}(G_{k_{v}},\bT_{\infty}) \otimes_{\bZ_{p}^{\rm ur}[[\Gamma_{\infty}]]} \bZ_{p}^{\rm ur}[[\Gal(K/k)]]\longrightarrow H^{1}(G_{k_{v}},T_{K})
\]
is an isomorphism. 
\end{itemize}
\end{lemma}
\begin{proof}
Since  $\chi(G_{k_{v}}) \neq 1$, (i) follows from the isomorphism~(\ref{exact1}). 
Still based on the fact that $\chi(G_{k_{v}}) \neq 1$, we see that $H^{2}(G_{k_{v}}, T)=0$ by local duality, and (ii) follows from Corollary~\ref{cor:H2vanish}. 
\end{proof}


\begin{definition}\label{def:col1}
Let $k_{\Gamma}/k$ be a $\bZ_{p}$-extension {(which is not necessarily ramified)}. It follows from {Lemma~\ref{lem:str}} that we have a unique isomorphism 
\[
{\rm Col}_{\Gamma,v} \colon H^{1}(G_{k_{v}},\bT_{\Gamma}) \stackrel{\sim}{\longrightarrow} \bZ_{p}^{\rm ur}[[\Gamma]]
\] 
given by the commutativity of the following diagram 
\[
\xymatrix@C=60pt{
H^{1}(G_{k_{v}},\bT_{\infty}) \ar@{->>}[d] \ar[r]^-{{\rm Col}_{v}} & \bZ^{\rm ur}_{p}[[\Gamma_{\infty}]] \ar@{->>}[d]
\\
H^{1}(G_{k_{v}},\bT_{\Gamma}) \ar[r]^-{{\rm Col}_{\Gamma,v}} & \bZ_{p}^{\rm ur}[[\Gamma]]
}
\]
Notice that if $k_{\Gamma}/k$ is a ramified $\bZ_{p}$-extension, then  
${\rm Col}_{\Gamma,v} = {\rm Col}_{\Gamma,k,v}$ {(where the latter is given as in Section~\ref{sec:coleman constr} above)}. 
\end{definition}

Note in addition that, we have by Lemma~\ref{lem:str}(i) the following logarithm map:
\[
\log_{v} \colon H^{1}(G_{k_{v}},T) = H^{1}_{f}(G_{k_{v}},T) \stackrel{\sim}{\longrightarrow} p\bZ_{p}. 
\]

\begin{proposition}\label{prop:const1}
The diagram 
\[
\xymatrix@C=100pt{
H^{1}(G_{k_{v}},\bT_{\Gamma}) \ar@{->>}[d] \ar[r]^-{{\rm Col}_{\Gamma,v}} & \bZ^{\rm ur}_{p}[[\Gamma]] \ar@{->>}[d]
\\
H^{1}(G_{k_{v}},T) \ar[r]^-{\left(1-\frac{\chi(v)}{p}\right)(1-\chi(v)^{-1})^{-1} {\rm log}_{v}} & \bZ^{\rm ur}_{p}
}
\]
commutes. Here $\chi(v)$ denotes  the value of $\chi$ at the Frobenius element at $v$. 
\end{proposition}
\begin{proof}
This proposition follows from Proposition~\ref{prop:coleman map} and Lemma~\ref{lem:const}. 
\end{proof}

\subsubsection{Split Case}\label{subsec:split}

Suppose throughout \S\ref{subsec:split} that $v$ splits completely in $L$. 
Let $k_{\Gamma}/k$ be a $\bZ_{p}$-extension with Galois group $\Gamma$ and 
let $k_{n}$ denote the $n^{\rm th}$ layer of $k_{\Gamma}/k$. 
Set $L_{n} := k_{n}L$ and $L_{\Gamma} := k_{\Gamma}L$. 
We denote by $\Gamma_{v}\subset \Gamma$ the decomposition group at $v$. 
Assume that $v$ does not split completely in $k_{\Gamma}/k$, so that $\Gamma/\Gamma_{v}$ is a finite group. 

\begin{lemma}\label{lem:free1}
The $\bZ_{p}^{\rm ur}[[\Gamma]]$-module 
$H^{1}(G_{k_{v}},\bT_{\Gamma})$ is free of rank $1$.  
\end{lemma}
\begin{proof}
Note that $H^{2}(G_{k_{v}},\bT_{\Gamma})$ is a torsion $\bZ_{p}^{\rm ur}[[\Gamma]]$-module since $v$ does not split completely in $k_{\Gamma}/k$. 
Our lemma therefore follows from Corollary~~\ref{cor:free}.  
\end{proof}

Put $\cA_{\Gamma_{v}} := {\rm Ann}_{\bZ_{p}^{\rm ur}[[\Gamma]]}(\bZ_{p}^{\rm ur}[\Gamma/\Gamma_{v}])$.

\begin{proposition}\label{prop:image}
${\rm im}\left(H^{1}(G_{k_{v}},\bT_{\infty})  \longrightarrow 
H^{1}(G_{k_{v}},\bT_{\Gamma})\right)=\cA_{\Gamma_{v}} \cdot H^{1}(G_{k_{v}},\bT_{\Gamma})$. 
\end{proposition}
\begin{proof}
By Corollary~\ref{cor:amp}, there is an $\bZ_{p}^{\rm ur}[[\Gamma_{\infty}]]$-homomorphism 
$P^{1} \to P^{2}$ of finitely generated free $\bZ_{p}^{\rm ur}[[\Gamma_{\infty}]]$-modules such that 
\[
{\bf R}\Gamma(G_{k_{v}},\bT_{\infty}) = [\ \cdots\longrightarrow 0 \longrightarrow P^{1} \longrightarrow P^{2} \longrightarrow \cdots \ ]. 
\]
Set $I := {\rm im}(P^{1} \to P^{2})$ and 
\[
I_{\Gamma} := {\rm im}(P^{1} \otimes_{\bZ_{p}^{\rm ur}[[\Gamma_{\infty}]]} \bZ_{p}^{\rm ur}[[\Gamma]] \longrightarrow P^{2} \otimes_{\bZ_{p}^{\rm ur}[[\Gamma_{\infty}]]} \bZ_{p}^{\rm ur}[[\Gamma]]).
\] 
Since ${\bf R}\Gamma(G_{k_{v}},\bT_{\infty}) \otimes_{\bZ_{p}^{\rm ur}[[\Gamma_{\infty}]]}^{\bL} 
\bZ_{p}^{\rm ur}[[\Gamma]] \cong {\bf R}\Gamma(G_{k_{v}},\bT_{\Gamma})$ by Lemma~\ref{lem:descent}, the snake lemma shows that there is a natural isomorphism 
\begin{align*}
{\rm coker}\left(H^{1}(G_{k_{v}},\bT_{\infty})  \longrightarrow H^{1}(G_{k_{v}},\bT_{\Gamma})\right) &\cong \ker\left(I \otimes_{\bZ_{p}^{\rm ur}[[\Gamma_{\infty}]]} \bZ_{p}^{\rm ur}[[\Gamma]] \longrightarrow I_{\Gamma}\right) 
\\
&= \ker\left(I \otimes_{\bZ_{p}^{\rm ur}[[\Gamma_{\infty}]]} \bZ_{p}^{\rm ur}[[\Gamma]] \longrightarrow P^{2} \otimes_{\bZ_{p}^{\rm ur}[[\Gamma_{\infty}]]} \bZ_{p}^{\rm ur}[[\Gamma]]\right) 
\\
&\cong {\rm Tor}^{\bZ_{p}^{\rm ur}[[\Gamma_{\infty}]]}_{1}(\bZ_{p}^{\rm ur}[[\Gamma]],H^{2}(G_{k_{v}},\bT_{\infty})). 
\end{align*}
Since $v$ splits completely in $L$, we have an isomorphism 
\[
H^{2}(G_{k_{v}},\bT_{\infty}) \cong \bZ_{p}^{\rm ur}[[\Gamma_{\infty}/\Gamma_{\infty,v}]] 
\]
by local duality. Here, $\Gamma_{\infty,v} \subseteq \Gamma_{\infty}$ denotes the decomposition group at $v$. 
We claim that there is a canonical isomorphism 
\begin{equation}
\label{eqn_claim_in_prop_3_11}
{\rm Tor}^{\bZ_{p}^{\rm ur}[[\Gamma_{\infty}]]}_{1}(\bZ_{p}^{\rm ur}[[\Gamma]], \bZ_{p}^{\rm ur}[[\Gamma_{\infty}/\Gamma_{\infty,v}]]) \stackrel{?}{\cong} \bZ_{p}^{\rm ur}[\Gamma/\Gamma_{v}]. 
\end{equation}
Since $H^{1}(G_{k_{v}},\bT_{\Gamma})$ is free of rank $1$, 
the claimed isomorphism \eqref{eqn_claim_in_prop_3_11} concludes the proof of Proposition~\ref{prop:image}.

We now verify \eqref{eqn_claim_in_prop_3_11}. First, we note that $\Gamma_{\infty,v} \cong \bZ_{p}^{2}$ since $k_{v} = \bQ_{p}$. 
As the canonical map $\Gamma_{\infty,v} \longrightarrow \Gamma_{v}$ is surjective and $\Gamma_{v} \cong \bZ_{p}$, there is a regular sequence $\{x,y\}\subset\bZ_{p}^{\rm ur}[[\Gamma_{\infty}]]$ such that 
\begin{itemize}
\item $\bZ_{p}^{\rm ur}[[\Gamma_{\infty}]]/(x,y) \cong \bZ_{p}^{\rm ur}[[\Gamma_{\infty}/\Gamma_{\infty,v}]]$,
\item the image of $y$ in $\bZ_{p}^{\rm ur}[[\Gamma]]$ is zero,  
\item $\bZ_{p}^{\rm ur}[[\Gamma]]/x \bZ_{p}^{\rm ur}[[\Gamma]] \cong \bZ_{p}^{\rm ur}[\Gamma_{\infty}/\Gamma_{v}]$.
\end{itemize}
Since $y$ maps to zero in $\bZ_{p}^{\rm ur}[[\Gamma]]$, 
on applying $-\otimes_{\bZ_{p}^{\rm ur}[[\Gamma_{\infty}]]} \bZ_{p}^{\rm ur}[[\Gamma]]$  on the exact sequence 
\[
0 \longrightarrow \bZ_{p}^{\rm ur}[[\Gamma_{\infty}]] \xrightarrow{y} \bZ_{p}^{\rm ur}[[\Gamma_{\infty}]] \longrightarrow 
\bZ_{p}^{\rm ur}[[\Gamma_{\infty}]]/(y) \longrightarrow 0, 
\]
we have 
\[
{\rm Tor}^{\bZ_{p}^{\rm ur}[[\Gamma_{\infty}]]}_{1}(\bZ_{p}^{\rm ur}[[\Gamma]], \bZ_{p}^{\rm ur}[[\Gamma_{\infty}]]/(y)) = \bZ_{p}^{\rm ur}[[\Gamma]]. 
\]
Since $x$ is a regular element in $\bZ_{p}^{\rm ur}[[\Gamma]]$, again on applying $-\otimes_{\bZ_{p}^{\rm ur}[[\Gamma_{\infty}]]} \bZ_{p}^{\rm ur}[[\Gamma]]$  on the exact sequence 
\[
0 \longrightarrow \bZ_{p}^{\rm ur}[[\Gamma_{\infty}]]/(y) \xrightarrow{x} \bZ_{p}^{\rm ur}[[\Gamma_{\infty}]]/(y) \longrightarrow 
\bZ_{p}^{\rm ur}[[\Gamma_{\infty}/\Gamma_{\infty,v}]] \longrightarrow 0, 
\]
the equality ${\rm Tor}^{\bZ_{p}^{\rm ur}[[\Gamma_{\infty}]]}_{1}(\bZ_{p}^{\rm ur}[[\Gamma]], \bZ_{p}^{\rm ur}[[\Gamma_{\infty}]]/(y)) = \bZ_{p}^{\rm ur}[[\Gamma]]$ shows that 
\[
{\rm Tor}^{\bZ_{p}^{\rm ur}[[\Gamma_{\infty}]]}_{1}(\bZ_{p}^{\rm ur}[[\Gamma]], \bZ_{p}^{\rm ur}[[\Gamma_{\infty}/\Gamma_{\infty,v}]]) =  {\rm coker}\left(\bZ_{p}^{\rm ur}[[\Gamma]] \xrightarrow{x} \bZ_{p}^{\rm ur}[[\Gamma]] \right)
\cong \bZ_{p}^{\rm ur}[\Gamma/\Gamma_{v}],
\]
as required.
\end{proof}

Recall that $\cA_{\Gamma} = \ker\left(\bZ_{p}^{\rm ur}[[\Gamma]] \longrightarrow \bZ_{p}^{\rm ur}\right)$. 
\begin{definition}\label{def:col2}
\item[(i)] We let
\[
{\rm Col}_{\Gamma,v} \colon \cA_{\Gamma_{v}} \otimes_{\bZ_{p}^{\rm ur}[[\Gamma]]} H^{1}(G_{k_{v}},\bT_{\Gamma}) \stackrel{\sim}{\longrightarrow} \bZ_{p}^{\rm ur}[[\Gamma]]\,;\,\,\,\, \alpha \otimes x \mapsto {\rm Col}_{v}(\widetilde{\alpha x})\vert_{\Gamma}, 
\]
denote the isomorphism of   $\bZ_{p}^{\rm ur}[[\Gamma]]$-modules induced by the Coleman map ${\rm Col}_{v}$, using Proposition~\ref{prop:image}. Here, $\widetilde{\alpha x}\in H^{1}(G_{k_{v}},\bT_{\infty})$ is a lift of $\alpha x$ and 
${\rm Col}_{v}(\widetilde{\alpha x})\vert_{\Gamma}$ denotes the image of ${\rm Col}_{v}(\widetilde{\alpha x})$ in $\bZ_{p}^{\rm ur}[[\Gamma]]$. 
\item[(ii)] We set 
\[
\cU_{\Gamma,v}^{\chi} := {\rm im}\left(H^{1}(G_{k_{v}},\bT_{\Gamma})  \longrightarrow H^{1}(G_{k_{v}},T)\right). 
\]
\item[(iii)] 
Let $\widetilde{u}\in H^{1}(G_{k_{v}},\bT_{\Gamma})$ be any lift of $u\in \cU_{\Gamma,v}^{\chi}$. We let
\begin{align*}
\widetilde{\rm Col}_{\Gamma, v} \colon \cA_{\Gamma_{v}}/\cA_{\Gamma}\cA_{\Gamma_{v}} \otimes_{\bZ_{p}^{\rm ur}} \cU_{\Gamma,v}^{\chi} &\stackrel{\sim}{\longrightarrow} \bZ_{p}^{\rm ur}\,;\,\,\, \,
(\gamma-1) \otimes u \mapsto 
{\rm Col}_{\Gamma,v}((\gamma-1)\widetilde{u}) \bmod \cA_{\Gamma}, 
\end{align*}
denote the isomorphism induced by the map ${\rm Col}_{\Gamma,v}$.
\end{definition}

Recall the map
\[
{\rm rec}_{\Gamma,v} \colon H^{1}(G_{k_{v}},T) \longrightarrow \cA_{\Gamma}/\cA_{\Gamma}^{2}
\] 
induced from the reciprocity map
\[
{\rm rec}_{w} \colon L_{w}^{\times} \longrightarrow \Gal(L_{w}^{\rm ab}/L_{w}) \twoheadrightarrow \Gamma_{v} \hookrightarrow \Gamma. 
\]
Since the image of $\bZ_{p}^{\rm ur} \otimes_{\bZ_{p}}\Gamma_{v} \hookrightarrow \bZ_{p}^{\rm ur} \otimes_{\bZ_{p}}\Gamma \stackrel{\sim}{\longrightarrow} \cA_{\Gamma}/\cA_{\Gamma}^{2}$ is canonically isomorphic to $\cA_{\Gamma_{v}}/\cA_{\Gamma}\cA_{\Gamma_{v}}$, 
we have a surjection 
\[
{\rm rec}_{\Gamma,v} \colon H^{1}(G_{k_{v}},T) \twoheadrightarrow  \cA_{\Gamma_{v}}/\cA_{\Gamma}\cA_{\Gamma_{v}}. 
\]
Local class field theory shows
\[
\cU_{\Gamma,v}^{\chi} = \ker\left( {\rm rec}_{\Gamma,v} \colon H^{1}(G_{k_{v}},T) \twoheadrightarrow   \cA_{\Gamma_{v}}/\cA_{\Gamma}\cA_{\Gamma_{v}} \right). 
\]
We therefore have an isomorphism
\[
{\rm rec}_{\Gamma,v} \wedge \widetilde{\rm Col}_{\Gamma, v} \colon {{\bigcap}}^{2}_{\bZ_{p}^{\rm ur}} 
H^{1}(G_{k_{v}},T) \stackrel{\sim}{\longrightarrow} \bZ_{p}^{\rm ur}. 
\]
Recall the ``$v$-adic valuation'' map 
\[
{\rm ord}_{v} \colon H^{1}(G_{k_{v}},T) \twoheadrightarrow \bZ_{p}^{\rm ur}
\]
and note that $\ker({\rm ord}_{v})=H^{1}_{f}(G_{k_{v}},T)$.  
\begin{proposition}\label{prop:const2}
The map 
\[
{\rm rec}_{\Gamma, v} \wedge \widetilde{\rm Col}_{\Gamma, v} \colon {{\bigcap}}^{2}_{\bZ_{p}^{\rm ur}} 
H^{1}(G_{k_{v}},T)  \stackrel{\sim}{\longrightarrow} \bZ_{p}^{\rm ur}
\]
is $\left(1-1/p\right) {\rm ord}_{v} \wedge \log_{v}$. 
\end{proposition}

\begin{proof}
Let us first suppose that $k_{\Gamma}/k$ is unramified at $v$. 
We take elements $\pi, u$ of $H^{1}(G_{k_{v}},T)$ such that ${\rm ord}_{v}(\pi) = 1$ and ${\rm ord}_{v}(u) = 0$, so that we have
\[
{\rm ord}_{v} \wedge \log_{v}(\pi \wedge u) = \log_{v}(u)
\]
by definition; it is also helpful to notice that we have ${\bigwedge}^{2}_{\bZ_{p}^{\rm ur}} 
H^{1}(G_{k_{v}},T) = {\bigcap}^{2}_{\bZ_{p}^{\rm ur}} 
H^{1}(G_{k_{v}},T)$. Since $k_{\Gamma}/k$ is unramified at $v$, we conclude by local class field theory and the isomorphism~(\ref{exact1}) that
\[
H^{1}_{f}(G_{k_{v}},\bT_{\Gamma}) = H^{1}(G_{k_{v}},\bT_{\Gamma}) \ \text{ and } \ 
\cU_{\Gamma,v}^{\chi} = H^{1}_{f}(G_{k_{v}},T).  
\]
Since $k_{\Gamma}/k$ is unramified at $v$, we have 
\[
{\rm rec}_{\Gamma,v}(\pi \wedge u) = (\varphi-1) \otimes u = (1 - \varphi^{-1}) \otimes u 
\]
in $\cA_{\Gamma_{v}}/\cA_{\Gamma}\cA_{\Gamma_{v}} \otimes_{\bZ_{p}^{\rm ur}} H^{1}_{f}(G_{k_{v}},T)$. 
Here, $\varphi \in \Gal(k_{\Gamma}/k)$ stands for the Frobenius element at $v$. 
Let $\widetilde{u} \in H^{1}(G_{k_{v}},\bT_{\Gamma})$ be a lift of $u$ and 
$\beta \in H^{1}(G_{k_{v}},\bT_{\infty})$ a lift of $(1-\varphi^{-1})\widetilde{u}$. 
We have by construction
\[
\rec_{\Gamma,v} \wedge \widetilde{\rm Col}_{\Gamma, v}(\pi \wedge u) = 
{\rm Col}_{v}(\beta) \bmod \cA_{\Gamma_{\infty}}. 
\]
Since $k_{\Gamma}/k$ is unramified at $v$, one may use the $p$-adic logarithm to define the map 
\[
{\rm Log}_{v} \colon H^{1}(G_{k_{v}},\bT_{\Gamma}) \longrightarrow \bZ_{p}^{\rm ur}[[\Gamma]]. 
\]
Since $\beta$ is a lift of $(1-\varphi^{-1})\widetilde{u}$, we have 
\[
(1-\varphi^{-1}){\rm Col}_{v}(\beta)\vert_{\Gamma} = \left(1-\frac{\varphi}{p}\right)(1-\varphi^{-1}){\rm Log}_{w}(\widetilde{u})
\]
by the construction of ${\rm Col}_{w}$ and Lemma~\ref{lem:const}. 
Since $1-\varphi^{-1}$ is a regular element of $\bZ_{p}^{\rm ur}[[\Gamma]]$, we conclude that 
\[
{\rm Col}_{v}(\beta)\vert_{\Gamma} = \left(1-\frac{\varphi}{p}\right){\rm Log}_{v}(\widetilde{u}), 
\]
which in turn shows that 
\[
{\rm rec}_{\Gamma,v} \wedge \widetilde{\rm Col}_{\Gamma, v}(\pi \wedge u) = 
\left(1-\frac{1}{p}\right){\rm log}_{v}(u) = \left(1-\frac{1}{p}\right) {\rm ord}_{v} \wedge \log_{v}(\pi \wedge u)
\]
and concludes the proof when  $k_{\Gamma}/k$ is unramified at $v$. 

Next, we assume that $k_{\Gamma}/k$ is ramified at $v$. 
Let $\{\pi\} \subset \cU_{\Gamma,v}^{\chi}$ be a basis and set $q := [k_{\Gamma,w} \cap k^{\rm ur}_{v} \colon k_{v}]$. 
Since ${\rm ord}_{v}(\cU_{\Gamma,v}^{\chi}) = q\bZ_{p}^{\rm ur}$,  there is an element $u \in H^{1}_{f}(G_{k_{v}},T)$ such that 
\[
{\rm ord}_{v} \wedge \log_{v}(\pi \wedge u) = q\log_{v}(u). 
\]
On the other hand, since ${\rm rec}_{\Gamma,v}(\cU_{\Gamma,v}^{\chi}) = 0$, we have 
\[
{\rm rec}_{\Gamma,v}(\pi \wedge u) = -({\rm rec}_{v}(u)-1) \otimes \pi = 
({\rm rec}_{v}(u^{-1})-1) \otimes \pi. 
\]
Let $\widetilde{\pi} \in H^{1}(G_{k_{v}},\bT_{\Gamma})$ be a lift of $\pi$ and 
put $\beta := ({\rm rec}_{v}(u^{-1})-1)\widetilde{\pi} \in H^{1}_{f}(G_{k_{v}},\bT_{\Gamma})$. 
Using Lemma~\ref{lem:const1} and Remark~\ref{rem:exten}, we compute 
\begin{align*}
{\rm rec}_{\Gamma,v} \wedge \widetilde{\rm Col}_{\Gamma, v}(\pi \wedge u) &= {\rm Col}_{\Gamma,k,v}(\beta) \bmod \cA_{\Gamma}
\\
&= \left(1-\frac{1}{p}\right) q {\rm log}_{v}(u)
\\
&= \left(1-\frac{1}{p}\right) {\rm ord}_{v} \wedge \log_{v}(\pi \wedge u).
\end{align*}
\end{proof}

\section{Non-critical exceptional zeros of the Katz $p$-adic $L$-function}
\label{section_main_conjANDExceptionalZeros}
Our goal in this section is to recall the definition of the Katz' $p$-adic $L$-function and investigate its exceptional zeros at the trivial character (which lies outside its range of interpolation). We will also formulate the Explicit Reciprocity Conjecture~\ref{conj:rec} for Rubin--Stark elements, which is akin to that formulated in \cite[Conjecture 4.18]{buyukbodukleiPLMS} (see also \cite{BBS21} for a refinement of \cite[Conjecture 4.18]{buyukbodukleiPLMS} in certain aspects).

\emph{Until the end of this article,  the following $p$-ordinary hypothesis of Katz will be in effect:}
\\
\eqref{item_ord} \qquad \qquad Every prime of $k^{+}$ above $p$ splits in $k$. 

Let $c \in \Gal(k/k^{+})$ denote the generator. 
Then we can take a subset $\Sigma \subseteq S_{p}(k)$ such that 
$\Sigma \cup \Sigma^{c} = S_{p}(k)$ and $\Sigma \cap \Sigma^{c} = \emptyset$. 
Until the end, we fix such a subset $\Sigma\subset S_{p}(k)$. 

Recall that $k(p^{\infty})$ denotes the compositum of all $\bZ_{p}$-extensions of $k$ and $\Gamma_\infty=\Gal(k(p^{\infty})/k)$.  Thanks to \eqref{item_ord}, we have Katz' $p$-adic $L$-function (see \cite{Katz78} and \cite[Theorem~II]{HT93})
\[
L_{p,\Sigma}^{\chi} \in \bZ_{p}^{\rm ur}[[\Gamma_{\infty}]]
\] 
that interpolates $p$-adically the algebraic part of critical $\Sigma^{c}$-truncated Hecke $L$-values for $\chi^{-1}$ twisted by characters of $\Gamma_{\infty}$. {In more precise terms, it is characterized by the following interpolative property: For the $p$-adic avatar $\widehat \Xi$ of each Hecke character  $\Xi$ of infinity type $m\Sigma_\infty$ with $m>4$, we have 
\begin{equation}
\label{Eqn_Katz_padic_interpolation}
\frac{\widehat \Xi (L_{p,\Sigma}^\chi)}{\Omega_p^{m\Sigma_\infty}} = 
 \frac{\Gamma(m)^g}{\sqrt{|D_{k^+}|_{\mathbb{R}}}}\cdot 
\prod_{v\in \Sigma} \dfrac{1-\frac{\psi_v^{-1}(\varpi_v)}{\mathbb{N}v}}{\epsilon(0,\psi_v)}\cdot
\prod_{v\in \Sigma^c}(1-\psi_v(\varpi_v))\cdot 
\frac{L(\psi,0)}{\Omega_\infty^{m\Sigma_\infty}}
\end{equation}}
where \begin{itemize}
\item $\Sigma_\infty:=\{j^{-1}\circ\iota_v: k\hookrightarrow \mathbb{C}\mid  v\in \Sigma\}$ is the CM type corresponding to $\Sigma$,
\item $D_{k^+}$ is the absolute discriminant of $k^+$,
\item $\psi$ is the Hecke character $\chi^{-1}{\Xi} $ and $\psi_v$ its $v$-component,
\item $\varpi_v\in k_v$ is a uniformizer, 
\item $\epsilon(s,\psi_v)$ is Tate's local epsilon factor,
\item $\Omega_p$ and $\Omega_\infty$ are the $p$-adic and archimedean CM periods (which we shall not define here).
\end{itemize}

\begin{remark}
We note that the interpolation formula \eqref{Eqn_Katz_padic_interpolation} differs from that in \cite[Theorem~II]{HT93} with a factor of $[\cO_k^\times:\cO_{k^+}^\times]$. This alteration is to ensure the coherence with the corresponding formulae in \cite{deshalit} in the special case when $k$ is imaginary quadratic. With this modification, the Reciprocity Conjecture~\ref{conj:rec} below reduces to a theorem of Yager (in the form presented in \cite[Equation (49)]{deshalit}) in this special case.
\end{remark}

\begin{definition}
Set $e := \#\{v \in \Sigma^{c} \mid \chi(G_{k_{v}}) = 1\}$. 
\end{definition}
\begin{remark}\label{remark_Katz_interpolation_trivial_char}
\item[i)] The trivial character $\mathds{1}$ is not in the range of interpolation for Katz' $p$-adic $L$-function.
\item[ii)] When $\Xi =\mathds{1}$, precisely $e$ of the factors $\{1-\psi_v(\varpi_v): {v\in \Sigma^c}\}$ that appear in \eqref{Eqn_Katz_padic_interpolation} vanish. 
\end{remark}
One is led to the following conjecture, based on Remark~\ref{remark_Katz_interpolation_trivial_char}(ii):
\begin{conjecture}[Weak Exceptional Zero Conjecture]\label{conj:ex zero}
$L_{p,\Sigma}^{\chi} \in \cA_{\Gamma_{\infty}}^{e}.$
\end{conjecture}


We have the following evidence in favor of this conjecture. It is proved in Appendix~\ref{appendix_main_conjecture} as Proposition~\ref{prop_app_IMC_reformulated}:
\begin{proposition}\label{imp2}
If the Iwasawa Main Conjecture~\ref{conj:iwasawa} holds true, then so does the Weak Exceptional Zero Conjecture~\ref{conj:ex zero}. 
\end{proposition}

In view of Theorem~\ref{HsiehTheorem}(iii) of Hsieh, one has the following:
\begin{corollary}[Hsieh]\label{corollary_imp2}
In addition to our running hypotheses, suppose $p>3$ is coprime to $h_{k}^{-}$. Assume $k=k^+M$ where $M$ is an imaginary quadratic field in which $p$ splits, and that $\Sigma_p$ is obtained by extending $\iota_p:M\hookrightarrow \mathbb{C}_p$, and finally that $L$ is abelian over $M$ and $p\nmid [L:M]$. 
Then the Weak Exceptional Zero Conjecture~\ref{conj:ex zero} holds true. 
\end{corollary}
\begin{remark}
When $k$ is an imaginary quadratic field, a proof of Conjecture~\ref{conj:ex zero} is already available in \cite[Theorem 1.5.7]{HT94}. Their line of argument (which is based on a $p$-adic Kronecker limit formula) is different from the one presented in this article.
\end{remark}

\subsection{Explicit Reciprocity Conjecture for Rubin--Stark elements}
Fix an ordering of the places in $S = S_{\infty}(k) \cup S_{p}(k) \cup S_{\rm ram}(L/k)$ as in Section~\ref{subsec:Rubin--Stark}. 
We have a Rubin--Stark element 
\[
\epsilon^{S_{\infty}(k), \chi}_{KL/k, S} \in {\bigwedge}^{g}_{\bC_{p}[\Gal(K/k)]}(\bC_{p}\cO_{KL,S}^{\times})^{\chi} 
\]
for any finite extension $K/k$ with $S_{\rm ram}(K/k) \subseteq S$. 
Assume the truth of the Rubin--Stark conjecture for all  finite subextensions $Lk(p^\infty)/k$. Using Proposition~\ref{prop:stark-reln} and Corollary~\ref{cor:inv-bidual}, we define
\[
\epsilon^{S_{\infty}(k), \chi}_{L(p^{\infty})/k, S} := \varprojlim_{K \subseteq k(p^{\infty})}
\epsilon^{S_{\infty}(k), \chi}_{KL/k, S} \in 
{\bigcap}^{g}_{\bZ_{p}^{\rm ur}[[\Gamma_{\infty}]]}H^{1}(G_{k,S},\bT_{\infty}). 
\]

Suppose that $p$ splits completely in $k$. 
As explained in Section~\ref{section_Coleman_map}, we have a 
Coleman map ${\rm Col}_{v} \colon H^{1}(G_{k_{v}},\bT_{\infty}) \stackrel{\sim}{\longrightarrow} \bZ_{p}^{\rm ur}[[\Gamma_{\infty}]]$ for each place $v$ of $k$ above $p$. 
Put 
\[
H^{1}_{\Sigma}(\bT_{\infty}) := \bigoplus_{v \in \Sigma}H^{1}(G_{k_{v}},\bT_{\infty}). 
\]
The Coleman maps $\{{\rm Col}_{v}\}_{v\in \Sigma}$ in turn induce an isomorphism  
\[
\bigwedge_{v \in \Sigma}{\rm Col}_{v} \colon 
{\bigcap}^{g}_{\bZ_{p}^{\rm ur}[[\Gamma_{\infty}]]}
H^{1}_{\Sigma}(\bT_{\infty}) \stackrel{\sim}{\longrightarrow} \bZ_{p}^{\rm ur}[[\Gamma_{\infty}]]. 
\]
Let 
\[
{\rm loc}_{\Sigma} \colon {\bigcap}^{g}_{\bZ_{p}^{\rm ur}[[\Gamma_{\infty}]]}H^{1}(G_{k,S},\bT_{\infty}) \longrightarrow
{\bigcap}^{g}_{\bZ_{p}^{\rm ur}[[\Gamma_{\infty}]]}H^{1}_{\Sigma}(\bT_{\infty})
\]
denote the obvious map induced by the localization maps at primes above $p$. 

\begin{conjecture}[Explicit Reciprocity Conjecture for Rubin--Stark elements]\label{conj:rec}
\[
\bigwedge_{v \in \Sigma}{\rm Col}_{v}({\rm loc}_{\Sigma}(\epsilon^{S_{\infty}(k), \chi}_{L(p^{\infty})/k, S})) = L_{p,\Sigma}^{\chi,\iota}. 
\]
\end{conjecture}

\begin{remark}\label{rem:rec1} When $k$ is an imaginary quadratic field, 
the Rubin--Stark units are constructed from elliptic units  (c.f. Section~\ref{subsubsec_new_expand_on_remark_1_7_new_remark_1_3} below for an elaboration on this point). In this particular scenario, Conjecture~\ref{conj:rec} is a theorem of Yager~\cite[Theorem 15]{Yag82};  see also Proposition in \cite[Chapter III, \S1.4]{deshalit} and Equation (49) in \cite[page 86]{deshalit}. 
\end{remark}

Let $k_{\Gamma}/k$ be a $\bZ_{p}$-extension with Galois group $\Gamma$. 
We set $V := \{v \in S \mid \chi(G_{k_{v}})=1\}$ and 
\begin{align*}
H^{1}_{\Sigma}(\bT_{\Gamma}) := \bigoplus_{v \in \Sigma}H^{1}(G_{k_{v}},\bT_{\Gamma}). 
\end{align*}
We also denote by  
\begin{align*}
{\rm loc}_{\Sigma} \colon {\bigcap}^{g}_{\bZ_{p}^{\rm ur}[[\Gamma]]}H^{1}(G_{k,S},\bT_{\Gamma}) \longrightarrow
{\bigcap}^{g}_{\bZ_{p}^{\rm ur}[[\Gamma]]}H^{1}_{\Sigma}(\bT_{\Gamma})
\end{align*}
the map induced by the localization maps. 
Suppose that no prime in $\Sigma$ splits completely in $k_{\Gamma}/k$. 
We then have the commutative diagram 
\begin{align}\label{diag1}
\xymatrix@C=1pt{
{\bigcap}_{\bZ_{p}^{\rm ur}[[\Gamma_{\infty}]]}^{g}H^{1}(G_{k,S},\bT_{\infty}) \ar[r]^{{\rm loc}_{\Sigma}}
 & {\bigcap}^{g}_{\bZ_{p}^{\rm ur}[[\Gamma_{\infty}]]}H^{1}_{\Sigma}(\bT_{\infty})\ar[d]\ar[rrrrrrr]^-{{\bigwedge}_{v \in \Sigma}{\rm Col}_{v}} &&&&&&& \bZ_{p}^{\rm ur}[[\Gamma_{\infty}]] \ar[d]
\\
& \bigotimes_{v \in V \cap \Sigma}\cA_{\Gamma_{v}} \otimes_{\bZ_{p}^{\rm ur}[[\Gamma]]}{\bigcap}^{g}_{\bZ_{p}^{\rm ur}[[\Gamma]]}H^{1}_{\Sigma}(\bT_{\Gamma})\ar[rrrrrrr]^-{{\bigwedge}_{v \in \Sigma}{\rm Col}_{\Gamma,v}} &&&&&&& \bZ_{p}^{\rm ur}[[\Gamma]]. 
}
\end{align}
Here, $\cA_{\Gamma_{v}} := {\rm Ann}_{\bZ_{p}^{\rm ur}[[\Gamma]]}(\bZ_{p}^{\rm ur}[\Gamma/\Gamma_{v}])$. 
Conjecture~\ref{conj:rec}  therefore implies that 
\begin{align}\label{eq:gamma}
\bigwedge_{v \in \Sigma}{\rm Col}_{\Gamma,v}({\rm loc}_{\Sigma}(\epsilon^{S_{\infty},\chi}_{L_{\Gamma}/k,S}) )= L_{p,\Sigma}^{\chi,\iota}\vert_{\Gamma} \,\,\,\,\,\,\,\,\,\,\,\,\,\,\,\,\,\,\,\,\,\,\,\, \hbox{(conditional on Conjecture~\ref{conj:rec})}. 
\end{align}

\begin{lemma}\label{lem:imp1}
The Iwasawa Main Conjecture~\ref{conj:iwasawa} and the Explicit Reciprocity Conjecture~\ref{conj:rec} together imply Exceptional Zero Conjecture~\ref{MRS1} for Rubin--Stark elements. 
\end{lemma}
\begin{proof}
This lemma follows from (\ref{eq:gamma}), (\ref{exact4}), and Corollary~\ref{Cor_App_IMC_reformulated}. 
\end{proof}

\section{$\cL$-invariant}
\label{Section_L_invariant}
In this section, we shall introduce the $\cL$-invariants associated to the Katz $p$-adic $L$-function $L_{p,\Sigma}^\chi$, along any $\bZ_p$-extension of $k$. As we see shall see in Section~\ref{sec_leading_term_Katz_03_02}, these $\cL$-invariants naturally appear in the leading term formulae for the Katz $p$-adic $L$-function at the trivial character (which lies outside the interpolation range of $L_{p,\Sigma}^\chi$) in the presence of exceptional zeros. We note that in \cite[Section 3.4]{BS_IMRN}, the authors proved that these $\cL$-invariants can be interpolated, as we vary the choice of the $\bZ_p$-extension of $k$, to a universal (multivariate) $\cL$-invariant.
 
In this section, we assume the truth of the $\Sigma$--Leopoldt Conjecture (for the field $L$), which reads as follows:

\begin{conjecture}[$\Sigma$--Leopoldt Conjecture]\label{conj:leop}
The canonical map 
\[
e_{\chi}(\bZ_{p} \otimes_{\bZ} \cO_{L}^{\times}) \longrightarrow \bigoplus_{v \in \Sigma}H^{1}_{f}(G_{k_{v}},T)
\]
is injective. 
\end{conjecture}
We refer to \cite[Lemma 1.2.1]{HT94} for a further discussion on this conjecture, where the authors explain how to deduce Conjecture~\ref{conj:leop} from the $p$-adic Schanuel conjecture.

Let $k_{\Gamma}$ be a $\bZ_{p}$-extension of $k$ with Galois group $\Gamma$. 
Recall that the logarithm and the fixed prime $w$ induce an isomorphism of $\bZ_{p}^{\rm ur}$-modules: 
\[
{\rm log}_{v} \colon H^{1}_{f}(G_{k_{v}},T) \stackrel{\sim}{\longrightarrow} p (\cO_{L} \otimes_{\cO_{k}} \cO_{k_{v}} \otimes_{\bZ_{p}} \bZ_{p}^{\rm ur})^{\chi} \stackrel{\sim}{\longrightarrow} p (\cO_{k_{v}} \otimes_{\bZ_{p}} \bZ_{p}^{\rm ur}).  
\]
Furthermore, on fixing a basis of $\bigoplus_{v \in \Sigma}\cO_{k_{v}} \otimes_{\bZ_{p}} \bZ_{p}^{\rm ur}$ we have 
\[
\bigwedge_{v \in \Sigma}\log_{v} \colon {\bigcap}^{g}_{\bZ_{p}^{\rm ur}}
H^{1}_{\Sigma^{c}}(G_{k,S},T) \longrightarrow {\bigcap}^{g}_{\bZ_{p}^{\rm ur}} \biggl(\bigoplus_{v \in \Sigma}
p(\cO_{k_{v}} \otimes_{\bZ_{p}} \bZ_{p}^{\rm ur})\biggr) \cong p^{g}\bZ_{p}^{\rm ur}, 
\]
(where the isomorphism at the end is induced via the fixed basis).

Let us set $S = S_{\infty}(k) \cup S_{p}(k) \cup S_{\rm ram}(L/k)$ and $V := \{v \in S \mid \chi(G_{k_{v}}) = 1\}$. 
Put $r:=\#V - g$. 
Recall that $\cA_{\Gamma} := \ker\left(\bZ_{p}^{\rm ur}[[\Gamma]] \longrightarrow \bZ_{p}^{\rm ur}\right)$ and 
that we have the maps
\[
{\rm ord}_{v} \colon H^{1}(G_{k_{v}},T) \twoheadrightarrow \bZ_{p}^{\rm ur},  
\]
\[
{\rm rec}_{\Gamma,v} \colon H^{1}(G_{k_{v}},T) \longrightarrow \cA_{\Gamma}/\cA_{\Gamma}^{2}. 
\] 
Note that both maps depend on the choice of the prime $w$ above $v$ 
(see Remarks~\ref{rem:ord} and \ref{rem:rec}). 
We therefore have an element 
\begin{equation}\label{eqn_this_element_1}
\bigwedge_{v \in V \setminus S_{\infty}(k)}{\rm ord}_{v} \wedge \bigwedge_{v \in \Sigma }{\rm log}_{v} 
\in {\bigwedge}^{g+r}_{\bZ_{p}^{\rm ur}}\Hom_{\bZ_{p}^{\rm ur}}(H^{1}(G_{k,S},T),\bZ_{p}^{\rm ur}). 
\end{equation}
We also note that 
\[
{\rm rank}_{\bZ_{p}^{\rm ur}}H^{1}(G_{k,S},T) = g+r
\]
and that $\Sigma$--Leopoldt Conjecture holds true if and only if the element given in \eqref{eqn_this_element_1} does not vanish. 

\begin{definition}\label{defn_L_invariant}
Assume the $\Sigma$--Leopoldt conjecture. Set $e := \# (\Sigma^{c} \cap V)$. 
We define the $\cL$-invariant $\cL_{\Sigma, \Gamma}$ along $\Gamma$ to be the element 
$\cL_{\Sigma,\Gamma} \in \bC_{p} \otimes_{\bZ_{p}^{\rm ur}} \cA_{\Gamma}^{e}/\cA_{\Gamma}^{e+1}$ which verifies
\[
\bigwedge_{v \in V \setminus S_{\infty}(k)}\psi_{v} \wedge \bigwedge_{v \in \Sigma}\log_{v} = \cL_{\Sigma, \Gamma} \otimes \bigwedge_{v \in V \setminus S_{\infty(k)}}{\rm ord}_{v} \wedge \bigwedge_{v \in \Sigma }{\rm log}_{v}\,.
\]
Here, the equality takes place in
\[
\bC_{p} \otimes_{\bZ_{p}^{\rm ur}} \cA_{\Gamma}^{e}/\cA_{\Gamma}^{e+1} 
\otimes_{\bZ_{p}^{\rm ur}} {\bigwedge}^{g+r}_{\bZ_{p}^{\rm ur}}\Hom_{\bZ_{p}^{\rm ur}}(H^{1}(G_{k,S},T),\bZ_{p}^{\rm ur})
\] 
and we have set 
\[
\psi_{v} := 
\begin{cases}
{\rm rec}_{\Gamma,v} & \text{ if } v \in V \cap \Sigma^{c}, 
\\
{\rm ord}_{v} & \text{ if } v \in V \cap \Sigma\,. 
\end{cases}
\]

\end{definition}

\begin{remark}\label{remark_L_invariant_properties}
\item[(i)] Note that $\cL_{\Sigma,\Gamma}$ depends only on $\Sigma$ and $\Gamma$ and is independent of the choices
\begin{itemize} 
\item of the prime $w$ above $v \in S_{p}(k)$ (see Remark~\ref{rem:ord} and Remark~\ref{rem:rec}) and, 
\item of a basis of the module $\bigoplus_{v \in \Sigma}\cO_{k_{v}} \otimes_{\bZ_{p}} \bZ_{p}^{\rm ur}$. 
\end{itemize} 
\item[(ii)]  If there is a prime $v \in \Sigma^{c}$ which splits completely in $k_{\Gamma}/k$, then ${\rm rec}_{\Gamma,v}$ is trivial and in turn, we have $\cL_{\Sigma, \Gamma} = 0$. 
\end{remark}

\begin{lemma}\label{lemma_X}
Assume the validity of $\Sigma$--Leopoldt conjecture. 
Then the rank of free $\bZ_{p}^{\rm ur}$-module 
\[
X := \ker\left(H^{1}(G_{k,S},T) 
\longrightarrow \bigoplus_{v \in\Sigma}H^{1}(G_{k_{v}},T)
\right)
\]
is $e = \# (V \cap \Sigma^{c})$.
Moreover, 
\[
\bigwedge_{v \in V \cap \Sigma^{c}}{\rm rec}_{\Gamma,v} = \cL_{\Sigma, \Gamma} \otimes \bigwedge_{v \in V \cap \Sigma^{c}}{\rm ord}_{v}
\]
where the equality takes place in 
\[
\bC_{p} \otimes_{\bZ_{p}^{\rm ur}} \cA_{\Gamma}^{e}/\cA_{\Gamma}^{e+1} 
\otimes_{\bZ_{p}^{\rm ur}} {\bigwedge}^{e}_{\bZ_{p}^{\rm ur}}\Hom_{\bZ_{p}^{\rm ur}}(X,\bZ_{p}^{\rm ur}).
\] 
\end{lemma}
\begin{proof}
First, we note that there is the commutative diagram with exact rows
\[
\xymatrix{
0 \ar[r] & e_{\chi}(\bZ_{p}^{\rm ur} \otimes_{\bZ} \cO_{L}^{\times}) \ar[r] \ar[d] & H^{1}(G_{k,S},T) \ar[r] \ar[d] & \bigoplus_{v \in S_{p}(k)}H^{1}_{/f}(G_{k_{v}},T) \ar[d]
\\
0 \ar[r] & \bigoplus_{v \in \Sigma}H^{1}_{f}(G_{k_v},T)
 \ar[r] & \bigoplus_{v \in \Sigma}H^{1}(G_{k_v},T) \ar[r]  & \bigoplus_{v \in \Sigma}H^{1}_{/f}(G_{k,S},T)\ar[r] & 0,  
}
\]
where $H^{1}_{/f}(G_{k_{v}},T) := H^{1}(G_{k_{v}},T)/H^{1}_{f}(G_{k_{v}},T)$. 
Since 
\[
{\rm rank}_{\bZ_{p}^{\rm ur}} e_{\chi}(\bZ_{p}^{\rm ur} \otimes_{\bZ} \cO_{L}^{\times}) = [k \colon \bQ] = \sum_{v \in \Sigma} {\rm rank}_{\bZ_{p}^{\rm ur}} H^{1}_{f}(G_{k_v},T), 
\]
the $\Sigma$--Leopoldt conjecture implies that 
\[
e_{\chi}(\bZ_{p}^{\rm ur} \otimes_{\bZ} \cO_{L}^{\times}) \otimes_{\bZ} \bQ \stackrel{\sim}{\longrightarrow} \bigoplus_{v \in \Sigma} H^{1}_{f}(G_{k_v},T) \otimes_{\bZ} \bQ. 
\]
Moreover, the cokernel of the map 
$$H^{1}(G_{k,S},T) \longrightarrow \bigoplus_{v \in S_{p}(k)}H^{1}_{/f}(G_{k_{v}},T)$$ 
is a torsion $\bZ_{p}^{\rm ur}$-module, and hence we obtain 
\[
{\rm rank}_{\bZ_{p}^{\rm ur}}X = \sum_{v \in \Sigma^{c}}{\rm rank}_{\bZ_{p}^{\rm ur}}H^{1}_{/f}(G_{k_{v}},T) = e. 
\]
As $H^{1}(G_{k_{v}},T)$ is a  $\bZ_{p}^{\rm ur}$-free module for each prime $v \in \Sigma$, one may take a $\bZ_{p}^{\rm ur}$-submodule $Y$ of $H^{1}(G_{k,S},T)$ such that $H^{1}(G_{k,S},T) = X \oplus Y$. 
We then have an identification
\[
{\bigwedge}^{g+r}_{\bZ_{p}^{\rm ur}}H^{1}(G_{k,S},T) = 
{\bigwedge}^{e}_{\bZ_{p}^{\rm ur}}X \otimes {\bigwedge}^{g+r-e}_{\bZ_{p}^{\rm ur}}Y\,. 
\]
We take non-zero elements $x \in {\bigwedge}^{e}_{\bZ_{p}^{\rm ur}}X$ and $y \in  {\bigwedge}^{g+r-e}_{\bZ_{p}^{\rm ur}}Y$, and set
\[
\Psi := \bigwedge_{v \in \Sigma \cap V}{\rm ord}_{v} \wedge \bigwedge_{v \in \Sigma}\log_{v} \in {\bigwedge}^{g+r-e}_{\bZ_{p}^{\rm ur}}\Hom_{\bZ_{p}^{\rm ur}}(Y, \bZ_{p}^{\rm ur})\,.
\]
Since the maps ${\rm ord}_{v}$ and $\log_{v}$ are identically zero on $X$ by definition, we compute that 
\begin{align*}
\bigwedge_{v \in V \cap \Sigma^{c}}{\rm rec}_{\Gamma,v}(x) &= 
\pm \Psi(y)^{-1} 
\biggl(\bigwedge_{v \in V \setminus S_{\infty}(k)}\psi_{v} \wedge \bigwedge_{v \in \Sigma}\log_{v}\biggr)(x \otimes y) 
\\
&= \pm \Psi(y)^{-1} \cL_{\Sigma, \Gamma} \otimes \biggl(\bigwedge_{v \in V \setminus S_{\infty(k)}}{\rm ord}_{v} \wedge \bigwedge_{v \in \Sigma }{\rm log}_{v}\biggr)(x \otimes y)
\\
&= \cL_{\Sigma, \Gamma} \otimes \bigwedge_{v \in V \cap \Sigma^{c}}{\rm ord}_{v}(x). 
\end{align*}
\end{proof}

\section{Leading Term of Katz $p$-adic $L$-function}
\label{sec_leading_term_Katz_03_02}
Our goal in this section is to prove $p$-adic Beilinson formulae for the Katz $p$-adic $L$-function $L_{p, \Sigma}^\chi$ at the trivial character (which lies outside its range of interpolation), in the presence of exceptional zeros. As an artefact of the exceptional zero phenomenon, the $\cL$-invariants we have introduced in Section~\ref{Section_L_invariant} make an appearance in our leading term formulae. This is recorded as Theorem~\ref{thm:main} below and it is the main result of this paper.

Let $k_{\Gamma}$ be a $\bZ_{p}$-extension of $k$ with Galois group $\Gamma$ as before. We recall that $\cA_{\Gamma} = \ker\left(\bZ_{p}^{\rm ur}[[\Gamma]] \to \bZ_{p}^{\rm ur}\right)$. 

\begin{definition}\label{def_remark_derivative}
For a non-negative integer $s$ and an element $f \in \cA_{\Gamma}^{s}$, we define 
\[
f^{(s)}(\mathbbm{1}) := \left(f +  \cA_{\Gamma}^{s+1}\right) \in    \cA_{\Gamma}^{s}/\cA_{\Gamma}^{s+1}. 
\]
We remark that $(f^{\iota})^{(s)}(\mathbbm{1}) = (-1)^{s}f^{(s)}(\mathbbm{1})$, where $\iota$ is the usual involution on $\Gamma$. 
\end{definition}

Let $S := S_{\infty}(k) \cup S_{p}(k) \cup S_{\rm ram}(L/k)$, $V := \{v \in S \mid \chi(G_{k_{v}}) = 1\}$, and $V_{p} := V \cap S_{p}(k)$. 
We fix a labelling of elements in $S$ as in Section~\ref{subsec:Rubin--Stark}. 
We then obtain the Rubin--Stark element 
\begin{equation}\label{eqn_tame_Rubin--Stark}
\xi := \epsilon^{S_{\infty}(k),\chi}_{L/k,S_{\infty}(k) \cup S_{\rm ram}(L/k)} \in {\bigwedge}_{\bC_{p}}^{g}(\bC_{p}\cO^{\times}_{L,S_{\infty} \cup S_{\rm ram}(L/k)})^{\chi} = {\bigwedge}_{\bC_{p}}^{g}(\bC_{p}\cO^{\times}_{L})^{\chi}. 
\end{equation}

\begin{remark}\label{rem_arch_reg_rubin_stark}
The isomorphism 
\[
{\bigwedge}^{g}_{\bC_{p}}e_{\chi}(\bC_{p}\cO^{\times}_{L}) \stackrel{\sim}{\longrightarrow} {\bigwedge}^{g}_{\bC_{p}}e_{\chi}(\bC_{p}X_{L,S_{\infty}(k)}) \cong \bC_{p}
\]
(first induced from the the regulator map $\bC_{p}\cO^{\times}_{L}  \stackrel{\sim}{\longrightarrow} \bC_{p}X_{L,S_{\infty}(k)}$ and the second on fixing places that lie above $S_{\infty}(k)$) maps the element $\xi$  to the leading term  ${\displaystyle \lim_{s \to 0}s^{-g}L(\chi^{-1},s)}$ of the Hecke $L$-function associated to the character $\chi^{-1}$ at $s=0$. 
\end{remark}

Recall that we  have the Katz $p$-adic $L$-function $L_{p,\Sigma}^{\chi} \in \bZ_{p}^{\rm ur}[[\Gamma_{\infty}]]$ and we denote by $L_{p,\Sigma}^{\chi}\vert_{\Gamma}$ 
the image of $L_{p,\Sigma}^{\chi}$ in $\bZ_{p}^{\rm ur}[[\Gamma]]$.

\begin{theorem}\label{thm:main}
Let us set $e := \#(V \cap \Sigma^{c})$ and suppose that the hypotheses \eqref{item_H1}--\eqref{item_H5} hold true.
Then ${\rm ord}_{\cA_{\Gamma}}\left(L_{p,\Sigma}^\chi|_{\Gamma}\right) \geq e$ and  
\[
(L_{p,\Sigma}^\chi\vert_{\Gamma})^{(e)}(\mathbbm{1}) = (-1)^{e}\cL_{\Sigma, \Gamma} \cdot 
\prod_{v \in \Sigma}\left(1 - \frac{\chi(v)}{p}\right) \prod_{v \in \Sigma^{c} \setminus V_{p}}(1-\chi(v)^{-1})
\bigwedge_{v \in \Sigma}{\rm log}_{v}(\xi)
\]
in $ \cA_{\Gamma}^{e}/\cA_{\Gamma}^{e+1}$. 
In particular, ${\rm ord}_{\cA_{\Gamma}}\left(L_{p,\Sigma}^\chi\vert_{\Gamma}\right) = e$ if and only if 
$\cL_{\Sigma, \Gamma} \neq 0$. 
\end{theorem}

\begin{proof}
By (\ref{eq:gamma}), we have 
\begin{align*}
\bigwedge_{v \in \Sigma}{\rm Col}_{\Gamma,v}({\rm loc}_{\Sigma}(\epsilon^{S_{\infty}(k),\chi}_{L_{\Gamma}/k,S})) = L_{p,\Sigma}^{\chi,\iota}\vert_{\Gamma}. 
\end{align*}
Put $r := \#V_{p} = \#(V \cap S_{p}(k))$ and fix a topological generator $\gamma$ of $\Gamma$. 
Then, by Conjecture~\ref{MRS1}, there is a unique element $\kappa_{\infty,\gamma} \in {\bigcap}^{g}_{\bZ_{p}^{\rm ur}[[\Gamma]]}H^{1}(G_{k,S},\bT_{\Gamma})$ such that 
\[
\epsilon^{S_{\infty}(k),\chi}_{L_{\Gamma}/k,S} = (\gamma-1)^{r}\kappa_{\infty,\gamma}. 
\]
We therefore have 
\begin{align}\label{eq:gamma'}
\bigwedge_{v \in \Sigma}{\rm Col}_{\Gamma,v}\left((\gamma-1)^{r} \cdot {\rm loc}_{\Sigma}(\kappa_{\infty,\gamma})\right) = L_{p,\Sigma}^{\chi,\iota}\vert_{\Gamma}. 
\end{align}
By the definition of ${\rm Col}_{\Gamma,v}$, we see that 
\[
{\rm ord}_{\cA_{\Gamma}}\left(L_{p,\Sigma}\vert_{\Gamma}\right) \geq r - \#(V \cap \Sigma) = e. 
\]
Furthermore, since the $\bZ_{p}^{\rm ur}[[\Gamma]]$-homomorphism 
\[
\bigwedge_{v \in \Sigma}{\rm Col}_{\Gamma,v} \colon 
\bigotimes_{v \in V \cap \Sigma}\cA_{\Gamma_{v}} \otimes_{\bZ_{p}^{\rm ur}[[\Gamma]]}
{\bigcap}^{g}_{\bZ_{p}^{\rm ur}[[\Gamma]]} H^{1}_{\Sigma}(\bT_{\Gamma}) \longrightarrow \bZ_{p}^{\rm ur}[[\Gamma]] 
\] 
is an isomorphism, we have 
\[
(\gamma-1)^{r-e} \otimes {\rm loc}_{\Sigma}(\kappa_{\infty,\gamma}) \in \bigotimes_{v \in V \cap \Sigma}\cA_{\Gamma_{v}} \otimes_{\bZ_{p}^{\rm ur}[[\Gamma]]}{\bigcap}^{g}_{\bZ_{p}^{\rm ur}[[\Gamma]]} H^{1}_{\Sigma}(\bT_{\Gamma}). 
\]
Let $\kappa_{\gamma}$ denote the image of $\kappa_{\infty,\gamma}$ in ${\bigcap}^{g}_{\bZ_{p}^{\rm ur}}H^{1}(G_{k,S},T)$. 
We then have 
\[
(\gamma-1)^{r-e} \otimes {\rm loc}_{\Sigma}(\kappa_{\gamma}) \in 
\bigotimes_{v \in V \cap \Sigma}\cA_{\Gamma_{v}}/\cA_{\Gamma}\cA_{\Gamma_{v}} \otimes_{\bZ_{p}^{\rm ur}}{\bigcap}^{g}_{\bZ_{p}^{\rm ur}[[\Gamma]]} H^{1}_{\Sigma}(\bT_{\Gamma}). 
\]
It therefore follows from (\ref{eq:gamma'}) that 
\begin{align}\label{eq:L-gamma'}
\bigwedge_{v \in \Sigma}\widetilde{{\rm Col}}_{\Gamma,v}\bigl((\gamma-1)^{r-e} \otimes {\rm loc}_{\Sigma}(\kappa_{\gamma})\bigr) \otimes (\gamma-1)^{e} = (L_{p,\Sigma}^{\chi,\iota}\vert_{\Gamma})^{(e)}(\mathbbm{1}). 
\end{align}
Moreover, we have  by Conjecture~\ref{MRS} (which we assume) that
\[
\bigwedge_{v \in V_{p}} {\rm rec}_{\Gamma,v}(\epsilon^{V, \chi}_{L/k, S}) = (-1)^{gr} (\gamma-1)^{r} \otimes \kappa_{\gamma}.
\] 
Combining with (\ref{eq:L-gamma'}), we conclude that 
\begin{align}\label{Eqn_iota_katz_leading_1}
\notag\bigwedge_{v \in V_{p}} {\rm rec}_{\Gamma,v} \wedge \bigwedge_{v \in \Sigma}\widetilde{{\rm Col}}_{\Gamma,v}({\rm loc}_{\Sigma}(\epsilon^{V, \chi}_{L/k, S})) 
&= \bigwedge_{v \in \Sigma}\widetilde{{\rm Col}}_{\Gamma,v}\bigl((\gamma-1)^{r-e} \otimes {\rm loc}_{\Sigma}(\kappa_{\gamma})\bigr) \otimes (\gamma-1)^{e}
\\
&= (-1)^{gr}(L_{p,\Sigma}^{\chi, \iota}\vert_{\Gamma})^{(e)}(\mathbbm{1}). 
\end{align}
By Proposition~\ref{prop:const1} and Proposition~\ref{prop:const2}, we have 
\begin{equation}\label{Eqn_iota_katz_leading_X}
\bigwedge_{v \in V_{p}} {\rm rec}_{\Gamma,v} \wedge \bigwedge_{v \in \Sigma}\widetilde{{\rm Col}}_{\Gamma,v} 
= 
\prod_{v\in \Sigma}\left(1-\frac{\chi(v)}{p}\right) \prod_{v \in \Sigma \setminus V_{p}}(1-\chi(v)^{-1})^{-1}\bigwedge_{v \in V_{p}} \psi_{v} \wedge \bigwedge_{v \in \Sigma}\log_{v}, 
\end{equation}
where we recall that 
\[
\psi_{v} := 
\begin{cases}
{\rm rec}_{\Gamma,v} & \text{ if } v \in V \cap \Sigma^{c}, 
\\
{\rm ord}_{v} & \text{ if } v \in V \cap \Sigma\,. 
\end{cases}
\]
Proposition~\ref{prop:stark-reln} implies that 
\begin{equation}\label{Eqn_iota_katz_leading_2}
\bigwedge_{v \in V_{p}}{\rm ord}_{v}(\epsilon^{V}_{L/k,S}) = (-1)^{gr}
\epsilon^{S_{\infty}(k)}_{L/k,S \setminus V_{p}}
= (-1)^{gr} \prod_{v \in S_{p}(k) \setminus V_{p}}(1-\chi(v)^{-1}) \xi. 
\end{equation}
We see by the very definition of $\cL_{\Gamma,\Sigma}$ that 
\begin{align}\label{Eqn_iota_katz_leading_3}
\notag\bigwedge_{v \in V } \psi_{v} \wedge \bigwedge_{v \in \Sigma}\log_{v}(\epsilon^{V, \chi}_{L/k, S}) &= 
\cL_{\Gamma,\Sigma} \bigwedge_{v \in V } {\rm ord}_{v} \wedge \bigwedge_{v \in \Sigma}\log_{v}(\epsilon^{V, \chi}_{L/k, S})
\\
&= (-1)^{gr}\cL_{\Gamma,\Sigma} \prod_{v \in S_{p}(k) \setminus V_{p}}(1-\chi(v)^{-1})\bigwedge_{v \in \Sigma}\log_{v}(\xi)
\end{align}
where the second equality follows from \eqref{Eqn_iota_katz_leading_2}. Thence,
\begin{align*}
(L_{p,\Sigma}^{\chi}\vert_{\Gamma})^{(e)}(\mathbbm{1}) &= (-1)^{e}(L_{p,\Sigma}^{\chi,\iota}\vert_{\Gamma})^{(e)}(\mathbbm{1})
\\
&= (-1)^{gr+e}
\bigwedge_{v \in V } {\rm rec}_{\Gamma,v} \wedge \bigwedge_{v \in \Sigma}\widetilde{{\rm Col}}_{\Gamma,v}(\epsilon^{V, \chi}_{L/k, S'}) 
\\
&=  (-1)^{e}\cL_{\Gamma,\Sigma} \prod_{v \in \Sigma}\left(1 - \frac{\chi(v)}{p}\right)\prod_{v \in \Sigma^{c} \setminus V_{p}}(1-\chi(v)^{-1}) \bigwedge_{v \in \Sigma}\log_{v}(\xi)
\end{align*}
where the first equality is the remark in Definition~\ref{def_remark_derivative}, second is \eqref{Eqn_iota_katz_leading_1} and the final equality is \eqref{Eqn_iota_katz_leading_X} combined with \eqref{Eqn_iota_katz_leading_3}.
\end{proof}

\subsection{Remarks in the special case when $k$ is imaginary quadratic}
\label{subsubsec_new_expand_on_remark_1_7_new_remark_1_3}  
We will conclude the main body of our article with an elaboration on Remark~\ref{remark_Kronecker_Katz_Gross_Ferrero_Greenberg_Koblitz}.

Suppose that $k$ is an imaginary quadratic field in which $p = \fp\fp^c$ splits and put $\Sigma = \{\fp\}$. We also assume that the class number of $k$ is coprime to $p$. In this scenario, the hypotheses \eqref{item_H1}--\eqref{item_H5} are satisfied. We assume in addition that $\chi(\fp^c) = 1$. 
For national simplicity, we put 
\[
L_{\fp}^\chi := L_{p,\Sigma}^\chi 
\]
 as in Section~\ref{subsec_intro_imag_quad_case}. Then by Theorem \ref{thm_intro_imaginary_quadratic_full} (which is a special case of Theorem~\ref{thm:main}, c.f. Remark~\ref{remark_how_to_deduce_the_case_of_imaginary_fields_from_the_general_case}), we have 
\begin{align}\label{eq:main-imaginary}
    (L_{\fp}^\chi\vert_{\Gamma})^{'}(\mathbbm{1}) = -\cL_{\{\fp\}, \Gamma} \cdot 
\left(1 - \frac{\chi(\fp)}{p}\right){\rm log}_{\fp}(\xi)
\end{align}
 where $\xi$ is the Rubin--Stark element in this setting (that we will identify below with a suitable elliptic unit). Our goal in Section~\ref{subsubsec_new_expand_on_remark_1_7_new_remark_1_3} is to explain how to recover \eqref{eq:main-imaginary} in two particular scenarios (treated separately in Sections \ref{subsubsec_611_06_02_2022} and \ref{subsubsec_612_06_02_2022}), relying on various earlier important results on relevant themes.

Before moving in this direction, we recall that the regulator isomorphism 
$\lambda_{L,S_{\infty}(k)} \colon \bC \cO^{\times}_{L} \stackrel{\sim}{\longrightarrow} \bC X_{L, S_{\infty}(k)} \cong \bC[\Gal(L/k)]$ is given  by 
\[
\lambda_{L,S_{\infty}(k)}(a) =  - \sum_{\sigma \in \Gal(L/k)}\log |\iota_{\infty}(\sigma a)|^2 \sigma^{-1}. 
\]
 Using  the isomorphism $\bC \stackrel{j}{\cong} \bC_p$ we have fixed at the beginning of our article and passing to $\chi$-isotypic components, we obtain the map ${\log}_\infty$ given as the compositum of the isomorphisms
\[
  {\log}_\infty\,:\quad e_\chi \bC_p \cO^{\times}_{L} \xrightarrow[-e_\chi \lambda_{L,S_{\infty}(k)}]{\sim} e_\chi \bC_p[\Gal(L/k)] \xrightarrow[\,\,\,\chi\,\,\,]{\sim} \bC_p\,. 
\]
As explained in \cite[Chapter II, \S5.1, Theorem]{deshalit}, the identity 
\begin{equation}
    \label{eqn_RS_element_for_k_imag_quadratic}
    -\mathrm{\log}_\infty(\xi) = L'(\chi^{-1}, 0)
\end{equation}
that the $\chi$-part of the Rubin--Stark element ought to satisfy (c.f. Equation \eqref{eqn_tame_Rubin--Stark}) is indeed verified in the present set-up by the $\chi$-part of an elliptic unit (defined in terms of ``Robert's invariants'', in the terminology of op.cit.). In other words, when $k$ is imaginary quadratic, the $\chi$-part of the Rubin--Stark element $\xi$ is given as this particular elliptic unit.



\subsubsection{The case $\Gamma = \Gamma_\fp$}
\label{subsubsec_611_06_02_2022}

Let $\Gamma_\fp = \Gal(k(\fp^\infty)/k)$ denote the Galois group of the $\bZ_p$-extension of $k$ which is unramified outside $\fp$ and let us denote by $\tilde{\fp}^c$ the prime of $k(\fp^\infty)$ that lies above $\fp^c$. 
We then have a pair of isomorphisms 
\[
\bZ_p \stackrel{\mathrm{ord}}{\longleftarrow} \bQ_p^{\times, \wedge}/(1+p\bZ_p) \stackrel{\mathrm{rec}}{\longrightarrow} \Gal(k(\fp^\infty)_{\tilde{\fp}^c}/k_{\fp^c}). 
\]
In particular, $\cL_{\{\fp\}, \Gamma_\fp}$  generates $\cA_{\Gamma_\fp}/\cA_{\Gamma_\fp}^2$ and  we fix an isomorphism $\Gamma_\fp \cong 1 + p\bZ_p$ such that 
\begin{equation}
    \label{eqn_L_invarinat_equals_one_Gammape}
    \cL_{\{\fp\}, \Gamma_\fp} = 1\,.
\end{equation}

On the other hand, the $p$-adic Kronecker limit formula (due to Katz) reads as follows: 
\begin{theorem}[\cite{deshalit}, Chapter II, \S5.2, Theorem]
Let us denote by $\tilde{L}_{\fp, \ff}(\chi^{-1})$ the value of the one-variable $p$-adic $L$-function $\tilde{L}_{\fp, \ff}$ (given as in \cite{deshalit}, Chapter II, \S4.16, Definition) at the character $\chi^{-1}$, where $\ff$ is the conductor of $\chi$. Then we have 
\begin{align}\label{eq:kronecker-limit-formula}
    \tilde{L}_{\fp, \ff}(\chi^{-1}) = -
    \left(1 - \frac{\chi(\fp)}{p}\right){\rm log}_{\fp}(\xi)\,. 
\end{align}
 Note that, as the character $\chi$ is unramified at $\fp$, the Gauss sums are absent in the formula \eqref{eq:kronecker-limit-formula}, c.f. \cite[Page 75, Remark (i)]{deshalit}. 
\end{theorem}

By the constructions of $L_{\fp}^{\chi}$ and $\tilde{L}_{\fp, \ff}$ (see \cite[Chapter II, \S4]{deshalit} for  details), we have 
\[
(L_{\fp}^\chi\vert_{\Gamma_\fp})^{'}(\mathbbm{1}) = \tilde{L}_{\fp, \ff}(\chi^{-1})
\]
(we remark that the branch character of our $p$-adic $L$-function $L_{\fp}^\chi$ is $\chi^{-1}$).  In particular, the formula \eqref{eq:main-imaginary} (Theorem~\ref{thm_intro_imaginary_quadratic_full} in the special case when $\Gamma=\Gamma_\fp$, combined with \eqref{eqn_L_invarinat_equals_one_Gammape}) and formula \eqref{eq:kronecker-limit-formula} (the $p$-adic Kronecker limit formula) are indeed identical. 


\subsubsection{The case $\Gamma = \Gamma_{\rm cyc}$ and $\chi$ is the restriction of a Dirichlet character $\chi_\QQ$ of $\QQ$.}
\label{subsubsec_612_06_02_2022}

Suppose that $\Gamma = \Gamma_{\rm cyc}$ is the Galois group of the cyclotomic $\bZ_p$-extension of $k$ and $\chi$ arises as the restriction of an even Dirichlet character $\chi_\QQ$ of $\QQ$.  In Section~\ref{subsubsec_612_06_02_2022}, we will explain using Gross' factorization of Katz $p$-adic $L$-function that the identity \eqref{eq:main-imaginary} is equivalent (in this particular set-up) to the Ferrero--Greenberg theorem.

Let us denote by $f := \mathrm{cond}(\chi_{\bQ})$ the conductor of $\chi_{\bQ}$. 
Let $\omega$ denote the Teichm\"uller character and $\epsilon_{k}$ denote the quadratic character associated with $k/\bQ$. 

Until the end of this article, we will regard
$L_{\fp}^\chi\vert_{\Gamma_{\rm cyc}}$ as an element of the formal power series ring $\bZ_p^{\rm ur}[[s]]$ via the Amice transform. Let $L_p(s, \chi^{-1}_{\bQ}\epsilon_{k}\omega)$ and $L_p(s, \chi_{\bQ})$ denote the Kubota--Leopoldt $p$-adic $L$-functions associated with Dirichlet characters $\chi^{-1}_{\bQ}\epsilon_{k}\omega$ and $\chi_{\bQ}$, respectively. 
Then the factorization formula of Gross (the main theorem of \cite{GroosFactorisation}) reads
\begin{equation}
\label{eqn_Gross_factorization_nontrivial_f}
    L_{\fp}^\chi\vert_{\Gamma_{\rm cyc}} = \frac{G(\chi_{\bQ}^{-1})}{4} L_p(s, \chi^{-1}_{\bQ}\epsilon_{k}\omega)L_p(1-s, \chi_{\bQ}).
\end{equation}
Here $G(\chi_{\bQ}^{-1}) = \sum_{a=1}^{f}\chi_{\bQ}^{-1}(a)\zeta_f^a$ is the Gauss sum attached to the Dirichlet character $\chi_{\bQ}^{-1}$, where $\zeta_f$ stands for a primitive $f^{\rm th}$ root of unity. 

\begin{remark}
In his paper \cite{GroosFactorisation}, Gross only considers the case when $f$ is a power of $p$. However, with suitable adjustments, his argument in op. cit. works for general $\chi_{\QQ}$. In the setting we have placed ourselves, the Guass sum $G(\chi_{\bQ}^{-1})$ in \eqref{eqn_Gross_factorization_nontrivial_f} comes from the functional equation for the Dirichlet $L$-series $L(s, \chi_{\bQ})$, namely, the identity 
$$2 L'(0, \chi_{\bQ}^{-1}) = G(\chi_{\bQ}^{-1})L(1, \chi_{\bQ}).$$ 
When the conductor of $\chi_\QQ$ is a power of $p$ (as in the case of Gross' paper), the Gauss sum also appears as a factor of the value of the left hand side of \eqref{eqn_Gross_factorization_nontrivial_f} at $s=0$ (c.f. \cite{deshalit}, Chapter II, Theorem 4.14), and these two Gauss sums cancel each other out. This is the reason why the Gauss sum does not appear in the main theorem of \cite{GroosFactorisation}, where the conductor of $\chi_\QQ$ is a power of $p$. 

In our case, the conductor of $\chi_\QQ$ is coprime to $p$, and therefore the Gauss sum $G(\chi_{\bQ}^{-1})$ does not appear in the special value formulae for $L_{\fp}^\chi$ (c.f. \cite{deshalit}, Page 75, Remark (i)) and there is no longer a cancellation of Gauss sums. This is the reason for the appearance of $G(\chi_{\bQ}^{-1})$ in \eqref{eqn_Gross_factorization_nontrivial_f}.  

We also note that our $p$-adic $L$-function $L_{\fp}^\chi$ and the normalization of Gross for the Katz $p$-adic $L$-function differ by a factor of $4$; 
c.f. \cite{GroosFactorisation} Theorem 2.3 and Equation (2.4) in comparison with \cite{deshalit} Page 86, Equation (50). 
\end{remark}

The interpolation formula for the Kubota--Leopoldt $p$-adic $L$-function $L_p(s, \chi^{-1}_{\bQ}\epsilon_{k}\omega)$ shows that $L_p(0, \chi^{-1}_{\bQ}\epsilon_{k}\omega) = 0$.  We therefore have, thanks to \eqref{eqn_Gross_factorization_nontrivial_f},
\begin{align}\label{eq:grossfactorisation}
    (L_{\fp}^\chi\vert_{\Gamma_{\rm cyc}})'(0) = \frac{G(\chi_{\bQ}^{-1})}{4} L_p'(0, \chi^{-1}_{\bQ}\epsilon_{k}\omega)L_p(1, \chi_{\bQ}). 
\end{align}
The value of $L_p(s, \chi_{\bQ})$ at $s=1$ is well-known (c.f. \cite{washington}, Theorem 5.18): 
\begin{align}\label{eq:L_p(1)}
    L_p(1, \chi_{\bQ}) 
    &= -\left(1 - \frac{1}{p}\right) \cdot \frac{G(\chi_{\bQ})}{f}\sum_{a=1}^{f}\chi_{\bQ}(a)^{-1}\log_p(1-\zeta^a_f). 
\end{align}
It is also well-known (c.f. \cite[Theorem 4.9]{washington} for the computation of $L(1, \chi_{\bQ})$, which can be translated to the formula below for $L'(0, \chi_{\bQ}^{-1})$ using the functional equation for Dirichlet $L$-functions and the identity  $G(\chi_\QQ^{-1})G(\chi_\QQ)=f$ for the Gauss sum of the even Dirichlet character $\chi_\QQ$) that 
\begin{align}
\label{eqn_LprimezerochiQ}
\begin{aligned}
    L'(0, \chi_{\bQ}^{-1}) &= - \frac{1}{2}\sum_{a=1}^{f}\chi_{\bQ}(a)^{-1}\log|1-\zeta^a_f| 
    \\
    &= - \frac{1}{4}\sum_{\sigma \in \Gal(\bQ(\zeta_f)/\bQ)}\chi_{\bQ}(\sigma)^{-1}\log|\iota_{\infty}(\sigma(1-\zeta_f))|^2
\end{aligned}
\end{align}
 Observe that 
\begin{align}
\label{eqn_aligned_compare_regulator_images_of_cyclo_vs_elliptic}
    \begin{aligned}
        -\log_\infty(\xi) &\stackrel{\eqref{eqn_RS_element_for_k_imag_quadratic}}{=}\,L'(0, \chi^{-1}) = L(0, \chi^{-1}_{\bQ}\epsilon_{k}) \cdot L'(0, \chi^{-1}_{\bQ})\\
        &\stackrel{\eqref{eqn_LprimezerochiQ}}{=} - \frac{1}{4}L(0, \chi^{-1}_{\bQ}\epsilon_{k})\cdot \sum_{\sigma \in \Gal(\bQ(\zeta_f)/\bQ)}\chi_{\bQ}(\sigma)^{-1}\log|\iota_{\infty}(\sigma(1-\zeta_f))|^2\,.
    \end{aligned}
\end{align}
Combining \eqref{eq:L_p(1)} with \eqref{eqn_aligned_compare_regulator_images_of_cyclo_vs_elliptic} and the fact that $e_{\chi}\mathbb{C}\cO_L^\times$ is a $\mathbb{C}$-vector space of dimension one, we deduce that
\begin{align}\label{eq:L_p(1,chi)}
    L_p(1, \chi_{\bQ}) = -\left(1 - \frac{1}{p}\right) \cdot \frac{4 }{L(0, \chi^{-1}_{\bQ}\epsilon_{k})} \frac{G(\chi_{\bQ})}{f} \log_\fp(\xi).  
\end{align}
Gross' formula (proved in \cite[Section 4]{grossconjecture} using the Gross--Koblitz formula \cite{GrossKoblitz} and the work of Ferrero and Greenberg \cite{FerrreroGreenberg}) reads
\begin{align}\label{eq:L_p'(0,chi-epsilon)}
    L_p'(0, \chi^{-1}_{\bQ}\epsilon_{k}\omega) = \cL(\chi^{-1}_{\bQ}\epsilon_{k}) \cdot L(0, \chi^{-1}\epsilon_{k})\,.
\end{align}
(See also \cite{DarmonDasguptaPollack}, Equation (8) and Remark 8). Here $\cL(\chi^{-1}_{\bQ}\epsilon_{k})$ is Greenberg's $\cL$-invariant given as in  \cite[Equation (7)]{DarmonDasguptaPollack} 
(we remark for the convenience of the reader that in \cite{DarmonDasguptaPollack}, the $\chi_{\bQ}\epsilon_{k}$-components are notated by $(-)_{\chi^{-1}_{\bQ}\epsilon_{k}}$). 
As we have explained in Remark \ref{remark_L_invariant_for_Gamma_cyc_imaginary_quadratic}, we have 
\begin{equation}
\label{eqn_compare_our_L_invariant_to_Greenberg_final}
 \cL(\chi^{-1}_{\bQ}\epsilon_{k}) = \chi_{\rm cyc}(\cL_{\{\fp\}, \Gamma_{\rm cyc}}). 
\end{equation}
We combine \eqref{eq:L_p(1,chi)}, \eqref{eq:L_p'(0,chi-epsilon)} and \eqref{eqn_compare_our_L_invariant_to_Greenberg_final} together with the fact that $G(\chi_{\bQ})G(\chi_{\bQ}^{-1}) = f$ to conclude that 
\[
(L_{\fp}^\chi\vert_{\Gamma_{\rm cyc}})'(0) = -  \chi_{\rm cyc}(\cL_{\{\fp\}, \Gamma_{\rm cyc}})\left(1 - \frac{1}{p}\right) \log_\fp(\xi), 
\]
which is a restatement of \eqref{eq:main-imaginary} in the present set-up.

\appendix

\section{Continuous Cochain Complex}
In this appendix, we record various properties of complexes of continuous cochains. We have made use of these properties in Section~\ref{subsec_constant_terms_Coleman_maps}.

Let $(R,\fm)$ be a complete noetherian local ring with residue characteristic $p$. 
Let $S$ be a finite set of places of $k$ with $S_{\infty}(k) \cup S_{p}(k) \subseteq S$. 
In this appendix, let $T$ be a free $R$-module of finite rank on which $G_{k,S}$ acts continuously; namely, the composite map 
\[
G_{k,S} \hookrightarrow R[G_{k,S}] \longrightarrow {\rm End}_{R}(T)
\]
is continuous. 
Let us set
\[
\cG := \{G_{k_{v}} \mid v \in S_{p}(k)\} \cup \{G_{k,S}\}. 
\]
We note that, for $G \in \cG$, the pair $(T,G)$ satisfies \cite[Hypotheses A]{Jonathan13}. 
For $G \in \cG$ and a topological $G$-module $M$, let $C^{\bullet}(G,M)$ denote the complex of continuous cochains. 
We also denote the object in the derived category corresponding to the complex $C^{\bullet}(G,M)$ by ${\bf R}\Gamma(G,M)$. 

\begin{definition}
Let $A$ be a noetherian ring and $M^{\bullet}$ a complex of $A$-modules. 
\begin{itemize}
\item[(i)] We say that $M^{\bullet}$ is a perfect complex if there exists a quasi-isomorphism 
\[
P^{\bullet} \longrightarrow M^{\bullet},
\] 
where $P^{\bullet}$ is a bounded complex of projective $A$-modules of finite type. 
Furthermore, if $P^{i} = 0$ for every $i < a$ and $i > b$, then we say that $M^{\bullet}$ has perfect amplitude contained in $[a,b]$. 
\item[(ii)] We denote by $D_{\rm parf}^{[a,b]}(A)$ the full subcategory consisting of perfect complexes having perfect amplitude contained in $[a,b]$ in the derived category of the abelian category of $A$-modules. 
\end{itemize}
\end{definition}

\begin{lemma}[{\cite[Theorem~1.4(1)]{Jonathan13}}]
\label{lem:descent}
For $G \in \cG$ and an ideal $I$ of $R$, we have 
\[
{\bf R}\Gamma(G,T) \otimes^{\bL}_{R}R/I \cong {\bf R}\Gamma(G,T/IT). 
\]
\end{lemma}

\begin{lemma}[{\cite[Corollary~1.2]{Jonathan13}}]
\label{lem:perf}
For $G \in \cG$, we have ${\bf R}\Gamma(G,T) \in D_{\rm parf}^{[0,2]}(R)$. 
\end{lemma}

\begin{corollary}\label{cor:H2}
For $G \in \cG$ and an ideal $I$ of $R$, there is a canonical isomorphism 
\[
H^{2}(G,T) \otimes_{R}R/I \cong H^{2}(G,T/IT). 
\]
\end{corollary}
\begin{proof}
This corollary easily follows from Lemma~\ref{lem:descent} and Lemma~\ref{lem:perf}. 
\end{proof}

\begin{corollary}\label{cor:amp}
Suppose $G \in \cG$ and assumme that $H^{0}(G,T/\fm T)=0$.  
Then,
\[
{\bf R}\Gamma(G,T) \in D_{\rm parf}^{[1,2]}(R).
\] 
\end{corollary}
\begin{proof}
By Lemma~\ref{lem:perf}, there is a complex 
\[
P^{\bullet} = [ \ \cdots \longrightarrow 0 \longrightarrow P^{0} \longrightarrow P^{1} \longrightarrow P^{2} \longrightarrow 0 \longrightarrow \cdots \ ]
\] 
of finitely generated projective $R$-modules such that $P^{\bullet}$ is quasi-isomorphic to ${\bf R}\Gamma(G,T)$. 
Since $H^{0}(G,T/\fm T)$ vanishes, we have 
\[
(\fm^{i}/\fm^{i+1} \otimes_{R} T)^{G} \cong (T \otimes_{R} (R/\fm)^{\dim_{R/\fm}\fm^{i}/\fm^{i+1}})^{G} = 0. 
\]
The exact sequence 
$$0 \longrightarrow \fm^{i}/\fm^{i+1} \longrightarrow R/\fm^{i+1} \longrightarrow R/\fm^{i} \longrightarrow 0$$ 
shows that $(T/\fm^{i} T)^{G} = 0$ for any integer $i \geq 1$ and 
hence $H^{0}(P^{\bullet}) = T^{G} = 0$. In other words, the map $P^{0} \longrightarrow P^{1}$ is injective. 
To conclude with the proof of Corollary~\ref{cor:amp}, it suffices to show that the cokernel $X := {\rm coker}\left(P^{0} \longrightarrow P^{1}\right)$ is a projective $R$-module. 
Since $R$ is a noetherian local ring, we only need to check ${\rm Tor}_{1}^{R}(X,R/\fm) = 0$ by the local criterion for flatness. 

Since $P^{1}$ is flat, we have 
\[
{\rm Tor}_{1}^{R}(X,R/\fm) = \ker \left(P^{0}\otimes_{R} R/\fm \longrightarrow P^{1}\otimes_{R} R/\fm\right) =  H^{0}({\bf R}\Gamma(G,T) \otimes^{\bL}_{R}R/\fm). 
\]
By Lemma~\ref{lem:descent}, we have ${\bf R}\Gamma(G,T) \otimes^{\bL}_{R}R/\fm \cong {\bf R}\Gamma(G,T/\fm T)$ and hence 
we conclude ${\rm Tor}_{1}^{R}(X,R/\fm) = H^{0}({\bf R}\Gamma(G,T/\fm T)) = (T/\fm T)^{G} = 0$, as required.
\end{proof}

\begin{corollary}\label{cor:free}
Suppose $G \in \cG$. 
Assume that 
\begin{itemize}
\item $R$ is a regular local ring with Krull dimension at most $2$, and 
\item $H^{0}(G,T/\fm T)$ vanishes. 
\end{itemize}
Then the $R$-module $H^{1}(G,T)$ is free. 
Furthermore, if $G = G_{k_{v}}$ for some $v \in S_{p}(k)$, then the rank of $H^{1}(G,T)$ is 
$[k_{v} \colon \bQ_{p}] \cdot {\rm rank}_{R}T - {\rm rank}_{R}H^{2}(G,T)$. 
\end{corollary}
\begin{proof}
By Corollary~\ref{cor:amp}, there is an $R$-homomorphism $P^{1} \to P^{2}$ such that $P^{i}$ is a free $R$-module and $H^{1}(G_{k,S},T) \cong \ker\left(P^{1} \to P^{2}\right)$. This shows that $H^{1}(G_{k,S},T)$ is a reflexive $R$-module. 
Since $R$ is a regular local ring with Krull dimension at most $2$, 
the $R$-module $H^{1}(G_{k,S},T)$ is free. 

Suppose that $G = G_{k_{v}}$ for some $v \in S_{p}(k)$. Lemma~\ref{lem:descent} combined with the local Euler characteristic formula shows  
\[
{\rm rank}_{R}P^{1} - {\rm rank}_{R}P^{2} = [k_{v} \colon \bQ_{p}] \cdot {\rm rank}_{R}T. 
\]
This completes the proof.  
\end{proof}

\begin{corollary}\label{cor:H2vanish}
Suppose $G \in \cG$. 
Assume that 
\begin{itemize}
\item $H^{0}(G,T/\fm T)$ vanishes, and 
\item $H^{2}(G,T/IT)$ vanishes for some proper ideal $I$ of $R$. 
\end{itemize}
Then the $R$-module $H^{1}(G,T)$ is free and $H^{1}(G,T) \otimes_{R} R/J \cong H^{1}(G,T/JT)$ for any ideal $J$ of $R$. 
\end{corollary}
\begin{proof}
It follows from the assumptions that $H^{2}(G,T/JT)=0$ for any ideal $J$ of $R$ and hence  
$
{\bf R}\Gamma(G,T/JT) \in D_{\rm parf}^{[1,1]}(R/J).
$  
In particular, $H^{1}(G,T)$ is free. The surjectivity follows from Lemma~\ref{lem:descent}. 
\end{proof}

\section{Exterior bi-dual}\label{sec:bi-dual}
In this appendix, we review the definition of the exterior bi-dual (which is a natural generalization of the Rubin-lattice) and prove various base-change properties that we utilize in the main body of our article. 

Let $R$ be a Gorenstein ring. 
For a finitely generated $R$-module $M$, 
we put $M^{*} := \Hom_{R}(M,R)$ and 
\[
{\bigcap}^{r}_{R}M := \left({\bigwedge}^{r}_{R}(M^{*})\right)^{*}
\] 
for any integer $r \geq 0$. 
Note that there is a canonical $R$-homomorphism 
\[
{\bigwedge}^{r}_{R}M \longrightarrow {\bigcap}^{r}_{R}M. 
\]

\begin{remark}\label{rem:rubin-lattice}
If $R$ is a finite group ring over a Dedekind domain $R'$, 
then there is a canonical isomorphism 
\[
\left\{ x \in  {\rm Frac}(R')\otimes_{R'} {\bigwedge}^{r}_{R}M \ \middle| \ \Phi(a) \in R \text{ for all } \Phi \in {\bigwedge}^{r}_{R}M^{*}\right\} 
\stackrel{\sim}{\longrightarrow} {\bigcap}^{r}_{R}M.  
\]
This tells us that the exterior bi-dual ${\bigcap}^{r}_{R}M$ simply agrees what is known in the literature as the Rubin-lattice. 
\end{remark}

For non-negative integers $r \geq s$ and $\Phi \in {\bigwedge}^{r}_{R}M^{*}$, we define a homomorphism 
\[
\Phi \colon {\bigcap}^{s}_{R}M \longrightarrow {\bigcap}^{s-r}_{R}M
\]
by setting it as the $R$-dual of the map ${\bigwedge}^{s-r}_{R}M^{*} \longrightarrow {\bigwedge}^{s}_{R}M^{*}; \Psi \mapsto \Phi \wedge \Psi$. 
For any $R$-homomorphism $\phi \in M^{*}$, we define an $R$-homomorphism (also denoted by $\phi$)
\begin{align*}
    \phi \colon {\bigwedge}^{s}_{R}M &\longrightarrow {\bigwedge}^{s-1}_{R}M
\\
x_{1} \wedge \cdots \wedge x_{s} &\longmapsto \sum_{i=1}^{s}(-1)^{i-1}\phi(x_{i})x_{1} \wedge \cdots \wedge x_{i-1} \wedge x_{i+1} \wedge \cdots \wedge x_{s}.  
\end{align*}
We therefore obtain an $R$-homomorphism 
\begin{align}\label{eq:Phi}
    \Phi \colon {\bigwedge}^{s}_{R}M \longrightarrow {\bigwedge}^{s-r}_{R}M. 
\end{align}
Note that, if $\Phi = \phi_{1} \wedge \cdots \wedge \phi_{r}$, then 
$\Phi \colon {\bigwedge}^{s}_{R}M \longrightarrow {\bigwedge}^{s-r}_{R}M$ is defined by $\phi_{r} \circ \cdots \circ \phi_{1}$. 
Furthermore, the diagram 
\[
\xymatrix{
{\bigwedge}^{s}_{R}M \ar[r]^{\Phi} \ar[d] & {\bigwedge}^{s-r}_{R}M \ar[d]
\\
{\bigcap}^{s}_{R}M \ar[r]^{\Phi} & {\bigcap}^{s-1}_{R}M
}
\]
commutes. 

\begin{lemma}\label{lem:expre}
Suppose that we have an exact sequence 
\[
0 \longrightarrow M \longrightarrow P^{1} \longrightarrow P^{2} \longrightarrow N \longrightarrow 0 
\]
of $R$-modules, where $P^{1}$ and $P^2$ are finitely generated free $R$-modules. 
Put $I = {\rm im}(P^{1} \to P^{2})$. 
Then the canonical map 
\[
{\bigcap}^{r}_{R}M \longrightarrow {\bigcap}^{r}_{R}P^{1} \cong {\bigwedge}^{r}_{R}P^{1} 
\]
induces an isomorphism 
\[
{\bigcap}^{r}_{R}M \cong \ker\left({\bigwedge}^{r}_{R}P^{1} \longrightarrow I \otimes_{R} {\bigwedge}^{r-1}_{R}P^{1}\right) 
\]
for any integer $r \geq 1$. 
\end{lemma}
\begin{proof}
Since $P^{2}$ is a free $R$-module, there is a canonical isomorphism $\Ext_{R}^{1}(I,R) \stackrel
{\sim}{\longrightarrow} \Ext_{R}^{2}(N,R)$. 
On passing to $R$-duals, we have 
$$
0 \longrightarrow I^{*} \longrightarrow (P^{1})^{*} \longrightarrow M^{*} \longrightarrow \Ext_{R}^{2}(N,R) \longrightarrow 0. 
$$ 
Put $X := \im((P^{1})^{*} \longrightarrow M^{*})$. 
This exact sequence gives rise to the exact sequence 
\[
I^{*} \otimes_{R} {\bigwedge}^{r}_{R} (P^{1})^{*} \longrightarrow {\bigwedge}^{r}_{R}(P^{1})^{*} \longrightarrow {\bigwedge}^{r}_{R}X \longrightarrow 0. 
\]
Passing to $R$-duals once again and using the canonical identification ${\bigwedge}^{r}_{R}P^{1} = {\bigcap}^{r}_{R}P^{1}$, we have an exact sequence 
\[
0 \longrightarrow \left({\bigwedge}^{r}_{R}X \right)^{*} \longrightarrow {\bigwedge}^{r}_{R}P^{1} \longrightarrow \left(I^{*} \otimes_{R} {\bigwedge}^{r-1}_{R}(P^{1})^{*}\right)^{*}
\]
since $P^{1}$ is free of finite rank. 
Note that there is canonical isomorphisms
\[
\left(I^{*} \otimes_{R} {\bigwedge}^{r-1}_{R}(P^{1})^{*}\right)^{*} \cong \Hom_{R}\bigr(I^{*},{\bigcap}_{R}^{r-1}P^{1}\bigl) \cong I^{**} \otimes_{R} {\bigcap}_{R}^{r-1}P^{1} 
\cong I^{**} \otimes_{R} {\bigwedge}_{R}^{r-1}P^{1} 
\]
and $I \longrightarrow I^{**}$ is injective since $I$ is a submodule of $P^{2}$. 
We conclude that 
\[
\left({\bigwedge}^{r}_{R}X \right)^{*} = \ker\left({\bigwedge}^{r}_{R}P^{1} \longrightarrow I \otimes_{R} {\bigwedge}^{r-1}_{R}P^{1}\right). 
\]
Hence in order to complete the proof, it suffices to show that a canonical homomorphism 
\[
{\bigcap}^{r}_{R}M \longrightarrow \left({\bigwedge}^{r}_{R}X \right)^{*}
\]
is an isomorphism. 
Since $R$ is a Gorenstein ring, $\Ext^{2}_{R}(N,R) \otimes_{R} R_{\fp}$ vanishes for any height one prime $\fp$ of $R$. This implies that the dimension of both the kernel and cokernel of the homomorphism 
\[
{\bigwedge}^{r}_{R}X \longrightarrow {\bigwedge}^{r}_{R}M^{*}
\]
is less than or equal to $\dim R -2$. 
Since we have $\Ext^{0}_{R}(Y,R) = \Ext^{1}_{R}(Y,R) = 0$ for a finitely generated $R$-module $Y$ satisfying $\dim Y \leq \dim R -2$, passing to $R$-duals, we see that the homomorphism 
\[
{\bigcap}^{r}_{R}M \longrightarrow \left({\bigwedge}^{r}_{R}X \right)^{*}
\]
is an isomorphism. 
\end{proof}

\begin{lemma}\label{lem:reduction}
Let $S$ be a Gorenstein ring and $R \longrightarrow S$ be a ring homomorphism. 
Let $C \in D^{[1,2]}_{\rm perf}(R) $. 
Then for any integer $r \geq 0$, there is a canonical homomorphism 
\[
{\bigcap}^{r}_{R}H^{1}(C) \longrightarrow {\bigcap}^{r}_{S}H^{1}(C \otimes^{\mathbb{L}}_{R} S)
\]
\end{lemma}
\begin{proof}
This follows from Lemma~\ref{lem:expre}. 
\end{proof}

\begin{lemma}\label{lem:inv-bidual}
Let $I$ be a directed poset and $\{R_{i}\}_{i\in I}$ an inverse system of Gorenstein rings satisfying $R = \varprojlim_{i \in I}R_{i}$. 
Let $C \in D^{[1,2]}_{\rm perf}(R)$. 
Then for any integer $r \geq 0$, there is a canonical isomorphism 
\[
{\bigcap}^{r}_{R}H^{1}(C) \stackrel{\sim}{\longrightarrow} \varprojlim_{i \in I}{\bigcap}^{r}_{R_{i}}H^{1}(C \otimes^{\mathbb{L}}_{R} R_{i}). 
\]
\end{lemma}
\begin{proof}
There is an $R$-homomorphism $P^{1} \to P^{2}$ such that each $P^{i}$ is a free $R$-module of finite rank and $H^{1}(C) = \ker(P^{1} \to P^{2})$. 
Set $I := {\rm im}(P^{1} \to P^{2})$. 
For $i \in I$, we put $P^{1}_{i} := P^{1} \otimes_{R} R_{i}$, $P^{2}_{i} := P^{2} \otimes_{R} R_{i}$, $I_{i} := {\rm im}(P^{1}_{i} \to P^{2}_{i})$. 
 We then have by Lemma~\ref{lem:expre} 
\[
{\bigcap}^{r}_{R}H^{1}(C) = \ker\left({\bigwedge}^{r}_{R}P^{1} \longrightarrow I \otimes_{R} {\bigwedge}^{r-1}_{R}P^{1}\right)
\]
and 
\[
\varprojlim_{i \in I}{\bigcap}^{r}_{R}H^{1}(C \otimes^{\mathbb{L}}_{R} R_{i}) =  \varprojlim_{i \in I}\ker\left({\bigwedge}^{r}_{R_{i}}P^{1}_{i} \longrightarrow I_{i} \otimes_{R_{i}} {\bigwedge}^{r-1}_{R_{i}}P^{1}_{i}\right). 
\]
Since $P^{1}_{i}$ is a free $R_{i}$-module, we have 
\[
\varprojlim_{i \in I}\ker\left({\bigwedge}^{r}_{R_{i}}P^{1}_{i} \longrightarrow I_{i} \otimes_{R_{i}} {\bigwedge}^{r-1}_{R_{i}}P^{1}_{i}\right) 
= 
\ker\left({\bigwedge}^{r}_{R}P^{1} \longrightarrow \varprojlim_{i}I_{i} \otimes_{R} {\bigwedge}^{r-1}_{R}P^{1}\right)
\]
Since $I_{i}$ is a submodule of $P^{2}_{i}$, the canonical map $I \longrightarrow \varprojlim_{i}I_{i}$ is injective. This shows that
\[
\ker\left({\bigwedge}^{r}_{R}P^{1} \longrightarrow I \otimes_{R} {\bigwedge}^{r-1}_{R}P^{1}\right)
= \ker\left({\bigwedge}^{r}_{R}P^{1} \longrightarrow \varprojlim_{i \in I}I_{i} \otimes_{R} {\bigwedge}^{r-1}_{R}P^{1}\right)
\]
and completes the proof. 
\end{proof}

\begin{corollary}\label{cor:inv-bidual}
Suppose that $(R,\fm)$ is a complete Gorenstein local ring with  residue characteristic $p$. 
Let $G \in \cG$ and $T$ a free $R$-module of finite rank on which $G$ acts continuously. 
Let $I$ be a directed poset and $\{R_{i}\}_{i\in I}$ an inverse system of complete Gorenstein local rings  satisfying $R = \varprojlim_{i \in I}R_{i}$. 

If $H^{0}(G,T/\fm T)$ vanishes, then, for any integer $r \geq 0$, there is a canonical isomorphism 
\[
{\bigcap}^{r}_{R}H^{1}(G,T) \stackrel{\sim}{\longrightarrow} \varprojlim_{i \in I}{\bigcap}^{r}_{R_{i}}H^{1}(G, T \otimes_{R} R_{i}). 
\]
\end{corollary}
\begin{proof}
This follows from Lemma~\ref{lem:descent}, Corollary~\ref{cor:amp}, and Lemma~\ref{lem:inv-bidual}. 
\end{proof}

\section{Main Conjectures and Exceptional Zeros}\label{appendix_main_conjecture}

In this appendix, we shall recast the Iwasawa main conjecture for the CM field $k$ and the branch character $\chi$ in terms of Selmer complexes. 
This will allow us (on combining Lemma~\ref{lem:excep}) to prove that the Iwasawa Main Conjecture implies the Weak Exceptional Zero Conjecture~{\ref{conj:ex zero}} (see Proposition~\ref{prop_app_IMC_reformulated}). 

We first recall the classical formulation of the Iwasawa Main Conjectures. 
Suppose that the $p$-ordinary hypothesis \eqref{item_ord} holds true, and we take a subset $\Sigma \subseteq S_{p}(k)$ such that 
$\Sigma \cup \Sigma^{c} = S_{p}(k)$ and $\Sigma \cap \Sigma^{c} = \emptyset$. 
Let $M_{\Sigma}$ be the maximal abelian pro-$p$-extension of $Lk(p^{\infty})$ which is $\Sigma$-ramified and let $X_{\Sigma} := \Gal(M_{\Sigma}/Lk(p^{\infty}))$. 
We put 
\[
X_{\Sigma}^{\chi} := e_{\chi}(\bZ_{p}^{\rm ur} \otimes_{\bZ_{p}} X_{\Sigma}). 
\] 
By \cite[Theorem~1.2.2]{HT94}, the $\bZ_{p}^{\rm ur}[[\Gamma_{\infty}]]$-module $X_{\Sigma}^{\chi}$ is torsion. 
The following is the Iwasawa Main Conjecture for the CM field $k$ and the branch character $\chi$ (see \cite{HT94}). 

\begin{conjecture}[Iwasawa Main Conjecture]\label{conj:iwasawa}
The characteristic ideal of $X_{\Sigma}^{\chi}$ is generated by the Katz $p$-adic $L$-function $L_{p,\Sigma}^{\chi}$: 
\[
{\rm char}(X_{\Sigma}^{\chi}) = (L_{p,\Sigma}^{\chi,\iota}). 
\]
\end{conjecture} 

\begin{remark}\label{rem:iwasawa}
If $K$ is imaginary quadratic field, then Conjecture~\ref{conj:iwasawa} is a theorem of Rubin, c.f. \cite[Theorem 4.1(i)]{rubin91}. 
\end{remark}

In general, the following result was proved by Hsieh in~\cite{HsiehJAMS_IMC} towards the validity of Conjecture~\ref{conj:iwasawa}:
\begin{theorem}[Hsieh]\label{HsiehTheorem} Suppose $p>3$ is coprime to $h_{k}^{-}[L:k]$, and in addition that it is unramified in $L/\QQ$. Then,
\item[i)] $ {\rm char}(X_{\Sigma}^{\chi}) \subseteq  L_{p,\Sigma}^{\chi}\cdot \Z_p^{\rm ur}[[\Gamma_\infty]]$ $($Theorem~2 in \cite{HsiehJAMS_IMC}$)$.
\item[ii)] Assume that $\chi$ is anticyclotomic and the local character $\chi_v$ is non-trivial for every $v\in S_p(k)$. Then, the Iwasawa Main Conjecture~\ref{conj:iwasawa} holds $($Corollary~2 in \cite{HsiehJAMS_IMC}$)$.
\item[iii)] Assume $k=k^+M$ where $M$ is an imaginary quadratic field in which $p$ splits, and that $\Sigma_p$ is obtained by extending $\iota_p:M\hookrightarrow \mathbb{C}_p$, and finally that $L$ is abelian over $M$ and $p\nmid [L:M]$. Then, the Iwasawa Main Conjecture~\ref{conj:iwasawa} holds $($Corollary~3 in \cite{HsiehJAMS_IMC}$)$.
\end{theorem}

\begin{remark}
\label{remark_Hsieh_when_useful}
Notice that Theorem~\ref{HsiehTheorem}(ii) excludes the treatment of exceptional zeros and in particular, it offers no utility in the setting of our article. However, Theorem~\ref{HsiehTheorem}(iii) allows exceptional zeros and allows one to deduce Corollary~\ref{corollary_imp2} above.
\end{remark}

Let us resume with our task to reformulate Conjecture~\ref{conj:iwasawa}  in terms of Selmer complexes. 
Let $T := \bZ_{p}^{\rm ur}(1) \otimes \chi^{-1}$. 
Fix a subfield $K$ of $k(p^{\infty})/k$. Recall that we put 
\[
T_{K} := T \otimes_{\bZ_{p}} \bZ_{p}[[\Gal(K/k)]]^{\iota}. 
\]
Let $S = S_{\infty}(k) \cup S_{p}(k) \cup S_{\rm ram}(L/k)$. 
Let $I_{v} := \Gal(\overline{k}_{v}/k^{\rm ur}_{v})$ denote the inertia group at $v \in S$. 
We set 
\[
U_{v}^{+} := 
\begin{cases}
C^{\bullet}(G_{k_{v}},T_{K}) & \text{if} \ v \in \Sigma^{c}, 
\\
0  & \text{if} \ v \in \Sigma, 
\\
C^{\bullet}(\Gal(k^{\rm ur}_{v}/k_{v}),T_{K}^{I_{v}}) & \text{if} \ v \not\in S_{p}(k). 
\end{cases}
\]
Here, for a pro-finite group $G$ and a topological $G$-module $M$, let $C^{\bullet}(G,M)$ denote the complex of continuous cochains. 
Then, for each prime $v$ of $k$, we have a canonical injection 
\[
i_{v}^{+} \colon U_{v}^{+} \hookrightarrow C^{\bullet}(G_{k_{v}},T_{K}). 
\]
Define 
\[
U_{v}^{-}(T) = {\rm Cone}\left(U_{v}^{+} \xrightarrow{-i_{v}^{+}} C^{\bullet}(G_{k_{v}},T_{K})\right). 
\]
Set $i_{S}^{+} := \{i_{v}^{+}\}_{v \in S \setminus S_{\infty}(k)}$ and $U_{S}^{\pm}(T_{K}) := \bigoplus_{v \in S \setminus S_{\infty}(k)}U^{\pm}_{v}(T_{K})$. 

\begin{definition}
Let 
\[
{\rm res}_{f} \colon C^{\bullet}(G_{k,S},T_{K}) \longrightarrow \bigoplus_{v \in S \setminus S_{\infty}(k)}C^{\bullet}(G_{k_{v}},T_{K})
\]
denote the sum of the restriction maps. 
The Selmer complex on $T_{K}$ associated with the local condition $\{U_{v}^{+}\}_{v}$ is given by the complex
\[
\widetilde{C}_{f}(G_{k,S},T_{K},\Delta_{\Sigma}) := {\rm Cone}\left(C^{\bullet}(G_{k,S},T_{K}) \oplus U^{+}_{S}(T_{K}) \xrightarrow{{\rm res}_{f}-i_{S}^{+}} \bigoplus_{v \in S \setminus S_{\infty}(k)}C^{\bullet}(G_{k_{v}},T_{K})\right)[1]
\]
where $[n]$ denotes the $n$-shift. 
We write $\widetilde{{\bf R}\Gamma}_{f}(G_{k,S},T_{K},\Delta_{\Sigma})$ for 
the corresponding object in the derived category and $\widetilde{H}^{i}_{f}(G_{k,S},T_{K},\Delta_{\Sigma})$ for its cohomology. 
\end{definition}

By definition, there is an exact triangle 
\begin{align}\label{exact seq}
U^{-}_{S}(T_{K})[-1] \longrightarrow \widetilde{{\bf R}\Gamma}_{f}(G_{k,S},T_{K},\Delta_{\Sigma}) \longrightarrow {\bf R}\Gamma(G_{k,S},T_{K}) \longrightarrow U^{-}_{S}(T_{K}). 
\end{align}
Furthermore, since $C^{\bullet}(G_{k_{v}},T_{K})$ is quasi-isomorphic to zero for $v \in S_{\rm ram}(L/k)$, we have an exact sequence 
\begin{align}
\begin{split}\label{exact2}
0 \longrightarrow &\,\widetilde{H}^{1}_{f}(G_{k,S},T_{K},\Delta_{\Sigma}) \longrightarrow H^{1}(G_{k,S},T_{K}) \longrightarrow \bigoplus_{v \in \Sigma}H^{1}(G_{k_{v}},T_{K})  
\\
&\longrightarrow\widetilde{H}^{2}_{f}(G_{k,S},T_{K},\Delta_{\Sigma}) \longrightarrow 
H^{2}(G_{k,S},T_{K}) \longrightarrow \bigoplus_{v\in \Sigma}H^{2}(G_{k_{v}}, T_{K}) \longrightarrow 0 
\end{split}
\end{align}
and we have 
\[
\widetilde{H}^{i}_{f}(G_{k,S},T_{K},\Delta_{\Sigma}) = 0 \ \text{ if } \ i \neq 1,2. 
\]
Hence by Corollary~\ref{cor:amp}, we have 
\begin{align}\label{complex_amp}
\widetilde{{\bf R}\Gamma}_{f}(G_{k,S}, T_{K},\Delta_{\Sigma}) \in D^{[1,2]}_{\rm parf}(\bZ_{p}^{\rm ur}[[\Gal(K/k)]]).
\end{align}
Put $(-)^{\vee} := \Hom(-,\bQ_{p}/\bZ_{p})$. Then by global duality, we have an  exact sequence 
\begin{align}
\begin{split}\label{exact3} 
H^{1}(G_{k,S},T_{K}) \longrightarrow &\,\bigoplus_{v \in S_{p}(k)}H^{1}(G_{k_{v}},T_{K}) \longrightarrow 
H^{1}(G_{k,S},T_{K}^{\vee}(1))^{\vee,\iota}  
\\
&\longrightarrow H^{2}(G_{k,S},T_{K}) \longrightarrow \bigoplus_{v\in S_{p}(k)}H^{2}(G_{k_{v}}, T_{K}) \longrightarrow 0. 
\end{split}
\end{align}
The exact sequences~(\ref{exact2}) and (\ref{exact3}) together show that we have a canonical homomorphism 
\[
H^{1}(G_{k,S},T_{K}^{\vee}(1))^{\vee,\iota} \longrightarrow \widetilde{H}^{2}_{f}(G_{k,S},T_{K},\Delta_{\Sigma})\,,
\]
and in fact, that there is an exact sequence 
\begin{align}\label{exact4}
0 \longrightarrow Y_{\Sigma,K}^{\chi} \longrightarrow \widetilde{H}^{2}_{f}(G_{k,S},T_{K},\Delta_{\Sigma}) \longrightarrow \bigoplus_{v \in \Sigma^{c}}H^{2}(G_{k_{v}},T_{K}) \longrightarrow 0, 
\end{align}
where 
\[
Y_{\Sigma,K}^{\chi} := 
{\rm coker}\biggl(\bigoplus_{v \in \Sigma^{c}}H^{1}(G_{k_{v}},T_{K}) \longrightarrow H^{1}(G_{k,S},T_{K}^{\vee}(1))^{\vee,\iota} \biggr). 
\]
By global class field theory, we also have an exact sequence 
\[
\bigoplus_{v\in\Sigma^{c}}H^{1}(G_{k_{v}},\bT_{\infty})/H^{1}_{f}(G_{k_{v}},\bT_{\infty}) \longrightarrow 
 X_{\Sigma}^{\chi} \longrightarrow Y_{\Sigma, k(p^{\infty})}^{\chi} \longrightarrow 0. 
\]
However, since $H^{1}(G_{k_{v}}, \bT_{\infty}) = H^{1}_{f}(G_{k_{v}}, \bT_{\infty})$, we conclude that  
$X_{\Sigma}^{\chi} = Y^{\chi}_{\Sigma,k(p^{\infty})}$ and hence
\begin{align}\label{exactseq}
0 \longrightarrow X_{\Sigma}^{\chi} \longrightarrow \widetilde{H}^{2}_{f}(G_{k,S},\bT_{\infty},\Delta_{\Sigma}) \longrightarrow \bigoplus_{v \in \Sigma^{c}}H^{2}(G_{k_{v}},\bT_{\infty}) \longrightarrow 0. 
\end{align}


\begin{lemma}\label{lem:null}
The $\bZ_{p}^{\rm ur}[[\Gamma_{\infty}]]$-module $H^{2}(G_{k_{v}}, \bT_{\infty})$ is pseudo-null for any prime $v$ of $k$ above $p$. 
\end{lemma}
\begin{proof}
Put $L(p^{\infty}) := Lk(p^{\infty})$. 
By the local duality, we have 
\[
H^{2}(G_{k_{v}},\bT_{\infty}) \cong e_{\chi}\bZ_{p}^{\rm ur}[[\Gal(L(p^{\infty})/k)/D_{L(p^{\infty}),v}]]
\]
where $D_{L(p^{\infty}),v}$ denotes the decomposition group of $\Gal(L(p^{\infty})/k)$ at $v$. 
By the local class field theory, we have 
\[
{\rm rank}_{\bZ_{p}}D_{L(p^{\infty}),v} = {\rm rank}_{\bZ_{p}}k_{v}^{\times,\wedge} \geq 2, 
\] 
which completes the proof. 
\end{proof}

\begin{lemma}\label{lem:fitt}
\item[(i)] The $\bZ_{p}^{\rm ur}[[\Gamma_{\infty}]]$-module $\widetilde{H}^{2}_{f}(G_{k,S},\bT_{\infty},\Delta_{\Sigma})$ is torsion. 
\item[(ii)] We have 
\begin{align}\label{eq:fitt=char}
{\rm Fitt}(\widetilde{H}^{2}_{f}(G_{k,S}, \bT_{\infty},\Delta_{\Sigma})) = {\rm char}(\widetilde{H}^{2}_{f}(G_{k,S}, \bT_{\infty},\Delta_{\Sigma})) = {\rm char}(X_{\Sigma}^{\chi}). 
\end{align}
\end{lemma}
\begin{proof}
Since $X_{\Sigma}^{\chi}$ is torsion, claim~(i) follows from \eqref{exactseq} and Lemma~\ref{lem:null}. 
Let us show claim~(ii). 
Recall that we have the exact triangle 
\begin{align*}
U^{-}_{S}(\bT_{\infty})[-1] \longrightarrow \widetilde{{\bf R}\Gamma}_{f}(G_{k,S},\bT_{\infty},\Delta_{\Sigma}) \longrightarrow {\bf R}\Gamma(G_{k,S},\bT_{\infty}) \longrightarrow U^{-}_{S}(\bT_{\infty}). 
\end{align*}
The local and global Euler characteristic formulae therefore imply that the Euler characteristic of the Selmer complex $\widetilde{{\bf R}\Gamma}_{f}(G_{k,S},\bT_{\infty},\Delta_{\Sigma})$ is zero. 
By \eqref{complex_amp}, we have 
\[
\widetilde{{\bf R}\Gamma}_{f}(G_{k,S}, \bT_{\infty},\Delta_{\Sigma}) \in D^{[1,2]}_{\rm parf}(\bZ_{p}^{\rm ur}[[\Gamma_{\infty}]]).
\] 
This shows that there is an $\bZ_{p}^{\rm ur}[[\Gamma_{\infty}]]$-homomorphism $P^{1} \longrightarrow P^{2}$ of free $\bZ_{p}^{\rm ur}[[\Gamma_{\infty}]]$-modules such that 
\[
\widetilde{H}^{1}_{f}(G_{k,S}, \bT_{\infty},\Delta_{\Sigma}) = \ker\left(P^{1} \longrightarrow P^{2}\right) \ \text{ and } \ 
\widetilde{H}^{2}_{f}(G_{k,S}, \bT_{\infty},\Delta_{\Sigma}) = {\rm coker}\left(P^{1} \longrightarrow P^{2}\right). 
\]
Since the Euler characteristic of the Selmer complex $\widetilde{{\bf R}\Gamma}_{f}(G_{k,S},\bT_{\infty},\Delta_{\Sigma})$ is zero, that is to say, 
\[
{\rm rank}_{\bZ_{p}^{\rm ur}[[\Gamma_{\infty}]]}\,P^{1} = {\rm rank}_{\bZ_{p}^{\rm ur}[[\Gamma_{\infty}]]}\,P^{2},
\] 
(i) shows that 
$\widetilde{H}^{1}_{f}(G_{k,S}, \bT_{\infty},\Delta_{\Sigma})=0$. 
We conclude that 
\[
{\rm Fitt}(\widetilde{H}^{2}_{f}(G_{k,S}, \bT_{\infty},\Delta_{\Sigma})) 
= {\rm det}(P^{1} \longrightarrow P^{2}) =  {\rm char}(\widetilde{H}^{2}_{f}(G_{k,S}, \bT_{\infty},\Delta_{\Sigma})). 
\]
By Lemma~\ref{lem:null} and the exact sequence (\ref{exactseq}), we have 
\[
{\rm char}(\widetilde{H}^{2}_{f}(G_{k,S}, \bT_{\infty},\Delta_{\Sigma})) = 
{\rm char}(X_{\Sigma}^{\chi}).
\]
The proof of our lemma follows.
\end{proof}

\begin{lemma}\label{lem:excep}
Let $e := \{v \in \Sigma^{c} \mid \chi(G_{k_{v}}) = 1\}$. 
Then we have 
\[
{\rm Fitt}(\widetilde{H}^{2}_{f}(G_{k,S},\bT_{\infty},\Delta_{\Sigma})) \in \cA_{\Gamma_{\infty}}^{e}. 
\]
\end{lemma}
\begin{proof}
Take a $\bZ_{p}$-extension $k_{\Gamma}/k$ with Galois group $\Gamma$. 
It suffices to show that 
\[
{\rm Fitt}(\widetilde{H}^{2}_{f}(G_{k,S},\bT_{\infty},\Delta_{\Sigma}) \otimes_{\bZ_{p}^{\rm ur}[[\Gamma_{\infty}]]} \bZ_{p}^{\rm ur}[[\Gamma]]) \in \cA_{\Gamma}^{e}. 
\]
By Lemma~\ref{lem:descent}, \eqref{exact seq}, and \eqref{complex_amp}, we have an isomorphism 
\[
\widetilde{H}^{2}_{f}(G_{k,S},\bT_{\infty},\Delta_{\Sigma}) \otimes_{\bZ_{p}^{\rm ur}[[\Gamma_{\infty}]]} \bZ_{p}^{\rm ur}[[\Gamma]] \cong \widetilde{H}^{2}_{f}(G_{k,S},\bT_{\Gamma},\Delta_{\Sigma}). 
\]
We note that ${\rm Fitt}(\widetilde{H}^{2}_{f}(G_{k,S},\bT_{\Gamma},\Delta_{\Sigma}))$ is principal ideal by (\ref{eq:fitt=char}) and hence  
\[
{\rm Fitt}(\widetilde{H}^{2}_{f}(G_{k,S},\bT_{\Gamma},\Delta_{\Sigma})) = {\rm char}(\widetilde{H}^{2}_{f}(G_{k,S},\bT_{\Gamma},\Delta_{\Sigma})). 
\]
Therefore, by the exact sequence (\ref{exact4}), we conclude that 
\[
{\rm Fitt}(\widetilde{H}^{2}_{f}(G_{k,S},\bT_{\Gamma},\Delta_{\Sigma})) \subseteq \prod_{v \in \Sigma^{c}}{\rm char}(H^{2}(G_{k_{v}},\bT_{\Gamma})). 
\]
By the local duality, we have $H^{2}(G_{k_{v}},\bT_{\Gamma}) = 0$ if $\chi(G_{k_{v}}) \neq 1$ and  
\[
H^{2}(G_{k_{v}},\bT_{\Gamma}) \cong \bZ_{p}^{\rm ur}[[\Gamma/\Gamma_{v}]] 
\]
if $\chi(G_{k_{v}}) = 1$. Here $\Gamma_{v} \subseteq \Gamma$ denotes the decomposition group at $v$. 
This fact implies 
\[
{\rm char}(H^{2}(G_{k_{v}},\bT_{\Gamma})) \subseteq  \cA_{\Gamma} 
\]
if $\chi(G_{k_{v}}) = 1$, which completes the proof. 
\end{proof}

The following proposition follows from (\ref{eq:fitt=char}) and Lemma~\ref{lem:excep}. 
\begin{proposition}\label{prop_app_IMC_reformulated}
Iwasawa Main Conjecture~\ref{conj:iwasawa} for the CM field $k$ and the branch character $\chi$ is equivalent to that the Fitting ideal 
of $\widetilde{H}^{2}_{f}(G_{k,S},\bT_{\infty},\Delta_{\Sigma})^{\iota}$ is generated by 
$L_{p,\Sigma}^{\chi}$:
\[
{\rm Fitt}(\widetilde{H}^{2}_{f}(G_{k,S},\bT_{\infty},\Delta_{\Sigma})^{\iota}) = (L_{p,\Sigma}^{\chi}). 
\]
In particular, Iwasawa Main Conjecture for the CM field $k$  implies exceptional zero conjecture. 
\end{proposition}

Let $k_{\Gamma}/k$ be a $\bZ_{p}$-extension with Galois group $\Gamma$. 
Then note that $T_{k_{\Gamma}} = \bT_{\Gamma}$. 

\begin{corollary}
\label{Cor_App_IMC_reformulated}
Assume the Iwasawa Main Conjecture for the CM field $k$. 
Then we have 
\[
{\rm Fitt}(\widetilde{H}^{2}_{f}(G_{k,S},\bT_{\Gamma},\Delta_{\Sigma})^{\iota}) = (L_{p,\Sigma}^{\chi}\vert_{\Gamma}). 
\]
Here, $L_{p,\Sigma}^{\chi}\vert_{\Gamma}$ denotes the image of $L_{p,\Sigma}^{\chi}$ in $\bZ_{p}^{\rm ur}[[\Gamma]]$. 
\end{corollary}

\begin{corollary}\label{lem:vanish-line}
If the $\bZ_{p}^{\rm ur}[[\Gamma]]$-module $\widetilde{H}^{2}_{f}(G_{k,S}, \bT_{\Gamma}, \Delta_{\Sigma})$ is torsion, then no primes in $\Sigma^{c}$ splits completely in $k_{\Gamma}/k$. 
\end{corollary}
\begin{proof}
Note that $H^{2}(G_{k_{v}},\bT_{\Gamma})$ is a torsion $\bZ_{p}^{\rm ur}[[\Gamma]]$-module if and only if $v$ does not split completely in $k_{\Gamma}/k$. 
Hence this corollary follows from the exact sequence (\ref{exact2}). 
%
\end{proof}


$\,$
\bibliographystyle{amsalpha}
\bibliography{references}

\end{document}